\documentclass{amsart}
\usepackage{amsmath}
\usepackage{amssymb}
\usepackage{latexsym}
\usepackage{graphicx}
\usepackage[final]{hyperref}
\hypersetup{colorlinks=true, linkcolor=blue, 
anchorcolor=blue, citecolor=red, filecolor=blue, 
menucolor=blue, pagecolor=blue, urlcolor=blue}

\newcommand{\boldN}{\boldsymbol{N}}
\newcommand{\boldL}{\boldsymbol{L}}

\newcommand{\boldgamma}{\boldsymbol{\gamma}}

\newcommand{\Nmin}{N_{\mathrm{min}}}
\newcommand{\Nmed}{N_{\mathrm{med}}}
\newcommand{\Nmax}{N_{\mathrm{max}}}
\newcommand{\Lmin}{L_{\mathrm{min}}}
\newcommand{\Lmed}{L_{\mathrm{med}}}
\newcommand{\Lmax}{L_{\mathrm{max}}}
\newcommand{\fourierabs}[1]{\lfloor #1 \rfloor}
\newcommand{\hypwt}{\mathfrak h}
\newcommand{\vangle}{\theta}

\newcommand{\Abs}[1]{\left\vert #1 \right\vert}
\newcommand{\abs}[1]{| #1 |}

\newcommand{\bigabs}[1]{\bigl\vert #1 \bigr\vert}

\newcommand{\norm}[1]{\left\Vert #1 \right\Vert}
\newcommand{\fixednorm}[1]{\Vert #1 \Vert}
\newcommand{\bignorm}[1]{\bigl\Vert #1 \bigr\Vert}
\newcommand{\Bignorm}[1]{\Bigl\Vert #1 \Bigr\Vert}

\newcommand{\C}{\mathbb{C}}
\newcommand{\N}{\mathbb{N}}
\newcommand{\Proj}{\mathbf{P}}

\newcommand{\R}{\mathbb{R}}

\newcommand{\Innerprod}[2]{\left\langle \, #1 , #2 \, \right\rangle}
\newcommand{\innerprod}[2]{\langle \, #1 , #2 \, \rangle}

\newcommand{\angles}[1]{\langle #1 \rangle}

\DeclareMathOperator{\diag}{diag}

\DeclareMathOperator{\im}{Im}

\DeclareMathOperator{\supp}{supp}

\newtheorem{theorem}{Theorem}[section]

\newtheorem{lemma}{Lemma}[section]

\theoremstyle{definition}
\newtheorem{definition}{Definition}

\theoremstyle{remark}
\newtheorem{remark}{Remark}

\numberwithin{equation}{section}
\title[Maxwell-Dirac]{Null structure and almost optimal local well-posedness of the Maxwell-Dirac system}

\author[P. D'Ancona]{Piero D'Ancona}
\address{Department of Mathematics\\
University of Rome ``La Sapienza''\\
Piazzale Aldo Moro 2\\
I-00185 Rome\\ Italy}
\email{dancona@mat.uniroma1.it}

\author[D. Foschi]{Damiano Foschi}
\address{Department of Mathematics\\
University of Ferrara\\
Via Macchiavelli 35\\
I-44100 Ferrara\\ Italy}
\email{damiano.foschi@unife.it}

\author[S. Selberg]{Sigmund Selberg}
\address{Department of Mathematical Sciences\\
Norwegian University of Science and Technology\\
Alfred Getz' vei 1\\
N-7491 Trondheim\\ Norway}
\thanks{The third author was partially supported by the Research Council of Norway, grant no.\ 160192/V30, PDE and Harmonic Analysis.}
\thanks{This work was completed while the third author visited Princeton University, and he wishes to thank Sergiu Klainerman for his hospitality and for helpful conversations.}

\email{sigmund.selberg@math.ntnu.no}


\begin{document}

\begin{abstract} 
We uncover the full null structure of the Maxwell-Dirac system in Lorenz gauge. This structure, which cannot be seen in the individual component equations, but only when considering the system as a whole, is expressed in terms of tri- and quadrilinear integral forms with cancellations measured by the angles between spatial frequencies. In the 3D case, we prove frequency-localized $L^2$ space-time estimates for these integral forms at the scale invariant regularity up to a logarithmic loss, hence we obtain almost optimal local well-posedness of the system by iteration.
\end{abstract}

\maketitle
\tableofcontents

\section{Introduction}\label{A}

In this paper we uncover the complete null structure of the Maxwell-Dirac system (M-D) in Lorenz gauge. This structure is expressed in terms of tri- and quadrilinear integral forms with certain cancellations measured by the angles between spatial frequencies. In the 3D case, we prove frequency-localized $L^2$ space-time estimates for these integral forms, at the optimal (i.e., scale invariant) regularity up to a logarithmic loss, and as a consequence we obtain almost optimal local well-posedness of the system by iteration.

The null structure that we have found is not the usual \emph{bilinear null structure} that may be seen in bilinear terms of each individual component equation of a system, but instead depends on the structure of the system as a whole, hence we call it \emph{system null structure}. System null structure has been found first by Machedon and Sterbenz \cite{Machedon:2004} for the Maxwell-Klein-Gordon system (M-K-G), and later by the present authors for the Dirac-Klein-Gordon system (D-K-G); see \cite{Selberg:2007d}; in both cases this structure was used to prove the almost optimal local well-posedness of the respective systems.

A key feature, distinguishing our work from earlier results on the regularity of nonlinear gauge field theories, is that we find a null structure relative to the Lorenz gauge condition, whereas up to now the gauge of choice has been the Coulomb gauge, following the seminal works \cite{Klainerman:1994b, Klainerman:1995a}, where a bilinear null structure was found for M-K-G and also the Yang-Mills equations (Y-M) in Coulomb gauge. The Coulomb gauge has been widely used since, also in \cite{Machedon:2004}. For Y-M, the temporal gauge has also been used; see \cite{Tao:2003}. 

An obvious point in favor of the Lorenz gauge is the fact that it is Lorentz (\emph{sic}) invariant, which entails that the equations take a much more symmetric form (they become nonlinear wave equations) than in Coulomb gauge, where one obtains a mix of hyperbolic and elliptic equations.

In fact, the only advantage of Coulomb gauge seems to be that the 4-potential $A$ of the electromagnetic field is better behaved in this gauge, whereas in Lorenz gauge it appears to have rather poor regularity properties. So if one walks down the usual path, thinking of M-K-G or M-D as systems of PDEs for either a scalar field $\phi$ or a spinor $\psi$, and the potential $A$, then this system will likely not be well-posed in Lorenz gauge near the scaling regularity. However, there is no compelling reason to take this point of view, because the regularity of $A$ in itself is really of no interest. Instead, it is the electric and magnetic fields $\mathbf E$ and $\mathbf B$ (or equivalently the tensor $F$) which matter, and these are perfectly well-behaved in Lorenz gauge.

The systems M-K-G and M-D are related, as can be seen from the fact that by ``squaring'' the Dirac part of M-D (which will destroy some of its structure) one obtains an equation that looks like the Klein-Gordon part of M-K-G, but with two bilinear terms added; in \cite{Selberg:2005} it was shown that these additional terms also have a null structure, in Coulomb gauge. Combining this fact with the null structure found in \cite{Klainerman:1994b} for M-K-G in Coulomb gauge, one can then conclude that all the bilinear terms in the ``squared M-D'' are null forms. In view of this, it is conceivable that the analysis in \cite{Machedon:2004}, where almost optimal local well-posedness was proved for M-K-G in Coulomb gauge, could be extended to cover also the ``squared'' M-D in Coulomb gauge, but we do not try to follow this path, which would only add to the already highly complicated analysis in \cite{Machedon:2004}.

In fact, ``squaring'' M-D is not a good idea, as it clearly destroys most of the spinorial structure of the system. And it is precisely the spinorial structure which allows us to find some very powerful cancellations, which however have a remarkably simple form, being expressed in terms of the six angles between the four spatial frequencies in a certain quadrilinear space-time integral form.

This structure enables us to prove closed estimates for the iterates of M-D at the scale-invariant data regularity, up to a logarithmic loss, but this turns out to be quite difficult, and takes up most of the paper. Since, as remarked, the systems M-D and M-K-G are related, it is natural to make a comparison of the techniques used here with those applied in \cite{Machedon:2004}. In fact, some superficial similarities aside, our approaches differ significantly.

As does \cite{Machedon:2004}, we rely on the usual dyadic decompositions of frequency space adapted to the null cone, and angular decompositions for the spatial frequencies taking into account the geometry of interacting null cones.

The key difference is that we do everything within the confines of $L^2$ theory. On the one hand, our spaces are very simple: We use only the standard $L^2$-based Sobolev and wave-Sobolev norms. On the other hand, working entirely in $L^2$ comes at a price: The existing $L^2$ theory for bilinear interactions of waves is not sufficient to handle the quadrilinear space-time integral form which is at the core of the M-D regularity problem, and we develop new techniques to deal with the difficulties that arise. Of course, we do use the $L^2$ bilinear generalizations of the $L^4$ estimate of Strichartz for the homogeneous wave equation (see \cite{Foschi:2000}), but in addition we require a number of modifications of these estimates, proved by the third author in \cite{Selberg:2008a}.

In \cite{Machedon:2004}, by comparison, highly sophisticated spaces were used, built from Besov versions of the $X^{s,b}$ spaces and Tataru's outer block norms, but only well-known bilinear estimates were applied ($L^2$ and mixed-norm generalizations of the Strichartz estimates for the wave equation). 

We now present the M-D system.

On the Minkowski space-time $\R^{1+3}$ we use coordinates $t=x^0$ and $x=(x^1,x^2,x^3)$. The corresponding Fourier variable is denoted $X=(\tau,\xi)$, where $\tau \in \R$ and $\xi \in \R^3$ correspond to $t$ and $x$, respectively. The partial derivative with respect to $x^\mu$ is denoted $\partial_\mu$, and we also write $\partial_t=\partial_0$, $\nabla = (\partial_1,\partial_2,\partial_3)$. Roman indices $j,k,\dots$ run over $1,2,3$, greek indices $\mu,\nu,\dots$ over $0,1,2,3$, and repeated indices are implicitly summed over these ranges. Indices are raised and lowered using the metric $\diag(-1,1,1,1)$. The conventions regarding indices do not apply to mere enumerations, of course, only to indices relating to the coordinates of space-time.

The M-D system describes an electron self-interacting with an electromagnetic field, and is obtained by coupling Maxwell's equations and the Dirac equation:
\begin{gather}\label{A:2}
  \nabla \cdot \mathbf E = \rho,
  \qquad
  \nabla \cdot \mathbf B = 0,
  \qquad
  \nabla \times \mathbf E + \partial_t \mathbf B = 0,
  \qquad
  \nabla \times \mathbf B - \partial_t \mathbf E = \mathbf J,
  \\
  \label{A:4}
  \left(\boldsymbol\alpha^\mu D_\mu + m\boldsymbol\beta\right)\psi = 0.
\end{gather}
The unknowns are the fields $\mathbf E = (E^1,E^2,E^3)$ and $\mathbf B = (B^1,B^2,B^3)$, which are $\R^3$-valued functions of $(t,x)$, and the Dirac spinor $\psi$, which is a $\C^4$-valued function of $(t,x)$; $m \ge 0$ is the rest mass of the electron. We regard elements of $\C^4$ as column vectors, hence it makes sense to premultiply them by the $4\times4$ Dirac matrices
\begin{equation}\label{A:5}
  \boldsymbol\alpha^0 = \mathbf I_{4\times4},
  \qquad
  \boldsymbol\alpha^j =
     \begin{pmatrix}
        0 & \boldsymbol\sigma^j \\
        \boldsymbol\sigma^j & 0 \\
     \end{pmatrix},
   \qquad
   \boldsymbol\beta =
     \begin{pmatrix}
        \mathbf I_{2 \times 2} & 0 \\
        0 & - \mathbf I_{2 \times 2} \\
     \end{pmatrix},
\end{equation}
where the $\boldsymbol\sigma^j$ are the Pauli matrices. The matrices in \eqref{A:5} are all hermitian, and satisfy $(\boldsymbol\alpha^\mu)^2 = (\boldsymbol\beta)^2 = \mathbf I_{4 \times 4}$ and $\boldsymbol\alpha^j \boldsymbol\alpha^k + \boldsymbol\alpha^k \boldsymbol\alpha^j = 0$ for $1 \le j < k \le 3$.

Formally, the second and third equations in \eqref{A:2} are equivalent to the existence of a  four-potential $A_\mu = A_\mu(t,x) \in \R$, $\mu=0,1,2,3$, such that
\begin{equation}\label{A:6}
  \mathbf B = \nabla \times \mathbf A,
  \qquad
  \mathbf E = \nabla A_0 - \partial_t \mathbf A,
\end{equation}
where we write $\mathbf A = (A_1,A_2,A_3)$. In the absence of an electromagnetic field, the operator $D_\mu$ in \eqref{A:4} would just be $-i\partial_\mu$, but in the presence of a field $(\mathbf E, \mathbf B)$ represented by $A_\mu$, this must be modified by the minimal coupling transformation, so that $D_\mu$ becomes the gauge covariant derivative
\begin{equation}\label{A:8}
  D_\mu = D_\mu^{(A)} = \frac{1}{i}\partial_\mu - A_\mu.
\end{equation}
To complete the coupling we plug into \eqref{A:2} the Dirac four-current density
\begin{equation}\label{A:10}
  J^\mu = \innerprod{\boldsymbol\alpha^\mu\psi}{\psi} \qquad (\mu=0,1,2,3),
\end{equation}
where $\innerprod{z}{w}$ is the standard inner product on $\C^4$; this splits into the charge density $\rho = J^0 = \abs{\psi}^2$ and the three-current density $\mathbf J = (J^1,J^2,J^3)$.

Eqs.\ \eqref{A:2}--\eqref{A:8} constitute the M-D system; to simplify it, we can express also the first and fourth equations in \eqref{A:2} in terms of $A_\mu$. Thus, \eqref{A:2} is replaced by
\begin{equation}\label{A:12}
  \square A_\mu - \partial_\mu ( \partial^\nu A_\nu ) = - J_\mu
  \qquad \left( \square = \partial_\mu \partial^\mu = -\partial_t^2 + \Delta \right).
\end{equation}
Here, by the conventions on indices, $J_0 = - J^0$ and $J_j = J^j$ for $j=1,2,3$.

Then M-D consists of \eqref{A:4} and \eqref{A:12}, coupled by \eqref{A:8} and \eqref{A:10}.

This system is invariant under the \emph{gauge transformation}
\begin{equation}\label{A:14}
  \psi \longrightarrow \psi' = e^{i\chi} \psi,
  \qquad A_\mu \longrightarrow A_\mu' = A_\mu + \partial_\mu \chi,
\end{equation}
for any $\chi : \R^{1+3} \to \R$, called the \emph{gauge function}. Indeed, if $(\psi,A_\mu)$ satisfies M-D, then so does $(\psi',A_\mu')$, in view of the identity
$$
  D_\mu^{(A')}\psi' = e^{i\chi}D_\mu^{(A)}\psi.
$$
Since the observables $\mathbf E, \mathbf B, \rho, \mathbf J$ are not affected by \eqref{A:14}, two solutions related by a gauge transformation are physically undistinguishable, and must be considered equivalent. In practice, a solution is therefore a representative of its equivalence class, and we can pick a representative whose potential $A_\mu$ is chosen so that it simplifies the analysis as much as possible. This is known as the \emph{gauge freedom}. In this paper we impose the Lorenz gauge condition (due to Ludvig Lorenz, not Hendrik Lorentz),
\begin{equation}\label{A:16}
  \partial^\mu A_\mu = 0 \qquad \left( \iff \partial_t A_0 = \nabla \cdot \mathbf A \right),
\end{equation}
which greatly simplifies \eqref{A:12}. Then the M-D system becomes
\begin{align}
  \label{A:20}
  \left(-i\boldsymbol\alpha^\mu \partial_\mu + m\boldsymbol\beta\right)\psi &= A_\mu \boldsymbol\alpha^\mu \psi,
  \\
  \label{A:22}
  \square A_\mu &= - \innerprod{\boldsymbol\alpha_\mu\psi}{\psi},
  \\
  \label{A:24}
  \partial^\mu A_\mu &= 0. 
\end{align}
We consider the initial value problem starting from data
\begin{equation}\label{A:30}
  \psi(0,x) = \psi_0(x) \in \C^4,
  \quad \mathbf E(0,x) = \mathbf E_0(x) \in \R^3,
  \\
  \quad \mathbf B(0,x) = \mathbf B_0(x) \in \R^3,
\end{equation}
which in view of the first two equations in \eqref{A:2} must satisfy the constraints
\begin{equation}\label{A:32}
  \nabla \cdot \mathbf E_0 = \abs{\psi_0}^2,
  \qquad
  \nabla \cdot \mathbf B_0 = 0. 
\end{equation}
The initial data for the four-potential $A_\mu$, which we denote by
\begin{equation}\label{A:40}
  A_\mu(0,x) = a_\mu(x) \in \R,
  \qquad
  \partial_t A_\mu(0,x) = \dot a_\mu(x) \in \R
  \qquad (\mu=0,1,2,3),
\end{equation}
must be constructed from the observable data $(\mathbf E_0,\mathbf B_0)$.
We write $\mathbf a = (a_1,a_2,a_3)$ and $\dot{\mathbf a} = (\dot a_1,\dot a_2,\dot a_3)$. By \eqref{A:24} and \eqref{A:6} we get the constraints
\begin{equation}\label{A:42}
  \dot a_0 = \nabla \cdot \mathbf a.
  \qquad
  \mathbf B_0 = \nabla \times \mathbf a,
  \qquad
  \mathbf E_0 = \nabla a_0 - \dot{\mathbf a},
\end{equation}
which determine $\mathbf a, \dot{\mathbf a}$, given $a_0,\dot a_0$. The simplest choice is
\begin{equation}\label{A:48}
  a_0 = \dot a_0 = 0.
\end{equation}
Then $\mathbf a,\dot{\mathbf a}$ are determined by \eqref{A:42}.

The Lorenz gauge condition \eqref{A:24} is automatically satisfied throughout the time interval of existence, for data satisfying \eqref{A:32} and \eqref{A:42}. It suffices to prove this for smooth solutions, since the solutions that we later obtain are limits of smooth solutions. So assume $(\psi,A_\mu)$ is a smooth solution of \eqref{A:20} and \eqref{A:22} on a time interval $(-T,T)$, with the $A_\mu$'s real-valued, and set $u=\partial^\mu A_\mu$. A calculation using \eqref{A:22} yields $\square u = - 2 \im \innerprod{ -i\boldsymbol\alpha^\mu \partial_\mu \psi}{\psi}$, and \eqref{A:20} implies $\square u = 0$; here we use the hermiticity of $\boldsymbol\alpha^\mu$ and $\boldsymbol\beta$. Moreover, $u(0)=\partial_t u(0) = 0$, on account of \eqref{A:42} and \eqref{A:32}, so we conclude that $u(t)=0$ for all $t \in (-T,T)$.

Thus, the Lorenz condition \eqref{A:24} can be removed from the system once we have data satisfying the proper constraints, and we are left with the equations \eqref{A:20} and \eqref{A:22}, but these can be combined into a single nonlinear Dirac equation by splitting the four-potential into its homogeneous and inhomogeneous parts:
\begin{gather}\label{A:60}
  A_\mu = A_\mu^{\mathrm{hom.}} + A_\mu^{\mathrm{inh.}},
  \\
  \label{A:62}
  \square A_\mu^{\mathrm{hom.}} = 0,
  \qquad
  A_\mu^{\mathrm{hom.}}(0,x) = a_\mu(x),
  \qquad
  \partial_t A_\mu^{\mathrm{hom.}}(0,x) = \dot a_\mu(x),
  \\
  \label{A:64}
  A_\mu^{\mathrm{inh.}} = - \square^{-1} \Innerprod{\boldsymbol\alpha_\mu\psi}{\psi}.
\end{gather}
Here we use the notation $\square^{-1}F$ for the solution of the inhomogeneous wave equation $\square u = F$ with vanishing data at time $t=0$.

Thus, M-D in Lorenz gauge has been reduced to the nonlinear Dirac equation
\begin{equation}\label{A:70}
  \left(-i\boldsymbol\alpha^\mu \partial_\mu + m\boldsymbol\beta\right)\psi
  = A_\mu^{\mathrm{hom.}} \boldsymbol\alpha^\mu \psi
  - \mathcal N(\psi,\psi,\psi),
\end{equation}
where
\begin{equation}\label{A:72}
  \mathcal N(\psi_1,\psi_2,\psi_3)
  =
  \left( \square^{-1} \Innerprod{\boldsymbol\alpha_\mu\psi_1}{\psi_2} \right) 
  \boldsymbol\alpha^\mu\psi_3.
\end{equation}

In order to uncover the null structure in \eqref{A:70}, we decompose the spinor as
\begin{equation}\label{B:12}
  \psi=\psi_+ + \psi_-,
  \qquad
  \qquad \psi_\pm \equiv \mathbf\Pi_\pm \psi,
\end{equation}
where $\mathbf\Pi_\pm \equiv \mathbf\Pi(\pm \nabla/i)$ is the multiplier whose symbol is the Dirac projection
\begin{equation}\label{B:8}
  \mathbf\Pi(\xi)
  =
  \frac12 \left( \mathbf I_{4 \times 4} + \frac{\xi^j \boldsymbol\alpha_j}{\abs{\xi}} \right) \qquad (\xi \in \R^3).
\end{equation}
To motivate this, we note that the stationary Dirac operator $-i\boldsymbol\alpha^j\partial_j$ has symbol $\xi^j \boldsymbol\alpha_j$ (recall that $\xi \in \R^3$ denotes the Fourier variable corresponding to $x$), whose eigenvalues are $\pm\abs{\xi}$, with associated eigenspace projections $\mathbf\Pi(\pm\xi)$, as can be seen using the algebraic properties of the Dirac matrices. Note the identities
\begin{gather}
  \label{B:14}
  \mathbf I_{4\times4} = \mathbf\Pi(\xi)+\mathbf\Pi(-\xi),
  \qquad
  R^j(\xi)\boldsymbol\alpha_j = \mathbf\Pi(\xi)-\mathbf\Pi(-\xi),
  \\
  \label{B:16}
  \mathbf\Pi(\xi)^* = \mathbf\Pi(\xi),
  \qquad
  \mathbf\Pi(\xi)^2 = \mathbf\Pi(\xi),
  \qquad
  \mathbf\Pi(\xi)\mathbf\Pi(-\xi) = 0.
\end{gather}
In view of the last two identities, it is no surprise that
\begin{equation}\label{B:17} 
  \abs{\mathbf\Pi(\xi_1)\mathbf\Pi(-\xi_2)z} \lesssim \abs{z} \vangle(\xi_2,\xi_2)
  \qquad \left( \forall \xi_1,\xi_2 \in \R^3 \setminus \{0\}, \, z \in \C^4 \right),
\end{equation}
where $\theta(\xi_1,\xi_2)$ denotes the angle between nonzero vectors $\xi_1,\xi_2$. This estimate, proved in \cite[Lemma 2]{Selberg:2007d}, is a key tool for identifying spinorial null structures.

The right member of \eqref{B:14} can also be restated as
\begin{equation}\label{B:10}
  -i\boldsymbol\alpha^j\partial_j = \abs{\nabla}\,\mathbf\Pi_+ - \abs{\nabla}\,\mathbf\Pi_-,
\end{equation}
where $\abs{\nabla}$ is the multiplier with symbol $\abs{\xi}$. Combining this with \eqref{B:16} and the fact that $\mathbf\Pi(\xi)\boldsymbol\beta = \boldsymbol\beta \mathbf\Pi(-\xi)$, we see that \eqref{A:70} splits into two equations:
\begin{subequations}\label{B:18}
\begin{align}
  \label{B:20}
  \left(-i\partial_t+\abs{\nabla}\right) \psi_+
  =
  - m \boldsymbol\beta \psi_-
  + \mathbf\Pi_+ \left( A_\mu^{\mathrm{hom.}} \boldsymbol\alpha^\mu \psi - \mathcal N(\psi,\psi,\psi) \right),
  \\
  \label{B:22}
  \left(-i\partial_t-\abs{\nabla}\right) \psi_-
  =
  - m \boldsymbol\beta \psi_+
  + \mathbf\Pi_-\left( A_\mu^{\mathrm{hom.}} \boldsymbol\alpha^\mu \psi - \mathcal N(\psi,\psi,\psi) \right).
\end{align}
\end{subequations}
Corresponding to the operators on the left, we define the following  spaces.

\begin{definition}\label{A:Def}
For $s,b \in \R$, $X_\pm^{s,b}$ is the completion of the Schwartz space $\mathcal{S}(\R^{1+3})$ with respect to the norm
\begin{equation}\label{B:24}
  \norm{u}_{X_\pm^{s,b}} = \bignorm{ \angles{\xi}^s
  \angles{\tau\pm\abs{\xi}}^b \,\widetilde
  u(\tau,\xi)}_{L^2_{\tau,\xi}},
\end{equation}
where $\widetilde u(\tau,\xi)$ denotes the Fourier transform of $u(t,x)$, and
$\angles{\xi} = (1+\abs{\xi}^2)^{1/2}$.
\end{definition}

Given $T > 0$, we denote by $X_\pm^{s,b}(S_T)$ the restriction of $X_\pm^{s,b}$ to the time-slab
$$
  S_T = (-T,T) \times \R^3.
$$
We recall the well-known fact that $X_\pm^{s,b}(S_T) \hookrightarrow C([-T,T]; H^s)$, for $b > 1/2$.

Our first main result is that \eqref{A:70} is locally well-posed almost down to the critical regularity determined by scaling. To see what this regularity is, observe that in the massless case $m=0$, M-D is invariant under the rescaling
$$
  \psi(t,x) \longrightarrow \frac{1}{L^{3/2}} \psi\left(\frac{t}{L},\frac{x}{L}\right),
  \qquad
  (\mathbf E,\mathbf B)(t,x) \longrightarrow \frac{1}{L^2} (\mathbf E,\mathbf B)\left(\frac{t}{L},\frac{x}{L}\right),
$$
hence the scale invariant data space is
$$
  \psi_0 \in L^2
  \qquad
  (\mathbf E_0,\mathbf B_0) \in \dot H^{-1/2} \times \dot H^{-1/2},
$$
and one does not expect well-posedness with less regularity than this. Our first main result is that local well-posedness holds with only slightly more regularity:

\begin{theorem}\label{A:Thm1}
Let $s > 0$. Assume given initial data \eqref{A:30} with the regularity
\begin{equation}\label{A:100}
  \psi_0 \in H^s(\R^3;\C^4),
  \qquad
  \mathbf E_0, \mathbf B_0 \in H^{s-1/2}(\R^3;\R^3),
\end{equation}
and satisfying the constraints \eqref{A:32}. Then:

\medskip
\emph{(a)} (Lorenz data.) There exist $\{a_\mu,\dot a_\mu\}_{\mu=0,1,2,3}$ with $a_0=\dot a_0 = 0$ and
\begin{equation}\label{A:102}
  \mathbf a \in \abs{D}^{-1} H^{s-1/2}(\R^3;\R^3),
  \qquad
  \dot{\mathbf a} \in H^{s-1/2}(\R^3;\R^3),
\end{equation}
and such that the constraint \eqref{A:42} is satisfied.

\medskip
\emph{(b)} (Local existence.) Use the data $\{a_\mu,\dot a_\mu\}$ from part \emph{(a)} to define $A_\mu^{\mathrm{hom.}}$ as in \eqref{A:62}. Then there exists a time $T > 0$, depending continuously on the norms of the data \eqref{A:100}, and there exists a
$$
  \psi \in C([-T,T];H^s(\R^3;\C^4))
$$
which solves \eqref{A:70} on $S_T = (-T,T) \times \R^3$ with initial data $\psi(0) = \psi_0$.

\medskip
\emph{(c)} (Uniqueness.) The solution has the regularity, with notation as in \eqref{B:12},
\begin{equation}\label{A:104}
  \psi_+ \in X^{s,b}_+(S_T),
  \qquad
  \psi_- \in X^{s,b}_-(S_T),
\end{equation}
where $b=1/2+\varepsilon$ for $\varepsilon > 0$ sufficiently small, depending on $s$. Moreover, the solution is unique in this regularity class.
\end{theorem}

The proof is by iteration in the space \eqref{A:104}, so the main challenge is to prove closed estimates in this space; then existence and uniqueness follow by standard arguments (which we do not repeat here), as do persistence of higher regularity and continuous dependence of the solution on the data, which we did not explicitly include in statement of Theorem \ref{A:Thm1}. (The latter two properties guarantee that our solutions are limits of smooth solutions, a fact which was used above to reduce the Lorenz gauge condition to constraints on the initial data.)

Since $\psi$ and $A_\mu^{\mathrm{hom.}}$ have very little regularity, it is far from obvious that the nonlinear terms in \eqref{A:70} make sense as distributions (so that it is meaningful to talk about a solution of \eqref{A:70}). The fact that they do make sense follows from the very estimates that we use to close the iteration.

Let us mention some earlier results for M-D. Local existence of smooth solutions was proved by Gross \cite{Gross:1966}. Georgiev \cite{Georgiev:1991} proved global existence for small, smooth data. Bournaveas \cite{Bournaveas:1996} proved local well-posedness for data \eqref{A:100} with $s > 1/2$; this was improved to $s=1/2$ by Masmoudi and Nakanishi \cite{Masmoudi:2003}.

We remark also that our work leaves open the important question whether M-D is well-posed (globally, for small-norm data) for some scale invariant data space.

Having obtained the solution $\psi$ of \eqref{A:70}, we can immediately construct the full four-potential by defining $A_\mu^{\mathrm{inh.}}$ as in \eqref{A:64}. This has poor regularity properties, however, due to the lack of null structure in the right hand side of \eqref{A:64}, and the fact that $\psi$ only has slightly more than $L^2$ regularity. We do prove that
\begin{equation}\label{A:110}
  A_\mu^{\mathrm{inh.}} \in C([-T,T];H^{s-1/2}),
\end{equation}
but this is a full degree lower than the regularity of the data for $A_\mu$ (cf.\ \eqref{A:102}).

But the regularity of $A_\mu$ is not of interest; what matters is the electromagnetic field $(\mathbf E, \mathbf B)$, and this turns out to have much better regularity properties, due to the structure of Maxwell's equations. In fact, the data regularity (see \eqref{A:100}) persists throughout the time interval of existence, as our second main result shows:

\begin{theorem}\label{A:Thm2} Assume that the hypotheses of Theorem \ref{A:Thm1} are satisfied, and let $\psi$ be the solution of \eqref{A:70} on $S_T = (-T,T) \times \R^3$ obtained in that theorem. Then there exists a unique solution
$$
  (\mathbf E, \mathbf B) \in C([-T,T];H^{s-1/2})
$$
of Maxwell's equations \eqref{A:2}, with the Dirac four-current \eqref{A:10} induced by $\psi$, and with data as in \eqref{A:100}. Moreover, $A_\mu^{\mathrm{inh.}}$, defined by \eqref{A:64}, has the regularity \eqref{A:110}.
\end{theorem}

\section{Notation}\label{B}

\subsection{Absolute constants} In estimates we use the shorthand $X \lesssim Y$ for $X \le CY$, where $C \gg 1$ is some absolute constant; $X=O(R)$ is short for $\abs{X} \lesssim R$; $X \sim Y$ means $X \lesssim Y \lesssim X$; $X \ll Y$ stands for $X \le C^{-1} Y$, with $C$ as above. We write $\simeq$ for equality up to multiplication by an absolute constant (typically factors involving $2\pi$).  Constants which are not absolute are always denoted explicitly, often with the parameters they depend on as subscripts or arguments.

\subsection{Fourier transforms and multipliers} We write
\begin{align*}
  \mathcal F_x f(\xi) = \widehat f(\xi) &= \int_{\R^3} e^{-ix\cdot\xi} f(x) \, dx,
  \\
  \mathcal F u(\tau,\xi) = \widetilde u(\tau,\xi) &= \int_{\R^{1+3}} e^{-i(t\tau+x\cdot\xi)} u(t,x) \, dt \, dx,
\end{align*}
where $\tau \in \R$ and $\xi \in \R^3$; we call $\xi$ the spatial frequency. We write $X=(\tau,\xi)$, and multiple frequencies are numbered by subscript, as in $X_j = (\tau_j,\xi_j)$. Coordinates, on the other hand, are always denoted by superscripts, as in $\xi_j = (\xi_j^1,\xi_j^2,\xi_j^3)$.

If $A \subset \R^3, B \subset \R^{1+3}$, then $\Proj_A, \Proj_B$ are the multipliers given by
$$
  \widehat{\Proj_A f}(\xi) = \chi_A(\xi)\widetilde f(\xi),
  \qquad
  \widetilde{\Proj_B u}(X) = \chi_B(X)\widetilde u(X),
$$
where $\chi_A, \chi_B$ are the characteristic functions of $A, B$. To simplify the notation when a set is given by some condition, we also write, for example, $\Proj_{\angles{\xi} \sim N}$ and $\chi_{\angles{\xi} \sim N}$ instead of $\Proj_{\{ \xi \colon \angles{\xi} \sim N \}}$ and $\chi_{\{ \xi \colon \angles{\xi} \sim N \}}$.

Let $D=-i\nabla$, where $\nabla=(\partial_1,\partial_2,\partial_3)$. Given $h : \R^3 \to \C$, we denote by $h(D)$ the multiplier given by $\widehat{h(D)f}(\xi) = h(\xi)\widehat f(\xi)$.

\subsection{Bilinear interactions}

Note the convolution identity
\begin{equation}\label{B:2}
  \widetilde{u_1\overline{u_2}}(X_0)
  \simeq \int
  \widetilde{u_1}(X_1)\,
  \widetilde{u_2}(X_2)
  \, d\mu^{12}_{X_0},
  \quad
  d\mu^{12}_{X_0} \equiv \delta(X_0-X_1+X_2) \, dX_1 \, dX_2.
\end{equation}
Here $\delta$ is the point mass at zero, hence
\begin{equation}\label{B:2:2}
  X_0 = X_1 - X_2 
  \qquad
  \left( \iff \text{$\tau_0 = \tau_1-\tau_2$ and $\xi_0 = \xi_1-\xi_2$} \right),
\end{equation}
motivating the following terminology: A triple $(X_0,X_1,X_2)$ of vectors in $\R^{1+3}$ is said to be a \emph{bilinear interaction} if \eqref{B:2:2} is verified.

Similarly, for a product without conjugation (then $X_0 = X_1 + X_2$)
\begin{equation}\label{B:2:4}
  \widetilde{u_1u_2}(X_0)
  \simeq \int \widetilde{u_1}(X_1)\widetilde{u_2}(X_2) \, d\nu^{12}_{X_0},
  \quad
  d\nu^{12}_{X_0} \equiv \delta(X_0-X_1-X_2) \, dX_1 \, dX_2.
\end{equation}

\subsection{$L^p$ and Sobolev norms}

All $L^p$ norms are taken with respect to Lebesgue measure. In fact, we use almost exclusively $L^2$ norms, hence we reserve the notation $\norm{\cdot}$ for the $L^2$ norm over $\R^n$, where $n=1$, $3$ or $1+3$, depending on the context. For example, if $u=u(t,x)$ is a space-time function, then $\norm{u}$ is understood to be taken over $\R^{1+3}$. If we are taking the norm of an expression, we indicate by a subscript which variable or variables the norm is taken over, as in $\fixednorm{F(\tau,\xi)}_{L^2_{\tau,\xi}}$. The $n$-dimensional Lebesgue-measure of a set $A \subset \R^n$ is denoted $\abs{A}$, where the value of $n$ will always be clear from the context.

Let $H^s$ be the completion of $\mathcal S(\R^3)$ with respect to $\norm{f}_{H^s} = \bignorm{\angles{\xi}^s\widehat f\,}_{L^2_\xi}$. Here, and throughout the paper, we use the shorthand
$$
  \angles{\xi} = \left(1 + \abs{\xi}^2\right)^{1/2}.
$$
An alternative, direct characterization is $H^s = \mathcal F_x^{-1} L^2 \left( \angles{\xi}^{2s} \, d\xi \right)$. In the statement of Theorem \ref{A:Thm1} we also refer to the space $\abs{D}^{-1} H^s$, which is defined by
$$
  \abs{D}^{-1} H^s
  = \mathcal F_x^{-1} \left\{ \frac{\widehat g(\xi)}{\abs{\xi}\angles{\xi}^s} \colon g \in L^2(\R^3) \right\}
  = \mathcal F_x^{-1} L^2 \left( \abs{\xi}^2\angles{\xi}^{2s} \, d\xi \right)
$$
with norm $\bignorm{\abs{\xi}\angles{\xi}^s\widehat f\,}_{L^2_\xi}$. Equivalently, $\abs{D}^{-1} H^s$ is the completion of $\mathcal S(\R^3)$ with respect to this norm.

With this definition, the regularity statement in part (a) of Theorem \ref{A:Thm1}  holds by the following lemma (with $f=0$ and $\mathbf u = \mathbf B_0$):

\begin{lemma}\label{P:Lemma}
Let $s \in \R$ and assume that $f \in H^s(\R^3;\R)$ and $\mathbf u \in H^s(\R^3,\R^3)$, with $\nabla \cdot \mathbf u = 0$. Then there exists a unique $\mathbf v \in \abs{D}^{-1}H^s( \R^3; \R^3)$ such that
\begin{equation}\label{P:2}
  \nabla \cdot \mathbf v = f,
  \qquad
  \nabla \times \mathbf v = \mathbf u.
\end{equation}
Moreover, $\partial_j v^k \in H^s(\R^3;\R)$ for $j,k=1,2,3$.
\end{lemma}

\begin{proof}
The identity $\nabla \times \nabla \times \mathbf v = \nabla (\nabla \cdot \mathbf v) - \Delta \mathbf v$ tells us that $\mathbf v$ must solve the Poisson equation $\Delta \mathbf v = - \nabla \times \mathbf u + \nabla f$. We therefore define $\mathbf v$ in Fourier space:
$$
  \widehat{\mathbf v}(\xi) = \frac{ \xi \times \widehat{\mathbf u}(\xi)}
  {\abs{\xi}^2} - \frac{\xi}{\abs{\xi}^2} \widehat f(\xi).
$$
Then clearly, $\mathbf v \in \abs{D}^{-1}H^s( \R^3; \R^3)$ (see section \ref{B} for the definition of this space) and it is easy to check that \eqref{P:2} is satisfied (here the assumption $\nabla \cdot \mathbf u = 0$ is needed). The uniqueness reduces to the fact that if $w \in \abs{D}^{-1}H^s$ and $\Delta w = 0$, then $w=0$. To prove this, note that $\abs{\xi}^2 \widehat w(\xi) = 0$ in the sense of tempered distributions. But since $w \in \abs{D}^{-1}H^s$, we know that $\widehat w$ is a measurable function, and it follows that $\widehat w = 0$ pointwise a.e., hence $w=0$.
\end{proof}

\subsection{Angles} Let $\vangle(a,b)$ be the angle between nonzero $a,b \in \R^3$. Then (see \cite{Selberg:2008a})
\begin{align}
  \label{B:4}
  \abs{a}+\abs{b}-\abs{a+b} &\sim \min(\abs{a},\abs{b}) \vangle(a,b)^2,
  \\
  \label{B:6}
  \abs{a-b}-\bigabs{\abs{a}-\abs{b}} &\sim \frac{\abs{a}\abs{b}}{\abs{a-b}} \vangle(a,b)^2 \qquad(a \neq b).
\end{align}

\subsection{Special sets}

The characteristic set of the operator $-i\partial_t\pm\abs{D}$ appearing in \eqref{B:18} is the null cone component
\begin{equation}\label{B:40}
  K^\pm = \left\{ (\tau,\xi) \in \R^{1+3} : \tau\pm\abs{\xi} = 0 \right\}.
\end{equation}
The union $K = K^+ \cup K^-$ is the full null cone; we say that a vector $X \in \R^{1+3}$ is \emph{null} if it belongs to $K$. The geometry of interactions of null cones plays a fundamental role in the analysis of the system \eqref{B:18}.

For $N,L \ge 1$, $r,\gamma > 0$ and $\omega \in \mathbb S^2$, where $\mathbb S^2 \subset \R^3$ is the unit sphere, define
\begin{align}
  \label{B:101}
  \Gamma_\gamma(\omega) &= \left\{ \xi \in \R^3 \colon \theta(\xi,\omega) \le \gamma \right\}
  \\
  \label{B:108}
  T_r(\omega) &= \left\{ \xi \in \R^3 \colon \abs{P_{\omega^{\perp}} \xi} \lesssim r \right\},
  \\
  \label{B:102}
  K^\pm_{N,L}
  &= \left\{ (\tau,\xi) \in \R^{1+3} \colon \angles{\xi} \sim N, \; \angles{\tau\pm\abs{\xi}} \sim L \right\},
  \\
  \label{B:104}
  K^{\pm}_{N,L,\gamma,\omega}
  &= \left\{ (\tau,\xi) \in \R^{1+3} \colon \angles{\xi} \sim N, \; \pm\xi \in \Gamma_{\gamma}(\omega), \; \angles{\tau\pm\abs{\xi}} \sim L \right\},
  \\
  \label{B:106}
  H_d(\omega) &= \left\{ (\tau,\xi) \in \R^{1+3} \colon \abs{\tau+\xi\cdot\omega} \lesssim d \right\},
\end{align}
where $P_{\omega^{\perp}}$ denotes the projection onto the orthogonal complement $\omega^\perp$ of $\omega$ in $\R^3$. Thus, $\Gamma_\gamma(\omega)$ is a conical sector around $\omega$, $T_r(\omega)$ is a tube of radius comparable to $r$ around $\R \omega$, $K^\pm_{N,L}$ and $K^{\pm}_{N,L,\gamma,\omega}$ consist of pieces of thickened null cones, and $H_d(\omega)$ is an $O(d)$-thickening of the null hyperplane $\tau+\xi\cdot\omega = 0$. Implicit absolute constants are used to make the notation more flexible.

Clearly,
\begin{equation}\label{B:110}
  K^{\pm}_{N,L,\gamma,\omega}
  \subset \R \times T_{N\gamma}(\omega).
\end{equation} 
We also claim that
\begin{equation}\label{B:112}
  K^{\pm}_{N,L,\gamma,\omega}
  \subset H_{\max(L,N\gamma^2)}(\omega).
\end{equation}
Indeed, if $(\tau,\xi) \in K^{\pm}_{N,L,\gamma,\omega}$, then $\tau+\xi \cdot \omega$ equals
$$
  \left(\tau\pm\abs{\xi}\right)
  - \left( \pm \abs{\xi} - \xi \cdot \omega \right)
  = O(L)
  - \frac{\abs{\xi}^2\bigl(1-\cos^2\theta(\pm\xi,\omega)\bigr)}{\pm\left(\abs{\xi}\pm\xi \cdot \omega\right)}
  = O(L) + O(N\gamma^2),
$$
where we used the fact that $\theta(\pm\xi,\omega) \le \gamma < 1$, hence $\pm\xi \cdot \omega \ge 0$.

\subsection{Angular decompositions}

Given $\gamma \in (0,\pi]$ and $\omega \in \mathbb S^2$, define $\Gamma_\gamma(\omega)$ as in \eqref{B:101}. For the purpose of decomposing $\R^3$ into such sectors without too much overlap, let $\Omega(\gamma)$ be a maximal $\gamma$-separated subset of the unit sphere $\mathbb S^2$. Then
\begin{equation}\label{B:66}
  1 \le \sum_{\omega \in \Omega(\gamma)} \chi_{\Gamma_\gamma(\omega)}(\xi)
  \le 5^2
  \qquad (\forall \xi \neq 0),
\end{equation}
where the left inequality holds by the maximality of $\Omega(\gamma)$, and the right inequality by the $\gamma$-separation, since the latter implies (we omit the proof):

\begin{lemma}\label{B:Lemma3} For $k \in \N$ and $\omega \in \mathbb S^2$,
$\#\left\{ \omega' \in \Omega(\gamma) : \vangle(\omega',\omega) \le k\gamma \right\}\le (2k+1)^2$.
\end{lemma}

The following will be used for angular decomposition in bilinear estimates.

\begin{lemma}\label{B:Lemma4} We have
$$
  1 \sim \sum_{\genfrac{}{}{0pt}{1}{0 < \gamma < 1}{\text{$\gamma$ \emph{dyadic}}}} 
  \sum_{\genfrac{}{}{0pt}{1}{\omega_1,\omega_2 \in \Omega(\gamma)}{3\gamma \le \vangle(\omega_1,\omega_2) \le 12\gamma}}
  \chi_{\Gamma_\gamma(\omega_1)}(\xi_1) \chi_{\Gamma_\gamma(\omega_2)}(\xi_2),
$$
for all $\xi_1,\xi_2 \in \R^3 \setminus \{0\}$ with $\vangle(\xi_1,\xi_2) > 0$.
\end{lemma}

We omit the straightforward proof. The condition $\vangle(\omega_1,\omega_2) \ge 3\gamma$ implies that the minimum angle between vectors in $\Gamma_{\gamma}(\omega_1)$ and $\Gamma_{\gamma}(\omega_2)$ is greater than or equal to $\gamma$, so the sectors are well-separated. If separation is not needed, the following variation may be preferable (again, we skip the easy proof):

\begin{lemma}\label{B:Lemma5} For any $0 < \gamma < 1$ and $k \in \N$,
$$
  \chi_{\vangle(\xi_1,\xi_2) \le k\gamma}(\xi_1,\xi_2) \lesssim \sum_{\genfrac{}{}{0pt}{1}{\omega_1,\omega_2 \in \Omega(\gamma)}{\vangle(\omega_1,\omega_2) \le (k+2)\gamma}}
  \chi_{\Gamma_\gamma(\omega_1)}(\xi_1) \chi_{\Gamma_\gamma(\omega_2)}(\xi_2),
$$
for all $\xi_1,\xi_2 \in \R^3 \setminus \{0\}$.
\end{lemma}

\subsection{Dyadic decompositions} Later we shall bound certain multilinear integral forms in terms of the norms \eqref{B:24}, and we will use an index $j$ to number the functions appearing in the multilinear forms. We rely on a dyadic decomposition based on the size of the weights in \eqref{B:24}. Throughout, $N$'s and $L$'s (indexed by $j$) will denote dyadic numbers greater than or equal to one, i.e., they are of the form $2^m$ for some nonnegative integer $m$. We then assign sizes $\angles{\xi} \sim N$ and $\angles{\tau\pm\abs{\xi}} \sim L$ to the weights in the norm \eqref{B:24}, and again this will be indexed by $j$. We call the $N$'s and the $L$'s \emph{elliptic} and \emph{hyperbolic} weights, respectively. To have a uniform notation, we shall assume throughout that the $F$'s are arbitrary nonnegative functions in $L^2(\R^{1+3})$, and that the $u$'s are given by
\begin{equation}\label{B:200}
  \widetilde u(X) = \chi_{K^{\pm}_{N,L}}(X) F(X) \qquad (X=(\tau,\xi) \in \R^{1+3}, \; F \in L^2(\R^{1+3}), \; F \ge 0),
\end{equation}
all of which is indexed by $j$. Often we also use angular decompositions, and for this we use the shorthand, for $\gamma > 0$ and $\omega \in \mathbb S^2$,
\begin{equation}\label{B:202}
  u^{\gamma,\omega}
  = \Proj_{\pm\xi \in \Gamma_{\gamma}(\omega)} u
  \qquad
  \left( \iff
  \widetilde{u^{\gamma,\omega}}
  = \chi_{K^{\pm}_{N,L,\gamma,\omega}} F
  \right),
\end{equation}
where everything except $\gamma$ is subject to indexation by $j$ ($\gamma$ is excepted because it relates two different $j$'s, as in Lemmas \ref{B:Lemma4} and \ref{B:Lemma5}). Then by \eqref{B:66},
\begin{equation}\label{B:204}
  \norm{u}
  \sim
  \left( \sum_{\omega \in \Omega(\gamma)}
  \norm{u^{\gamma,\omega}}^2 \right)^{1/2}.
\end{equation}
Finally, we note the following fact, which is used to sum $\omega_1,\omega_2$ in an angularly decomposed bilinear estimate: If $\gamma > 0$, then
\begin{equation}\label{B:208}
\begin{aligned}
  \sum_{\genfrac{}{}{0pt}{1}{\omega_1,\omega_2 \in \Omega(\gamma)}
  {\theta(\omega_1,\omega_2) \lesssim \gamma}}
  \norm{u_1^{\gamma,\omega_1}} \norm{u_2^{\gamma,\omega_2}}
  &\le \left( \sum_{\omega_1,\omega_2}
  \norm{u_1^{\gamma,\omega_1}}^2 \right)^{1/2}
  \left( \sum_{\omega_1,\omega_2}
  \norm{u_2^{\gamma,\omega_2}}^2 \right)^{1/2}
  \\
  &\lesssim \norm{u_1}\norm{u_2}.
\end{aligned}
\end{equation}
Here we applied the Cauchy-Schwarz inequality, then Lemma \ref{B:Lemma3} and \eqref{B:204}.

\section{Main estimates, duality and time cut-off}\label{C}

We shall iterate the Maxwell-Dirac-Lorenz system \eqref{B:18} in the spaces
$$
  \psi_\pm \in X_\pm^{s,1/2+\varepsilon}(S_T),
$$
where $\varepsilon > 0$ is sufficiently small, and $0 < T \le 1$ will depend on the data norm
\begin{equation}\label{C:3}
  \mathcal I_0
  = \norm{(\psi_0,\mathbf E_0,\mathbf B_0)}_{H^s \times H^{s-1/2} \times H^{s-1/2}}.
\end{equation}
By a standard argument (see \cite{Selberg:2007d}) the problem of obtaining closed estimates for the iterates reduces to proving the nonlinear estimates
\begin{align}
  \label{C:1}
  \bignorm{\mathbf\Pi_{\pm_2}\left( A_\mu^{\mathrm{hom.}} \boldsymbol\alpha^\mu \mathbf\Pi_{\pm_1} \psi_1 \right)}_{X_{\pm_2}^{s,-1/2+2\varepsilon}(S_T)}
  &\le C
  \, \mathcal I_0 \norm{\psi_1}_{X_{\pm_1}^{s,1/2+\varepsilon}(S_T)},
  \\
  \label{C:2}
  \bignorm{\mathbf\Pi_{\pm_4}\mathcal N\bigl(\mathbf\Pi_{\pm_1}\psi_1,\mathbf\Pi_{\pm_2}\psi_2,\mathbf\Pi_{\pm_3}\psi_3\bigr)}_{X_{\pm_4}^{s,-1/2+2\varepsilon}(S_T)}
  &\le C \prod_{j=1}^3 \norm{\psi_j}_{X_{\pm_j}^{s,1/2+\varepsilon}(S_T)},
\end{align}
where $C = C_{s,\varepsilon}$. By a standard argument, it suffices to prove these without the restriction to $S_T$, for all $\psi_j \in \mathcal S(\R^{1+3})$, and of course we can insert a smooth time cut-off $\rho : \R \to [0,1]$ with $\rho(t)=1$ for $\abs{t} \le 1$ and $\rho(t)=0$ for $\abs{t} \ge 2$.

Using also the fact that the dual of $X^{s,-1/2+2\varepsilon}_\pm$ is $X^{-s,1/2-2\varepsilon}_\pm$, we then reduce \eqref{C:1} and \eqref{C:2} to the respective integral estimates
\begin{align}
  \label{C:10}
  \Abs{I^{\pm_1,\pm_2}}
  &\le C_{s,\varepsilon}
  \mathcal I_0 \norm{\psi_1}_{X_{\pm_1}^{s,1/2+\varepsilon}}
  \norm{\psi_2}_{X_{\pm_2}^{-s,1/2-2\varepsilon}},
  \\
  \label{C:12}
  \Abs{J^{\pm_1,\dots,\pm_4}}
  &\le C_{s,\varepsilon} \norm{\psi_1}_{X_{\pm_1}^{s,1/2+\varepsilon}} \norm{\psi_2}_{X_{\pm_2}^{s,1/2+\varepsilon}} \norm{\psi_3}_{X_{\pm_3}^{s,1/2+\varepsilon}} \norm{\psi_4}_{X_{\pm_4}^{-s,1/2-2\varepsilon}},
\end{align}
where
\begin{align}
  \label{C:14}
  I^{\pm_1,\pm_2}
  &=
  \iint \rho A_\mu^{\mathrm{hom.}}
  \Innerprod{\boldsymbol\alpha^\mu \mathbf\Pi_{\pm_1} \psi_1}{\mathbf\Pi_{\pm_2} \psi_2} \, dt \, dx,
  \\
  \label{C:16}
  J^{\pm_1,\dots,\pm_4}
  &=
  \iint \rho \square^{-1} \Innerprod{\boldsymbol\alpha^\mu\mathbf\Pi_{\pm_1}\psi_1}{\mathbf\Pi_{\pm_2}\psi_2} \cdot\Innerprod{\boldsymbol\alpha_\mu\mathbf\Pi_{\pm_3}\psi_3}{\mathbf\Pi_{\pm_4}\psi_4} \, dt \, dx,
\end{align}
and the $\psi_j \in \mathcal S(\R^{1+3})$ are $\C^4$-valued.

The time cut-off $\rho$ is included in order to smooth out a singularity in the operator $\square^{-1}$ appearing in $\mathcal N$. In fact, the following holds (see \cite{Klainerman:1995b}):

\begin{lemma}\label{C:Lemma1} Given $G \in \mathcal S(\R^{1+3})$, set $u = \square^{-1} G$. Then $u = u_+ - u_-$, where
$$
  \widehat{u_\pm}(t,\xi) = \frac{e^{\mp it \abs{\xi}}}{4\pi\abs{\xi}} \int_{-\infty}^\infty \frac{e^{it(\tau'\pm\abs{\xi})}-1}{\tau'\pm\abs{\xi}} \widetilde G(\tau',\xi) \, d\tau'.
$$
Moreover, multiplying by the time cut-off and taking Fourier transform also in time,
$$
  \widetilde{\rho u_\pm}(\tau,\xi)
  \\
  = \int_{-\infty}^\infty
  \frac{\kappa_{\pm}(\tau,\tau';\xi)}{4\pi\abs{\xi}}
  \widetilde G(\tau',\xi) \, d\tau',
$$
where
$$
  \kappa_{\pm}(\tau,\tau';\xi)
  = \frac{\widehat\rho(\tau-\tau')-\widehat\rho(\tau\pm\abs{\xi})}{\tau'\pm\abs{\xi}}
$$
and $\widehat \rho(\tau)$ denotes the Fourier transform of $\rho(t)$.
\end{lemma}

We focus first on the quadrilinear estimate \eqref{C:12}, which is by far the most difficult of the two. The trilinear estimate \eqref{C:10} is proved at the end of the paper.

Write
\begin{gather*}
  s_1=s_2=s_3=s, \quad s_4=-s,
  \quad
  b_1=b_2=b_3=\frac12 + \varepsilon, \quad b_4=\frac12-2\varepsilon,
  \\
  \widetilde{\psi_j}
  = z_j \bigabs{\widetilde{\psi_j}},
  \qquad
  \bigabs{\widetilde{\psi_j}(X_j)}
  = \frac{F_j(X_j)}{\angles{\xi_j}^{s_j} \angles{\tau_j\pm_j\abs{\xi_j}}^{b_j}}
  \qquad (X_j = (\tau_j,\xi_j)), 
\end{gather*}
where $z_j : \R^{1+3} \to \C^4$ is measurable, $\abs{z_j}=1$ and $F_j \in L^2(\R^{1+3)}$, $F_j \ge 0$. Applying Plancherel's theorem, Lemma \ref{C:Lemma1} and \eqref{B:2} to \eqref{C:16}, we see that it is enough to prove \eqref{C:12} for
\begin{equation}\label{C:26}
  J^{\boldsymbol\Sigma}
  =
  \int
  \frac{\kappa_{\pm_0}(\tau_0,\tau_0';\xi_0)}{\abs{\xi_0}}
  \cdot
  \frac{q_{1234} \prod_{j=1}^4 F_j(X_j)}
  {\prod_{j=1}^4 \angles{\xi_j}^{s_j} \angles{\tau_j\pm_j\abs{\xi_j}}^{b_j}}
  \, d\mu^{12}_{X_0'} \, d\mu^{43}_{X_0}
  \, d\tau_0 \, d\tau_0' \, d\xi_0,
\end{equation}
where
\begin{gather}
  \boldsymbol\Sigma = (\pm_0,\pm_1,\pm_2,\pm_3,\pm_4),
  \\
  X_0' = (\tau_0',\xi_0), \qquad X_0 = (\tau_0,\xi_0),
  \qquad X_j = (\tau_j,\xi_j), \;\; j=1,\dots,4,
  \\
  \label{C:32}
  e_j = \pm_j \frac{\xi_j}{\abs{\xi_j}} \in \mathbb S^2 \qquad (j=0,\dots,4),
  \\
  \label{C:30}
  q_{1234} = \innerprod{\boldsymbol\alpha^\mu\mathbf\Pi(e_1)z_1(X_1)}{\mathbf\Pi(e_2)z_2(X_2)}
  \innerprod{\boldsymbol\alpha_\mu\mathbf\Pi(e_3)z_3(X_3)}{\mathbf\Pi(e_4)z_4(X_4)}.
\end{gather}
Note the implicit summation over $\mu=0,1,2,3$ in \eqref{C:30}. The convolution measures $d\mu^{12}_{X_0'}, d\mu^{43}_{X_0}$ are defined as in \eqref{B:2}, hence, in \eqref{C:26},
\begin{gather}
  \label{C:35:1}
  X_0' = X_1 - X_2, \qquad
  X_0 = X_4 - X_3,
  \\
  \label{C:35}
  \tau_0'=\tau_1-\tau_2, \qquad \tau_0=\tau_4-\tau_3,
  \qquad
  \xi_0=\xi_1-\xi_2=\xi_4-\xi_3.
\end{gather}
We may restrict the integration in \eqref{C:26} to the region where $\xi_j \neq 0$ for $j=0,\dots,4$, hence the unit vectors $e_j$ are well-defined, as are the angles
\begin{equation}\label{C:40}
  \theta_{jk} = \vangle(e_j,e_k) = \vangle(\pm_j\xi_j,\pm_k\xi_k) \qquad (j,k=0,\dots,4),
\end{equation}
which play a key role in our analysis.

The part of \eqref{C:26} corresponding to $\abs{\xi_0} \le 1$ is easy to treat (see section \ref{R}), so for the moment we shall restrict to $\abs{\xi_0} \ge 1$, hence we replace the weight $\abs{\xi_0}$ in \eqref{C:26} by $\angles{\xi_0}$. Now assign dyadic sizes to the weights in \eqref{C:26} (recall the convention that $N$'s and $L$'s are dyadic numbers greater than or equal to one)
$$
  \angles{\tau_0'\pm_0\abs{\xi_0}} \sim L_0',
  \qquad
  \angles{\tau_j\pm_j\abs{\xi_j}} \sim L_j,
  \qquad
  \angles{\xi_j} \sim N_j
  \qquad (j=0,\dots,4),
$$
and define
$$
  \boldN = (N_0,\dots,N_4),
  \qquad
  \boldL = (L_0,L_0',L_1,\dots,L_4).
$$
We will use the shorthand
$$
  \Nmin^{012} = \min(N_0,N_1,N_2),
  \qquad
  \Nmax^{012} = \max(N_0,N_1,N_2),
$$
and similarly for the $L$'s and for other index sets than $012$. In particular, an index $0'$ will refer to $L_0'$, so
$$
  \Lmin^{0'12} = \min(L_0',L_1,L_2)
$$
and
$$
  \Lmin^{00'} = \min(L_0,L_0'),
$$
for example. In the case of a three-index such as $0'12$ we also let $\Lmed^{0'12}$ denote the median.

By \eqref{C:35}, $\xi_0=\xi_1-\xi_2$ in \eqref{C:26}, so by the triangle inequality,  $N_j \lesssim N_k+N_l$ for all permutations $(j,k,l)$ of $(0,1,2)$, hence one of the following must hold:
\begin{subequations}\label{C:41}
\begin{alignat}{2}
  \label{C:41a}
  N_0 &\ll N_1 \sim N_2& \qquad &(\text{``low output''}),
  \\
  \label{C:41c}
  N_0 &\sim \Nmax^{12} \ge \Nmin^{12}& \qquad &(\text{``high output''}),
\end{alignat}
\end{subequations}
and similarly for the index 034. In view of \eqref{C:41}, $\Nmed^{012} \sim \Nmax^{012}$ and
\begin{equation}\label{C:43}
  \Nmin^{012} \Nmax^{012} \sim N_0\Nmin^{12},
\end{equation}
and similarly for the index 034.

We now pull out the dyadic weights in \eqref{C:26}, using the fact that
\begin{equation}\label{C:48}
  \kappa_{\pm}(\tau_0,\tau_0';\xi_0)
  \lesssim
  \frac{\sigma_{L_0,L_0'}(\tau_0-\tau_0')}{(L_0L_0')^{1/2}},
\end{equation}
where
\begin{equation}\label{C:60}
  \sigma_{L_0,L_0'}(r) 
  =
  \left\{
  \begin{alignedat}{2}
  &\frac{1}{\angles{r}^{2}}& \quad &\text{if $L_0 \sim L_0'$},
  \\
  &\frac{1}{(L_0L_0')^{1/2}}& \quad &\text{otherwise}.
  \end{alignedat}
  \right.
\end{equation}
To prove \eqref{C:48} we note that
\begin{equation}\label{C:49}
  \tau_0-\tau_0' = (\tau_0\pm_0\abs{\xi_0}) - (\tau_0'\pm_0\abs{\xi_0}),
\end{equation}
and apply the following with $p=\tau_0\pm_0\abs{\xi_0}$, $q=\tau_0-\tau_0'$:

\begin{lemma}\label{C:Lemma2} For any $M \in \N$,
\begin{equation}\label{C:46}
  \Abs{\frac{\widehat\rho(p)-\widehat\rho(q)}{p-q}}
  \le \frac{C_M}{\angles{p-q} \min(\angles{p},\angles{q})^M}.
\end{equation}
\end{lemma}

\begin{proof} Since $\widehat\rho$ is a Schwartz function, $\abs{\widehat\rho(\tau)} \le C_M\angles{\tau}^{-M}$. This immediately implies \eqref{C:46} if $\abs{p-q} > 1$. If, on the other hand, $\abs{p-q} \le 1$, then $\angles{p-q} \sim 1$ and $\angles{p} \sim \angles{q}$, and using $\abs{\widehat\rho\,'(\tau)} \le C_M\angles{\tau}^{-M}$ we get
$$
  \Abs{\frac{\widehat\rho(p)-\widehat\rho(q)}{p-q}}
  = \Abs{ \int_0^1 \widehat\rho\,'\bigl(q + \lambda(p-q)\bigr) \, d\lambda }
  \le \sup_{0 \le \lambda \le 1} \frac{C_M}{\angles{q + \lambda(p-q)}^M}
  \lesssim \frac{C_M}{\angles{q}^M}.
$$
\end{proof}

By \eqref{C:26} and \eqref{C:48},
\begin{equation}\label{C:56}
  \abs{J^{\mathbf\Sigma}}
  \lesssim \sum_{\boldN,\boldL} \frac{N_4^s J_{\boldN,\boldL}^{\mathbf\Sigma}}{N_0(L_0L_0')^{1/2}(N_1N_2N_3)^s
  (L_1 L_2 L_3)^{1/2+\varepsilon}L_4^{1/2-2\varepsilon}},
\end{equation}
where
\begin{multline}\label{C:58}
  J_{\boldN,\boldL}^{\mathbf\Sigma}
  =
  \int
  \Abs{q_{1234}}
  \sigma_{L_0,L_0'}(\tau_0-\tau_0') \,
  \chi_{K^{\pm_0}_{N_0,L_0}}\!\!(X_0) \,
  \chi_{K^{\pm_0}_{N_0,L_0'}}\!(X_0')
  \\
  \times
   \widetilde{u_1}(X_1) \widetilde{u_2}(X_2)
  \widetilde{u_3}(X_3) \widetilde{u_4}(X_4)
  \, d\mu^{12}_{X_0'} \, d\mu^{43}_{X_0}
  \, d\tau_0 \, d\tau_0' \, d\xi_0,
\end{multline}
with notation as in \eqref{B:102} and \eqref{B:200} (indexed by 1,2,3,4).

Note that the $\psi_j$ do not appear explicitly in \eqref{C:58}. In fact, from now on we can let the $F_j$ be arbitrary, nonnegative $L^2$-functions and the $z_j$ arbitrary measurable $\C^4$-valued functions with $\abs{z_j}=1$.
 
We have now reduced \eqref{C:12} to proving estimates for $J_{\boldN,\boldL}^{\mathbf\Sigma}$. These should of course be independent of $s$ and $\varepsilon$, so we would like to have the estimate which exactly balances the weights in \eqref{C:56} when $s=\varepsilon=0$:
\begin{equation}\label{C:80}
  J_{\boldN,\boldL}^{\mathbf\Sigma}
  \lesssim \bigl( N_0^2 L_0L_0'L_1L_2L_3L_4\bigr)^{1/2}
  \norm{u_1}\norm{u_2}\norm{u_3}\norm{u_4}.
\end{equation}
In fact, this holds in almost all the interactions, but for a certain case we have only been able to prove it up to a factor $\log\angles{L_0}$ on the right hand side.

We have the following result:

\begin{theorem}\label{C:Thm2}
The following estimate holds for all combinations of signs:
\begin{equation}\label{C:80:2}
  J_{\boldN,\boldL}^{\mathbf\Sigma}
  \lesssim \bigl( N_0^2 L_0L_0'L_1L_2L_3L_4\bigr)^{1/2}
  \log\angles{L_0}
  \norm{u_1}\norm{u_2}\norm{u_3}\norm{u_4}.
\end{equation}
\end{theorem}

The proof of this theorem takes up a large part of the paper (sections \ref{D:60}--\ref{E}). The logarithmic loss can likely be removed, but for our purposes it is harmless, since we assume $s,\varepsilon > 0$; once Theorem \ref{C:Thm2} has been proved, the main quadrilinear estimate \eqref{C:12} follows by a straightforward summation argument.

For later use we define the operator $T^{\pm_0}_{L_0,L_0'}$ by
\begin{equation}\label{C:90}
  T^{\pm_0}_{L_0,L_0'}F(\tau_0,\xi_0)
  =
  \int a^{\pm_0}_{L_0,L_0'}(\tau_0,\tau_0',\xi_0)
  F(\tau_0',\xi_0) \, d\tau_0',
\end{equation}
where
\begin{equation}\label{C:90:2}
  a^{\pm_0}_{L_0,L_0'}(\tau_0,\tau_0',\xi_0) 
  =
  \left\{
  \begin{alignedat}{2}
  &\frac{1}{\angles{\tau_0-\tau_0'}^{2}}& \quad &\text{if $L_0 \sim L_0'$},
  \\
  &\frac{\chi_{\tau_0\pm_0\abs{\xi_0}=O(L_0)}
  \chi_{\tau_0'\pm_0\abs{\xi_0}=O(L_0')}}{(L_0L_0')^{1/2}}&
  \quad &\text{otherwise}.
  \end{alignedat}
  \right.
\end{equation}

\begin{lemma}\label{D:Lemma3} $\bignorm{T^{\pm_0}_{L_0,L_0'}F} \lesssim \norm{F}$ for $F \in L^2(\R^{1+3})$.
\end{lemma}

\begin{proof}
By duality, this is equivalent to
\begin{equation}\label{C:92}
  \Abs{\iiint a^{\pm_0}_{L_0,L_0'}(\tau_0,\tau_0',\xi_0)
  F(\tau_0',\xi_0) G(\tau_0,\xi_0) \, d\tau_0 \, d\tau_0' \, d\xi_0}
  \lesssim \norm{F} \norm{G}.
\end{equation}
If $L_0 \sim L_0'$, we use the Cauchy-Schwarz inequality with respect to the measure $\angles{\tau-\tau_0}^{-2} d\tau_0 \, d\tau_0' \, d\xi_0$, obtaining
$$
  \text{l.h.s.}\eqref{C:92}
  \le
  \left( \iiint \frac{F^2(\tau_0',\xi_0)}{\angles{\tau_0-\tau_0'}^{2}} d\tau_0 \, d\tau_0' \, d\xi_0 \right)^{1/2}
  \left( \iiint \frac{G^2(\tau_0,\xi_0)}{\angles{\tau_0-\tau_0'}^{2}} d\tau_0 \, d\tau_0' \, d\xi_0 \right)^{1/2}.
$$
If $L_0 \ll L_0'$ or $L_0' \ll L_0$, then by the Cauchy-Schwarz inequality,
\begin{equation} \label{C:92:2}
\begin{aligned}
  \text{l.h.s.}\eqref{C:92}
  \le
  \frac{1}{(L_0L_0')^{1/2}}
  &\left( \iiint \chi_{\tau_0\pm_0\abs{\xi_0}=O(L_0)} F^2(\tau_0',\xi_0) d\tau_0 \, d\tau_0' \, d\xi_0 \right)^{1/2}
  \\
  \times
  &\left( \iiint \chi_{\tau_0'\pm_0\abs{\xi_0}=O(L_0')} G^2(\tau_0,\xi_0) d\tau_0 \, d\tau_0' \, d\xi_0 \right)^{1/2}.
\end{aligned}
\end{equation}
\end{proof}

In some situations we use the following variant of the last lemma:

\begin{lemma}\label{D:Lemma5} Assume that $L_0 \ll L_0'$ or $L_0' \ll L_0$. Let $\omega, \omega' \in \mathbb S^2$, $c,c' \in \R$ and $d,d' > 0$. Assume that $F,G \in L^2(\R^{1+3})$ satisfy
\begin{align*}
  \supp F &\subset \left\{ (\tau_0',\xi_0) \colon
  \tau_0'+\xi_0\cdot\omega' = c' + O(d') \right\},
  \\
  \supp G &\subset \left\{ (\tau_0,\xi_0) \colon
  \tau_0+\xi_0\cdot\omega = c + O(d) \right\}.
\end{align*}
Then
$$
  \bignorm{T^{\pm_0}_{L_0,L_0'}F}
  \lesssim \left(\frac{d'}{L_0'}\right)^{1/2} \norm{F},
$$
and
$$
  \Abs{\iint T^{\pm_0}_{L_0,L_0'}F(\tau_0,\xi_0) G(\tau_0,\xi_0) \, d\tau_0 \, d\xi_0}
  \lesssim 
  \left(\frac{dd'}{L_0L_0'}\right)^{1/2} \norm{F} \norm{G}.
$$
\end{lemma}

\begin{proof} Replace $\chi_{\tau_0'\pm_0\abs{\xi_0}=O(L_0')}$ in \eqref{C:92:2} by $\chi_{\tau_0'+\xi_0\cdot\omega'=c'+O(d')}$ to get the first estimate above. To get the second estimate we replace also $\chi_{\tau_0\pm_0\abs{\xi_0}=O(L_0)}$ in \eqref{C:92:2} by $\chi_{\tau_0+\xi_0\cdot\omega=c+O(d)}$.
\end{proof}

\section{Bilinear null structure: A review}\label{D}

Since bilinear $L^2$ estimates for the spaces $X^{s,b}_\pm$ are well understood, it is natural to try to reduce \eqref{C:80} directly to such estimates by applying the Cauchy-Schwarz inequality in the most obvious way. This approach fails, but it is worthwhile to go through the argument, since it leads us to the quadrilinear null structure.

Given a bounded, measurable $\sigma : \R^{1+3} \times \R^{1+3} \to \C$, we define the bilinear operator $\mathfrak B_\sigma(u_1,u_2)$ by inserting $\sigma(X_1,X_2)$ in the convolution integral \eqref{B:2}:
\begin{equation}\label{D:1}
  \mathcal F \mathfrak B_\sigma(u_1,u_2)
  (X_0)
  = \int
  \sigma(X_1,X_2)\,
  \widetilde{u_1}(X_1)\,
  \widetilde{u_2}(X_2)
  \, d\mu^{12}_{X_0}.
\end{equation}
Given signs $\pm_0,\pm_1,\pm_2$, we say that the bilinear interaction $(X_0,X_1,X_2)$ is \emph{null} if the hyperbolic weights
\begin{equation}\label{D:3}
  \hypwt_j \equiv \tau_j \pm_j \abs{\xi_j} \qquad (j=0,1,2)
\end{equation}
all vanish. This is dangerous, since we then get no help from the hyperbolic weights in the integral \eqref{C:26} (there we actually have two bilinear interactions, and the worst case would be when both are null simultaneously).

In the bilinear null interaction, $X_0,X_1,X_2$ are all null, and since $X_0 = X_1 - X_2$, it is clear geometrically that they must be collinear (otherwise $X_0$ could not end up lying on the null cone). Therefore, the angle $\theta_{12}=\theta(\pm_1\xi_1,\pm_2\xi_2)$ must vanish. It is therefore not surprising that if $\sigma(X_1,X_2) = O(\theta_{12})$, then we have better $L^2$ estimates for $\mathfrak B_\sigma(u_1,u_2)$ than for a generic product like $u_1\overline{u_2}$.

In fact, the following holds (we prove this below):

\begin{theorem}\label{D:Thm}
If the symbol $\sigma$ satisfies $\sigma(X_1,X_2) = O(\theta_{12})$, then with notation as in \eqref{B:200} (indexed by $1,2$),
$$
  \Bignorm{\Proj_{K^{\pm_0}_{N_0,L_0}} \mathfrak B_\sigma(u_1,u_2)}
  \lesssim (N_0L_0L_1L_2)^{1/2} \norm{u_1}\norm{u_2}.
$$
\end{theorem}

This estimate fails to hold for a generic product like $u_1\overline{u_2}$.

Now compare Theorem \ref{D:Thm} with the estimate \eqref{C:80} that we want to prove. Let us for simplicity replace $\sigma_{L_0,L_0'}(\tau_0-\tau_0')$ by $\delta(\tau_0-\tau_0')$, so that $\tau_0'=\tau_0$ and $L_0=L_0'$. Clearly, if we could estimate the absolute value of the symbol $q_{1234}$ by a product
\begin{equation}\label{D:4:2}
  \sigma_{12}(X_1,X_2) \sigma_{34}(X_3,X_4)
\end{equation}
such that
\begin{equation}
  \label{D:4:4}
  \sigma_{12}(X_1,X_2) = O(\theta_{12})
  \qquad
  \sigma_{34}(X_3,X_4) = O(\theta_{34}),
\end{equation}
where $\theta_{12},\theta_{34}$ are defined as in \eqref{C:40}, then we could apply the Cauchy-Schwarz inequality in \eqref{C:58}, and deduce \eqref{C:80} from Theorem \ref{D:Thm}.

Let us see if this works. Clearly, the absolute value of the symbol $q_{1234}$ is bounded by the sum over $\mu=0,1,2,3$ of the terms \eqref{D:4:2} with
\begin{align*}
  \sigma_{12}(X_1,X_2)
  &=
  \Abs{\Innerprod{\boldsymbol\alpha^\mu\mathbf\Pi(e_1)z_1(X_1)}{\mathbf\Pi(e_2)z_2(X_2)}},
  \\
  \sigma_{34}(X_3,X_4)
  &=
  \Abs{\Innerprod{\boldsymbol\alpha_\mu\mathbf\Pi(e_3)z_3(X_3)}{\mathbf\Pi(e_4)z_4(X_4)}}.
\end{align*}
But these symbols fail to satisfy \eqref{D:4:4} (take $\mu=0$ and recall that $\boldsymbol\alpha^0 = \mathbf I_{4 \times 4}$).

This is not quite the end of the story, however. Let us see what happens for $\mu = 1,2,3$. Then we can use the commutation identity
\begin{equation}\label{D:54}
  \boldsymbol\alpha^j \mathbf\Pi(e) = \mathbf\Pi(-e)\boldsymbol\alpha^j + e^j \mathbf I_{4\times4} \qquad (j=1,2,3; \; e \in \mathbb S^2).
\end{equation}
If it were not for the remainder term $e^j \mathbf I_{4\times4}$, we could apply the following:

\begin{lemma}\label{D:Lemma4}
Let $\boldsymbol \gamma$ be a $4 \times 4$ matrix. A sufficient condition for the symbol 
$$
  \sigma_{12}^{\boldsymbol\gamma}(X_1,X_2)
  =
  \Abs{\Innerprod{\boldsymbol\gamma\mathbf\Pi(e_1)z_1(X_1)}{\mathbf\Pi(e_2)z_2(X_2)}}
$$
to satisfy \eqref{D:4:4}, is that  
\begin{equation}\label{D:50}
  \boldsymbol\gamma\mathbf\Pi(\xi)=\mathbf\Pi(-\xi)\boldsymbol\gamma
  \qquad (\forall \xi \in \R^3).
\end{equation}
\end{lemma}

\begin{proof}
By \eqref{B:16}, \eqref{B:17} and  \eqref{D:50},
$$
  \Abs{\innerprod{\boldsymbol\gamma \mathbf\Pi(e_1) z_1}{\mathbf\Pi(e_2) z_2}}
  =
  \Abs{\innerprod{ \mathbf\Pi(e_2)\mathbf\Pi(-e_1) \boldsymbol\gamma z_1}{z_2}}
  \lesssim \theta(e_1,e_2) \abs{\boldgamma z_1}\abs{z_2},
$$
so \eqref{D:4:4} follows, since $\abs{z_1}=\abs{z_2}=1$.
\end{proof}

For the simpler Dirac-Klein-Gordon system (D-K-G), the analogue of \eqref{C:12} is obtained by replacing the $\boldsymbol\alpha$'s in \eqref{C:16} by $\boldsymbol\beta$, which does satisfy \eqref{D:50}. This was used in \cite{Selberg:2007d} to prove almost optimal local well-posedness for D-K-G in 3D. (Essentially this reduces to Theorem \ref{D:Thm}.)

For M-D, the part of \eqref{C:16} which corresponds to the sum over $\mu=1,2,3$, and which does not take into account the remainder term in \eqref{D:54}, can be treated in the same way as D-K-G. The crucial point, however, is that the remainder term can be combined with the term corresponding to $\mu=0$ in \eqref{C:16}, to produce a more complicated null structure, which is not bilinear, but quadrilinear; see section \ref{D:60}. In view of this, we cannot just rely on standard $L^2$ bilinear estimates such as Theorem \ref{D:Thm}, although these certainly play an important role. In addition, we will use a number of modified bilinear estimates proved by the third author in \cite{Selberg:2008a}. These estimates are recalled in section \ref{J}.

To end this section, we recall the standard $L^2$ bilinear estimates of the form
\begin{equation}\label{D:10}
  \Bignorm{ \Proj_{K^{\pm_0}_{N_0,L_0}} \left( u_1 \overline{u_2} \right) }
  \le C
  \norm{u_1} \norm{u_2}.
\end{equation}

We have the following:

\begin{theorem}\label{M:Thm} With notation as in \eqref{B:200} (indexed by $1,2$), the estimate \eqref{D:10} holds with
\begin{align}
\label{M:10}
  C
  &\sim \bigl( \Nmin^{012}\Nmin^{12} L_1 L_2 \bigr)^{1/2},
  \\
  \label{M:14}
  C
  &\sim \bigl( \Nmin^{012}\Nmin^{0j} L_0 L_j \bigr)^{1/2} \qquad (j=1,2),
  \\
  \label{M:17}
  C
  &\sim \bigl( N_0\Nmin^{12} \Lmin^{012} \Lmed^{012} \bigr)^{1/2},
  \\
  \label{M:18}
  C
  &\sim \bigl( (\Nmin^{012})^3 \Lmin^{012} \bigr)^{1/2},
\end{align}
for any choice of signs $\pm_0,\pm_1,\pm_2$.
\end{theorem}

\begin{proof}
By a standard argument (see, for example, \cite[Lemmas 3 and 4]{Selberg:2007d}), \eqref{M:10} follows from the analogous estimates for two solutions of the homogeneous wave equation, proved in \cite[Theorem 12.1]{Foschi:2000}. Alternatively, see \cite{Selberg:2008a} for a short, direct proof of \eqref{M:10}. By duality, \eqref{M:14} follows from \eqref{M:10}. Combining \eqref{M:10} and \eqref{M:14}, and recalling \eqref{C:43}, we then get \eqref{M:17}. Finally, via the Cauchy-Schwarz inequality, \eqref{M:18} reduces to a trivial volume estimate; see \cite[Eq.\ (37)]{Tao:2001}.
\end{proof}

It is now easy to prove Theorem \ref{D:Thm}. The only other ingredient needed is the following more or less standard result, which generalizes the observation, made above, that the angle $\theta_{12}$ must vanish in the null interaction.

\begin{lemma}\label{D:Lemma1} Consider a bilinear interaction $(X_0,X_1,X_2)$, with $\xi_j \neq 0$ for $j=1,2$. Given signs $(\pm_0,\pm_1,\pm_2)$, define the hyperbolic weights as in \eqref{D:3}, and define the angle $\theta_{12} = \theta(\pm_1\xi_1,\pm_2\xi_2)$. Then
\begin{equation}\label{D:30}
  \max\left( \abs{\hypwt_0}, \abs{\hypwt_1}, \abs{\hypwt_2} \right)
  \gtrsim
  \min\left(\abs{\xi_1},\abs{\xi_2}\right)\theta_{12}^2.
\end{equation}
Moreover, if
\begin{equation}\label{D:34}
  \abs{\xi_0} \ll \abs{\xi_1} \sim \abs{\xi_2}
  \qquad \text{and} \qquad
  \pm_1 \neq \pm_2,
\end{equation}
then
\begin{equation}\label{D:31}
  \theta_{12} \sim 1
  \qquad \text{and} \qquad
  \max\left( \abs{\hypwt_0}, \abs{\hypwt_1}, \abs{\hypwt_2} \right)
  \gtrsim
  \min\left(\abs{\xi_1},\abs{\xi_2}\right),
\end{equation}
whereas if \eqref{D:34} does not hold, then
\begin{equation}\label{D:32}
  \max\left( \abs{\hypwt_0}, \abs{\hypwt_1}, \abs{\hypwt_2} \right)
  \gtrsim
  \frac{\abs{\xi_1}\abs{\xi_2}\theta_{12}^2}{\abs{\xi_0}}.
\end{equation}
\end{lemma}

See \cite{Selberg:2008a} for a proof.

We can now prove Theorem \ref{D:Thm}. By Lemma \ref{D:Lemma1}, $\theta_{12} \lesssim (\Lmax^{012}/\Nmin^{12})^{1/2}$, hence
\begin{equation}\label{D:40}
\begin{aligned}
  \Bignorm{\Proj_{K^{\pm_0}_{N_0,L_0}} \mathfrak B_{\theta_{12}}(u_1,u_2)}
  &\lesssim \left( \frac{\Lmax^{012}}{\Nmin^{12}} \right)^{1/2}
  \Bignorm{ \Proj_{K^{\pm_0}_{N_0,L_0}} \left( u_1 \overline{u_2} \right) }
  \\
  &\lesssim \left( \frac{\Lmax^{012}}{\Nmin^{12}} \right)^{1/2}
  \bigl( N_0\Nmin^{12} \Lmin^{012} \Lmed^{012} \bigr)^{1/2}
  \norm{u_1}\norm{u_2},
\end{aligned}
\end{equation}
where we used \eqref{M:17} from Theorem \ref{M:Thm} to get the last inequality. Simplifying, we get Theorem \ref{D:Thm}.

We remark that Theorem \ref{M:Thm} is related to the so-called null form estimates first investigated in \cite{Klainerman:1993}, and subsequently in numerous other papers by various authors; see \cite{Foschi:2000} for a survey. For this reason, we call $\mathfrak B_{\theta_{12}}$ a \emph{null form}.

Later, we shall also use the related null form $\mathfrak B_{\theta_{12}}'$ corresponding to a product without conjugation. Let us write out both definitions here for easy reference:
\begin{align}
  \label{D:1:1}
  \mathcal F \mathfrak B_{\theta_{12}}(u_1,u_2)
  (X_0)
  &= \int
  \theta(\pm_1\xi_1,\pm_2\xi_2)\,
  \widetilde{u_1}(X_1)\,
  \widetilde{u_2}(X_2)
  \, d\mu^{12}_{X_0},
  \\
  \label{D:1:2}
  \mathcal F \mathfrak B'_{\theta_{12}}(u_1,u_2)
  (X_0)
  &= \int
  \theta(\pm_1\xi_1,\pm_2\xi_2)\,
  \widetilde{u_1}(X_1)\,
  \widetilde{u_2}(X_2)
  \, d\nu^{12}_{X_0},
\end{align}
with notation as in \eqref{B:2} and \eqref{B:2:4}.

\section{Quadrilinear null structure in Maxwell-Dirac}\label{D:60}

Consider the symbol
$$
  q_{1234} = \innerprod{\boldsymbol\alpha^\mu\mathbf\Pi(e_1)z_1}{\mathbf\Pi(e_2)z_2}
  \innerprod{\boldsymbol\alpha_\mu\mathbf\Pi(e_3)z_3}{\mathbf\Pi(e_4)z_4},
$$
appearing in \eqref{C:58}. Here $e_1,\dots,e_4 \in \R^3$ and $z_1,\dots,z_4 \in \C^4$ are unit vectors. The null structure will be expressed in terms of the angles
$$
  \theta_{jk} = \vangle(e_j,e_k),
$$
six of which are distinct. We shall refer to the index pairs 12 and 34 as the \emph{internal pairs}, and the angles $\theta_{12}$ and $\theta_{34}$ as the \emph{internal angles}. Angles between vectors from different internal pairs are then called \emph{external angles}. So the external angles are $\theta_{13}$, $\theta_{14}$, $\theta_{23}$ and $\theta_{24}$. Let us denote their minimum by
\begin{equation}\label{D:70}
  \phi = \min \left\{ \theta_{13}, \theta_{14}, \theta_{23}, \theta_{24} \right\}.
\end{equation}

\begin{lemma}\label{D:Lemma2}
With notation as above,
\begin{equation}\label{D:74}
  \abs{q_{1234}}
  \lesssim
  \theta_{12}\theta_{34} + \phi \max(\theta_{12},\theta_{34}) + \phi^2,
\end{equation}
for all unit vectors $e_1,\dots,e_4 \in \R^3$ and $z_1,\dots,z_4 \in \C^4$. \end{lemma}

\begin{proof}
By idempotency we may replace $z_j$ in $q_{1234}$ by $z_j' = \mathbf\Pi(e_j)z_j$, for $j=1,\dots,4$. We do this for notational convenience, the advantage being that
\begin{equation}\label{D:76}
  z_j' = \mathbf\Pi(e_j)z_j' \qquad (j=1,\dots,4).
\end{equation}
Note also that $\abs{z_j'} \le \abs{z_j} = 1$.

Let us first assume $\phi = \theta_{13}$. Apply \eqref{D:54} in $q_{1234}$, and make use of \eqref{D:76} and the identities \eqref{B:14}--\eqref{B:16}. Note the implicit summation in $q_{1234}$, and recall that $\boldsymbol\alpha_0 = - \boldsymbol\alpha^0 = \boldsymbol I_{4 \times 4}$. We thus get
$$
  q_{1234}
  = A + B + C + D,
$$
where
\begin{align*}
  A &= (e_1 \cdot e_3-1) \innerprod{z'_1}{z'_2}
  \innerprod{z'_3}{z'_4},
  \\
  B &=
  \innerprod{\mathbf\Pi(-e_1)\boldsymbol\alpha^j z'_1}{z'_2}
  \innerprod{\mathbf\Pi(-e_3)\boldsymbol\alpha_j z'_3}{z'_4},
  \\
  C &=
  \innerprod{e_1^j z'_1}{z'_2}
  \innerprod{\mathbf\Pi(-e_3)\boldsymbol\alpha_j z'_3}{z'_4},
  \\
  D &=
  \innerprod{\mathbf\Pi(-e_1)\boldsymbol\alpha_j z'_1}{z'_2}
  \innerprod{e_3^j z'_3}{z'_4}.
\end{align*}
Here $e_1^j$ are the coordinates of $e_1$, and we implicitly sum over $j=1,2,3$.

Clearly,
\begin{equation}\label{D:80}
  \abs{A} \le \abs{1-\cos\theta_{13}}
  \lesssim \theta_{13}^2.
\end{equation}
By \eqref{D:76} and \eqref{B:17},
\begin{equation}\label{D:82}
  \abs{B} \lesssim \theta_{12}\theta_{34}.
\end{equation}
Using \eqref{B:14} we write $C = C_1-C_2$, where
\begin{align*}
  C_1 &= \innerprod{z'_1}{z'_2}
  \innerprod{\mathbf\Pi(-e_3) \mathbf\Pi(e_1) z'_3}{z'_4},
  \\
  C_2 &= \innerprod{z'_1}{z'_2}
  \innerprod{\mathbf\Pi(-e_3)\mathbf\Pi(-e_1) z'_3}{z'_4}.
\end{align*}
By \eqref{B:16} and \eqref{B:17},
\begin{align*}
  \abs{C_1} \le \abs{\innerprod{\mathbf\Pi(-e_3) \mathbf\Pi(e_1) z'_3}{\mathbf\Pi(-e_3) \mathbf\Pi(e_4) z'_4}}
  \lesssim \theta_{13} \theta_{34},
  \\
  \abs{C_2} \le \abs{\innerprod{\mathbf\Pi(-e_1) \mathbf\Pi(e_3) z'_3}{\mathbf\Pi(-e_3) \mathbf\Pi(e_4) z'_4}}
  \lesssim \theta_{13} \theta_{34},
\end{align*}
hence
\begin{equation}\label{D:84}
  \abs{C} \lesssim \theta_{13}\theta_{34}.
\end{equation}
By symmetry with $C$,
\begin{equation}\label{D:86}
  \abs{D} \lesssim \theta_{13}\theta_{12}.
\end{equation}

Combining \eqref{D:80}--\eqref{D:86} gives $\abs{q_{1234}} \lesssim \theta_{12}\theta_{34} + \theta_{13}\theta_{12} + \theta_{13}\theta_{34} + \theta_{13}^2$, which proves the lemma under the assumption $\phi=\theta_{13}$. To remove that assumption, we observe that $q_{1234}$ can be written in four different ways, since the $\boldsymbol\alpha^\mu$ are hermitian:
\begin{align*}
  q_{1234} &= \innerprod{\boldsymbol\alpha^\mu\mathbf\Pi(e_1)z'_1}{\mathbf\Pi(e_2)z'_2}
  \innerprod{\boldsymbol\alpha_\mu\mathbf\Pi(e_3)z'_3}{\mathbf\Pi(e_4)z'_4}
  \\
  &= \innerprod{\boldsymbol\alpha^\mu\mathbf\Pi(e_1)z'_1}{\mathbf\Pi(e_2)z'_2}
  \innerprod{\mathbf\Pi(e_3)z'_3}{\boldsymbol\alpha_\mu\mathbf\Pi(e_4)z'_4}
  \\
  &= \innerprod{\mathbf\Pi(e_1)z'_1}{\boldsymbol\alpha^\mu\mathbf\Pi(e_2)z'_2}
  \innerprod{\boldsymbol\alpha_\mu\mathbf\Pi(e_3)z'_3}{\mathbf\Pi(e_4)z'_4}
  \\
  &= \innerprod{\mathbf\Pi(e_1)z'_1}{\boldsymbol\alpha^\mu\mathbf\Pi(e_2)z'_2}
  \innerprod{\mathbf\Pi(e_3)z'_3}{\boldsymbol\alpha_\mu\mathbf\Pi(e_4)z'_4},
\end{align*}
and by the same argument as above, these expressions can be used to prove the lemma in the cases where the minimum $\phi$ of the external angles is $\theta_{13}$, $\theta_{14}$, $\theta_{23}$ or $\theta_{24}$, respectively.
\end{proof}

When we apply this lemma, it is natural to distinguish the cases
\begin{gather}
  \label{D:97}
  \phi \lesssim \min(\theta_{12},\theta_{34}),
  \\
  \label{D:98}
   \min(\theta_{12},\theta_{34}) \ll \phi \lesssim \max(\theta_{12},\theta_{34}),
  \\
  \label{D:99}
  \max(\theta_{12},\theta_{34}) \ll \phi,
\end{gather}
which imply, respectively, that the first, second or third term in \eqref{D:74} dominates.

In certain situations, the last two cases can be treated simultaneously, by virtue of the following simplified version of Lemma \ref{D:Lemma2}.

\begin{lemma}\label{D:Lemma6} In the cases \eqref{D:98} and \eqref{D:99},
$$
  \abs{q_{1234}}
  \lesssim
  \theta_{13}\theta_{24}.
$$
\end{lemma}

\begin{proof}
Note that in cases \eqref{D:98} and \eqref{D:99},
\begin{equation}\label{D:127}
  \min(\theta_{12},\theta_{34}) \ll \phi \le \theta_{13}, \theta_{14}, \theta_{23}, \theta_{24}.
\end{equation}
In case \eqref{D:99} the dominant term in \eqref{D:74} is $\phi^2$, hence the lemma follows from \eqref{D:127}. To handle case \eqref{D:98}, note that by symmetry we may assume that $\theta_{12} \le \theta_{34}$. Then we claim that \eqref{D:127} implies
\begin{equation}\label{D:100}
  \theta_{13} \sim \theta_{23},
  \qquad
  \theta_{14} \sim \theta_{24}.
\end{equation}
Granting this, we next notice that since we are in case \eqref{D:98}, and since $\theta_{12} \le \theta_{34}$, the dominant term in \eqref{D:74} is $\phi\theta_{34}$, but $\theta_{34} \le \theta_{13} + \theta_{14} \lesssim \max(\theta_{13},\theta_{24})$ by \eqref{D:100}, and $\phi \le \min(\theta_{13},\theta_{24})$ by \eqref{D:127}, so the lemma holds.

It remains to prove \eqref{D:100}. By the triangle inequality for distances on the unit sphere, $\theta_{13} \le \theta_{12} + \theta_{23}$, and since $\theta_{12} \ll \theta_{13}$, it follows that $\theta_{13} \lesssim \theta_{23}$. Repeating this argument with the index 13 replaced by $23$, $14$ and $24$, we get \eqref{D:100}, so the claim is proved. This completes the proof of the lemma.
\end{proof}

Our general strategy is now to apply the Cauchy-Schwarz inequality in various ways to reduce \eqref{C:80} to bilinear $L^2$ estimates. The standard estimates suffice in the particularly easy case \eqref{D:97}, which we dispose of straight away: The term $\theta_{12}\theta_{34}$ then dominates in \eqref{D:74}, hence, with notation as in \eqref{D:1:1},
\begin{equation}\label{D:102}
\begin{aligned}
  J_{\boldN,\boldL}^{\mathbf\Sigma}
  &\lesssim
  \Bignorm{T^{\pm_0}_{L_0,L_0'}
  \mathcal F \, \Proj_{K^{\pm_0}_{N_0,L_0'}}
  \mathfrak B_{\theta_{12}}(u_1,u_2)}
  \Bignorm{\Proj_{K^{\pm_0}_{N_0,L_0}}
  \mathfrak B_{\theta_{34}}(u_3,u_4)}
  \\
  &\lesssim
  \Bignorm{\Proj_{K^{\pm_0}_{N_0,L_0'}}
  \mathfrak B_{\theta_{12}}(u_1,u_2)}
  \Bignorm{\Proj_{K^{\pm_0}_{N_0,L_0}}
  \mathfrak B_{\theta_{34}}(u_3,u_4)}
  \\
  &\lesssim
   \left(N_0L_0'L_1L_2\right)^{1/2}\left(N_0L_0L_3L_4\right)^{1/2}
  \norm{u_1} \norm{u_2} \norm{u_3} \norm{u_4},
\end{aligned}
\end{equation}
where we applied the Cauchy-Schwarz inequality, Lemma \ref{D:Lemma3} and Theorem \ref{D:Thm}.

This proves Theorem \ref{C:Thm2} in the case \eqref{D:97}, so henceforth we assume that \eqref{D:98} or \eqref{D:99} holds. Then the standard $L^2$ bilinear estimates do not suffice. In addition we need a number of modified bilinear estimates proved by the third author in \cite{Selberg:2008a}, which we recall in the next section. In section \ref{F} we fill in some details about bilinear interactions, and then in sections \ref{G} and \ref{E} we finally prove Theorem \ref{C:Thm2}.

\section{Additional bilinear estimates}\label{J}

Here we recall a number of bilinear estimates proved in \cite{Selberg:2008a}. For more about the motivation behind these estimates, see \cite{Selberg:2008a}.

\subsection{Anisotropic bilinear estimate}

\begin{theorem}\label{J:Thm3}
Let $\omega \in \mathbb S^2$, $I \subset \R$ a compact interval. In addition to the usual assumption \eqref{B:200}, assume now
$$
  \supp \widehat{u_1} \subset
  \left\{ (\tau,\xi) \colon
  \theta(\xi,\omega^\perp) \ge \alpha \right\}
  \qquad \text{for some $\;0 < \alpha \ll 1$}.
$$
Then
$$
  \norm{\Proj_{\xi_0 \cdot \omega \in I}
  (u_1 u_2)}
  \lesssim \left( \frac{\abs{I}\Nmin^{12}L_1L_2 }{\alpha} \right)^{1/2}
  \norm{u_1}
  \norm{u_2}.
$$
The same estimate holds for $\norm{\Proj_{\xi_1 \cdot \omega \in I} u_1 \cdot u_2}$ and $\norm{u_1 \cdot \Proj_{\xi_2 \cdot \omega \in I} u_2}$. 
\end{theorem}

Here $\omega^\perp \subset \R^3$ is the orthogonal complement of $\omega$, and $\abs{I}$ is the length of $I$.

\subsection{Null form estimate with tube restricition}

Recall that $T_r(\omega) \subset \R^3$, for $r > 0$ and $\omega \in \mathbb S^2$, denotes a tube of radius comparable to $r$ around $\R\omega$. Recall also the definition of the null forms in \eqref{D:1:1} and \eqref{D:1:2}.

\begin{theorem}\label{G:Thm} Let $r > 0$ and $\omega \in \mathbb S^2$. Then with notation as in \eqref{B:200},
\begin{equation}\label{G:16}
  \bignorm{\mathfrak B_{\theta_{12}}(\Proj_{\R \times T_r(\omega)} u_1,u_2)}
  \lesssim \left( r^2 L_1 L_2 \right)^{1/2}
  \norm{u_1}\norm{u_2}.
\end{equation}
Moreover, the same estimate holds for $\mathfrak B_{\theta_{12}}'$.
\end{theorem}

The key point here is that we are able to exploit concentration of the Fourier supports near a null ray, which is not possible for a standard product like $u_1 \overline{u_2}$.

\begin{remark} In \cite{Selberg:2008a}, \eqref{G:16} is proved under the hypothesis that $\abs{\xi_j} \sim N_j$ on the support of $\widetilde{u_j}$, for $j=1,2$, instead of $\angles{\xi_j} \sim N_j$, as we have here. This only makes as difference if $\Nmin^{12} \sim 1$, however, but in that case, recalling also the standing assumption $L_1,L_2 \ge 1$, \eqref{G:16} follows from
$$
  \bignorm{\Proj_{\R \times T_r(\omega)} u_1 \cdot \overline{u_2}}
  \lesssim \left( r^2 \Nmin^{12} \Lmin^{12} \right)^{1/2}
  \norm{u_1}\norm{u_2}.
$$
This estimate is also proved in \cite{Selberg:2008a}, but requires only $\abs{\xi_j} \lesssim N_j$ for $j=1,2$.
\end{remark}

\begin{remark} In \cite{Selberg:2008a}, \eqref{G:16} is proved for $\mathfrak B_{\theta_{12}}'$, but it is easy to see that this implies the same estimate for $\mathfrak B_{\theta_{12}}$.
\end{remark}

\subsection{Concentration/nonconcentration null form estimate}

In the following refinement of Theorem \ref{G:Thm} we limit attention to interactions which are nearly null, by restricting the symbol in \eqref{D:1:1} and \eqref{D:1:2} to $\theta_{12} \ll 1$; we denote these modified null forms by $\mathfrak B_{\theta_{12} \ll 1}$ and $\mathfrak B'_{\theta_{12} \ll 1}$.
 
\begin{theorem}\label{J:Thm1}
Let $r > 0$, $\omega \in \mathbb S^2$ and $I_0 \subset \R$ a compact interval. Assume that $N_1, N_2 \gg 1$ and that
\begin{equation}\label{J:0:2}
  r \ll \Nmin^{12}.
\end{equation}
Then with notation as in \eqref{B:200},
$$
  \norm{\Proj_{\xi_0 \cdot \omega \in I_0} \mathfrak B_{\theta_{12} \ll 1}(\Proj_{\R \times T_r(\omega)} u_1,u_2)}
  \lesssim \left( r^2 L_1 L_2 \right)^{1/2} 
  \left( \sup_{I_1} \norm{\Proj_{\xi_1 \cdot \omega \in I_1} u_1} \right)
  \norm{u_2},
$$
where the supremum is over all translates $I_1$ of $I_0$. The same holds for $\mathfrak B'_{\theta_{12} \ll 1}$.
\end{theorem}

Here the condition $N_1,N_2 \gg 1$ serves to ensure that the spatial Fourier supports, given by the conditions $\angles{\xi_j} \sim N_j$ for $j=1,2$, do not degenerate to balls. 

\subsection{Nonconcentration low output estimate}

The following result improves \eqref{M:17} in certain situations.

\begin{theorem}\label{J:Thm2} Assume $N_0 \ll N_1 \sim N_2$, and define $r = (N_0\Lmax^{012})^{1/2}$. Then with notation as in \eqref{B:200},
$$
  \Bignorm{\Proj_{K^{\pm_0}_{N_0,L_0}}\!\!(u_1u_2)}
  \lesssim \left( N_1^2 \Lmin^{012} \Lmed^{012} \right)^{1/2} \norm{u_1}
  \sup_{\omega \in \mathbb S^2} \norm{\Proj_{\R \times T_{r}(\omega)}u_2}.
$$
\end{theorem}

In fact, we do not use Theorem \ref{J:Thm2} as stated, but we use the main ideas from its proof, a key ingredient of which is the following partial orthogonality estimate for a family of thickened null hyperplanes corresponding to a set of well-separated directions on the unit sphere.

\begin{lemma}\label{J:Lemma} Suppose $N,d > 0$, $\omega_0 \in \mathbb S^2$ and $0 < \gamma < \gamma' < 1$. The estimate
\begin{equation}\label{J:40}
  \sum_{\genfrac{}{}{0pt}{1}{\omega \in \Omega(\gamma)}{\theta(\omega,\omega_0) \le \gamma'}} \chi_{H_d(\omega)}(\tau,\xi)
  \lesssim \frac{\gamma'}{\gamma} + \frac{d}{N\gamma^2}
\end{equation}
holds for all $(\tau,\xi) \in \R^{1+3}$ with $\abs{\xi} \sim N$.
\end{lemma}

\subsection{Null form estimate with ball restriction} In the following analogue of Theorem \ref{G:Thm}, the tube is replaced by a ball. Then the symbol $\theta_{12}$ in \eqref{D:1:1} and \eqref{D:1:2} can be replaced by $\sqrt{\theta_{12}}$, defining $\mathfrak B_{\sqrt{\theta_{12}}}$ and $\mathfrak B'_{\sqrt{\theta_{12}}}$.

\begin{theorem}\label{G:Thm2} Let $r > 0$, and let $B \subset \R^3$ be a ball of radius $r$ and arbitrary center. Then with notation as in \eqref{B:200},
$$
  \bignorm{\Proj_{\R \times B} \mathfrak B_{\sqrt{\theta_{12}}}(u_1,u_2)}
  \lesssim \left( r^2 L_1 L_2 \right)^{1/2}
  \norm{u_1}\norm{u_2}.
$$
The same holds if $\Proj_{\R \times B}$ is placed in front of either $u_1$ or $u_2$, instead of outside the product. Moreover, the same estimates hold for $\mathfrak B'_{\sqrt{\theta_{12}}}$
\end{theorem}

\section{Some properties of bilinear interactions}\label{F}

Here we fill in some details about the bilinear interaction $X_0=X_1-X_2$, where $X_j=(\tau_j,\xi_j)$. Given signs $(\pm_0,\pm_1,\pm_2)$, we define the hyperbolic weights $\hypwt_j$ as in \eqref{D:3}, and we assume $\xi_j \neq 0$ for $j=0,1,2$, so that the unit vectors $e_j$ and the angles $\theta_{jk}$ are well-defined, as in \eqref{C:32} and \eqref{C:40}.

In Lemma \ref{D:Lemma1} we related the angle $\theta_{12}$ to the size of the hyperbolic weights $\hypwt_j$ and the elliptic weights $\abs{\xi_j}$. The sign $\pm_0$ was arbitrary, but by keeping track of the sign $\pm_0$ we can get additional information. In fact, since $\tau_0=\tau_1-\tau_2$,
\begin{equation}\label{F:0}
  \hypwt_0 - \hypwt_1 + \hypwt_2
  = \pm_0\abs{\xi_0}-\pm_1\abs{\xi_1} \pm_2\abs{\xi_2},
\end{equation}
so by \eqref{B:4}, \eqref{B:6} and the fact that $\xi_0=\xi_1-\xi_2$, we get the information in Table \ref{F:Table}.

\begin{table}
\def\arraystretch{1.7}
\begin{center}
\begin{tabular}{|c|c|}
\hline
  $(\pm_0,\pm_1,\pm_2) =$ & $\abs{\hypwt_0} + \abs{\hypwt_1} + \abs{\hypwt_2} \ge$
  \\
  \hline\hline
  $(+,+,+)$ or $(-,-,-)$& $\abs{\xi_0} + \abs{\xi_2} - \abs{\xi_1}
  \sim \min(\abs{\xi_0},\abs{\xi_2}) \theta_{02}^2$
  \\
  \hline
  $(-,+,+)$ or $(+,-,-)$& $\abs{\xi_0} + \abs{\xi_1} - \abs{\xi_2}
  \sim \min(\abs{\xi_0},\abs{\xi_1}) \theta_{01}^2$
  \\
  \hline
  $(+,+,-)$ or $(-,-,+)$& $\abs{\xi_0} + \abs{\xi_1} + \abs{\xi_2}$
  \\
  \hline
  $(-,+,-)$ or $(+,-,+)$& $\abs{\xi_1} + \abs{\xi_2} - \abs{\xi_0}
  \sim \min(\abs{\xi_1},\abs{\xi_2}) \theta_{12}^2$
  \\
  \hline
\end{tabular}
\end{center}
\caption{Estimates for the bilinear interaction.}
\label{F:Table}
\end{table}

From this table we see that for certain bilinear interactions,
\begin{equation}\label{F:2}
  \max\left(\abs{\hypwt_0},\abs{\hypwt_1},\abs{\hypwt_2}\right)
  \gtrsim \abs{\xi_0},
\end{equation}
which is good because it excludes the null interaction. In general, \eqref{F:2} holds if $\pm_0$ is \emph{different from} the following sign:

\begin{definition}\label{F:Def}
We define the sign, depending on $(\pm_1,\pm_2)$ and $(\abs{\xi_1},\abs{\xi_2})$,
$$
  \pm_{12}
  =
  \begin{cases}
  + \qquad &\text{if $(\pm_1,\pm_2) = (+,+)$ and $\abs{\xi_1} > \abs{\xi_2}$},
  \\
  - \qquad &\text{if $(\pm_1,\pm_2) = (+,+)$ and $\abs{\xi_1} \le \abs{\xi_2}$},
  \\
  + \qquad &\text{if $(\pm_1,\pm_2) = (+,-)$}.
  \end{cases}
$$
The definitions in the remaining cases $(\pm_1,\pm_2) = (-,-), (-,+)$ are then obtained by reversing all three signs $\pm_{12},\pm_1,\pm_2$ above.
\end{definition}

As remarked already we then have, by the above table:

\begin{lemma}\label{F:Lemma2}
If $\pm_0 \neq \pm_{12}$, then \eqref{F:2} holds.
\end{lemma}

Now consider the case of equal signs. The possible interactions are illustrated in Figures \ref{F:Fig1} and \ref{F:Fig2}. From the sine rule we obtain:
\begin{lemma}\label{F:Lemma1}
If $\pm_0 = \pm_{12}$, then
\begin{equation}\label{F:4}
  \min\left(\theta_{01},\theta_{02}\right)
  \sim
  \frac{\min\left(\abs{\xi_1},\abs{\xi_2}\right)}{\abs{\xi_0}} \sin \theta_{12}.
\end{equation}
Moreover, if $\pm_0 = \pm_{12}$ and $\pm_1 \neq \pm_2$, then
\begin{equation}\label{F:5}
  \max\left(\theta_{01},\theta_{02}\right)
  \sim
  \theta_{12}.
\end{equation}
\end{lemma}

\begin{proof} It suffices to consider the cases $(\pm_1,\pm_2) = (+,+), (+,-)$. In both cases, the law of sines yields (see Figures \ref{F:Fig1} and \ref{F:Fig2})
\begin{align}
  \label{F:6}
  \abs{\xi_2}\sin\theta_{12} &= \abs{\xi_0}\sin\theta_{01},
  \\
  \label{F:8}
  \abs{\xi_1}\sin\theta_{12} &= \abs{\xi_0}\sin\theta_{02}.
\end{align}
Now we claim that
\begin{equation}\label{F:10}
  \min\left(\theta_{01},\theta_{02}\right)
  = \begin{cases}
  \theta_{01} \quad &\text{if $\abs{\xi_1} > \abs{\xi_2}$},
  \\
  \theta_{02} \quad &\text{if $\abs{\xi_1} \le \abs{\xi_2}$},
  \end{cases}
\end{equation}
and
\begin{equation}\label{F:12}
  0 \le \min\left(\theta_{01},\theta_{02}\right) \le \frac{\pi}{2}.
\end{equation}
Clearly, \eqref{F:6}--\eqref{F:12} imply \eqref{F:4}, so it remains to prove the claim.

First, assume that $(\pm_1,\pm_2) = (+,+)$. If $\abs{\xi_1} > \abs{\xi_2}$, then from Figure \ref{F:Fig1}(a) we see that $\theta_{02} = \theta_{01} + \theta_{12}$ and $\theta_{01} \in [0,\pi/2]$, whereas if $\abs{\xi_1} \le \abs{\xi_2}$, then as in Figure \ref{F:Fig1}(b) we have $\theta_{01} = \theta_{02} + \theta_{12}$ and $\theta_{02} \in [0,\pi/2]$. Thus, \eqref{F:10} and \eqref{F:12} hold.

Second, consider $(\pm_1,\pm_2) = (+,-)$. From Figure \ref{F:Fig2} we see that $\theta_{12} = \theta_{01} + \theta_{02}$, so the minimum of $\theta_{01}$ and $\theta_{02}$ must belong to $[0,\pi/2]$, proving \eqref{F:12}.  If $\abs{\xi_1} > \abs{\xi_2}$, the minimum is $\theta_{01}$, whereas if $\abs{\xi_1} \le \abs{\xi_2}$, the minimum is $\theta_{02}$, so \eqref{F:10} holds. Since $\theta_{12} = \theta_{01} + \theta_{02}$, \eqref{F:5} is immediate.
\end{proof}

\begin{figure}
\begin{center}
  \includegraphics{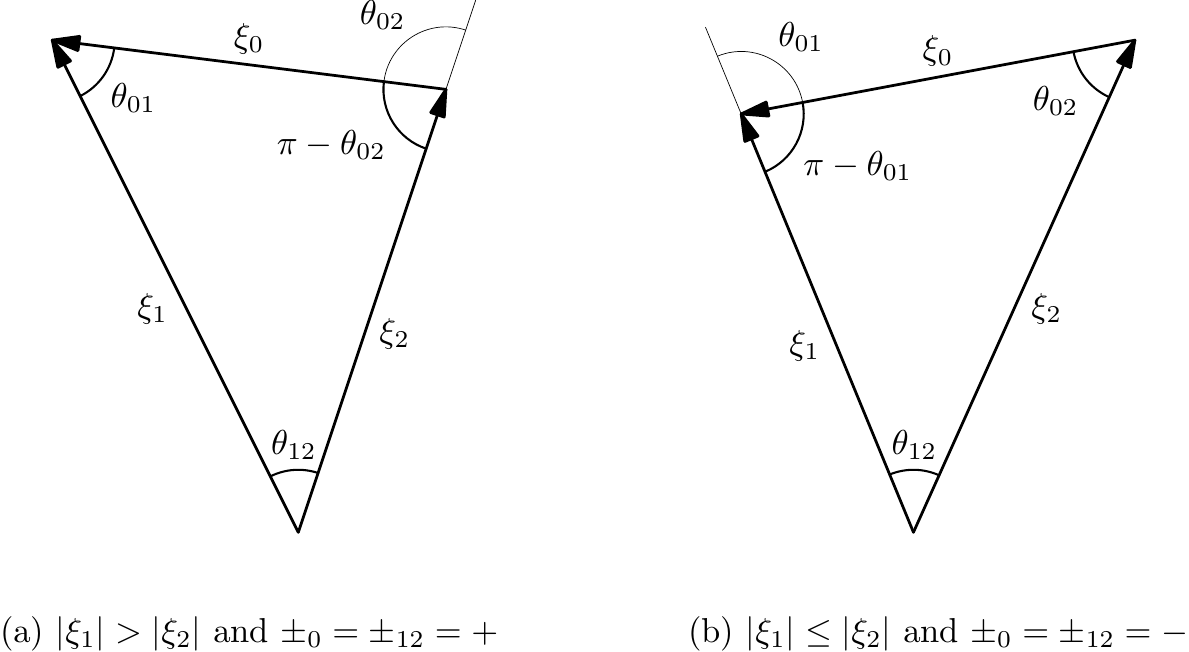}
\caption{$(\pm_1,\pm_2) = (+,+)$ and $\pm_0=\pm_{12}$.}
\label{F:Fig1}
\end{center}
\end{figure}

\begin{figure}
\begin{center}
  \includegraphics{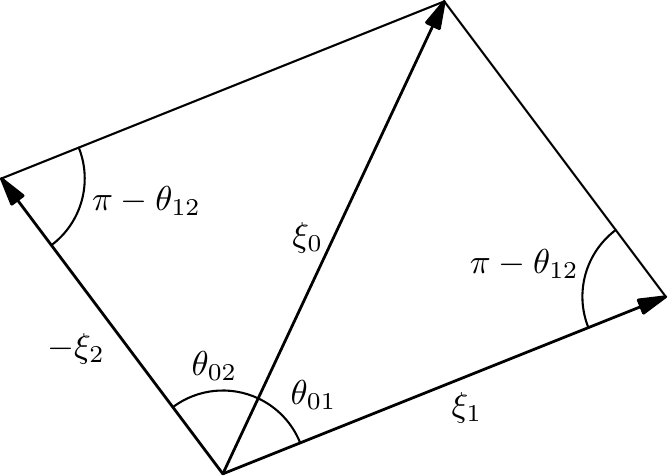}
\caption{$(\pm_1,\pm_2) = (+,-)$ and $\pm_0=\pm_{12}=+$. Here $\theta_{12} = \theta_{01} + \theta_{02}$.}
\label{F:Fig2}
\end{center}
\end{figure}

The following is reminiscent of \eqref{D:30} in Lemma \ref{D:Lemma1}.

\begin{lemma}\label{F:Lemma4}
For all signs,
\begin{equation}\label{F:14}
  \max\left( \abs{\hypwt_0}, \abs{\hypwt_1}, \abs{\hypwt_2} \right)
  \gtrsim \abs{\xi_0} \min\left(\theta_{01},\theta_{02}\right)^2. 
\end{equation}
\end{lemma}

\begin{proof} If $\pm_0\neq\pm_{12}$, \eqref{F:14} holds by Lemma \ref{F:Lemma2}, so we may assume $\pm_0=\pm_{12}$, hence \eqref{F:4} is valid. We split into the cases $\theta_{12} \ll 1$ and $\theta_{12} \sim 1$.

If $\theta_{12} \ll 1$, Lemma \ref{D:Lemma1} implies \eqref{D:32}, and using also \eqref{F:4} we get
$$
  \abs{\xi_0} \min\left(\theta_{01},\theta_{02}\right)^2
  \sim
  \frac{\min\left(\abs{\xi_1},\abs{\xi_2}\right)^2}{\abs{\xi_0}} \theta_{12}^2
  \lesssim
  \frac{\min\left(\abs{\xi_1},\abs{\xi_2}\right)\max\left( \abs{\hypwt_0}, \abs{\hypwt_1}, \abs{\hypwt_2} \right)}{\max\left(\abs{\xi_1},\abs{\xi_2}\right)},
$$
so \eqref{F:14} holds. 

Next, assume $\theta_{12} \sim 1$. Then by \eqref{D:30} in Lemma \ref{D:Lemma1},
$$
  \min\left(\abs{\xi_1},\abs{\xi_2}\right)
  \lesssim\max\left( \abs{\hypwt_0}, \abs{\hypwt_1}, \abs{\hypwt_2} \right),
$$
and combining this with \eqref{F:4} we get
$$
  \abs{\xi_0} \min\left(\theta_{01},\theta_{02}\right)^2
  \lesssim
  \abs{\xi_0} \min\left(\theta_{01},\theta_{02}\right)
  \lesssim
  \min\left(\abs{\xi_1},\abs{\xi_2}\right)
  \lesssim
  \max\left( \abs{\hypwt_0}, \abs{\hypwt_1}, \abs{\hypwt_2} \right),
$$
so \eqref{F:14} again holds. 
\end{proof}

Based on the previous lemmas we now prove the following.

\begin{lemma}\label{F:Lemma3} If $\pm_0 = \pm_{12}$ and $\pm_1=\pm_2$, then
\begin{equation}\label{F:20}
  \frac{\abs{\xi_1}\abs{\xi_2}\theta_{12}^2}{\abs{\xi_0}} \sim \min\left( \abs{\xi_0}, \abs{\xi_1}, \abs{\xi_2}\right) \max(\theta_{01},\theta_{02})^2
  \lesssim
  \max\left( \abs{\hypwt_0}, \abs{\hypwt_1}, \abs{\hypwt_2} \right),
\end{equation}
whereas if $\pm_0 = \pm_{12}$ and $\pm_1 \neq \pm_2$, then
\begin{equation}\label{F:22}
  \max\left(\theta_{01},\theta_{02}\right)
  \sim
  \theta_{12}.
\end{equation}

\end{lemma}

\begin{proof} The last statement was proved in Lemma \ref{F:Lemma1}, so we only prove \eqref{F:20}. It suffices to consider $(\pm_1,\pm_2) = (+,+)$. If $\abs{\xi_1} > \abs{\xi_2}$, then $\pm_0 = \pm_{12} = +$, and by Table \ref{F:Table},
$$
  \max\left( \abs{\hypwt_0}, \abs{\hypwt_1}, \abs{\hypwt_2} \right)
  \gtrsim 
  \abs{\xi_0} + \abs{\xi_2} - \abs{\xi_1}
  \sim \min\left( \abs{\xi_0}, \abs{\xi_2}\right) \theta_{02}^2,
$$
but also
$$
  \max\left( \abs{\hypwt_0}, \abs{\hypwt_1}, \abs{\hypwt_2} \right)
  \gtrsim 
  \abs{\xi_0} + \abs{\xi_2} - \abs{\xi_1}
  = \abs{\xi_0} - \bigabs{\abs{\xi_1} - \abs{\xi_2}}
  \sim \frac{\abs{\xi_1}\abs{\xi_2}\theta_{12}^2}{\abs{\xi_0}},
$$
where we used \eqref{B:6}. Combining these estimates, and noting that $\theta_{01} \le \theta_{02}$ by \eqref{F:10}, we get the desired estimate. The case $\abs{\xi_1} \le \abs{\xi_2}$ is treated similarly.
\end{proof}

\section{Proof of the dyadic quadrilinear estimate, Part I}\label{G}

Here we prove Theorem \ref{C:Thm2} under the assumption
\begin{equation}\label{G:2}
  L_0 \sim L_0'.
\end{equation}
By symmetry, we may assume
\begin{equation}\label{G:6}
  \theta_{12} \le \theta_{34},
\end{equation}
and
\begin{equation}\label{G:6:2}
  L_1 \le L_2,
  \qquad
  L_3 \le L_4,
\end{equation}
We distinguish the cases
\begin{subequations}\label{G:4}
\begin{alignat}{2}  
  \label{G:4a}
  &L_2 \le L_0',&&
  \\
  \label{G:4b}
  &&
  &L_4 \le L_0,
  \\
  \label{G:4c}
  &L_2 > L_0',&
  \qquad 
  &L_4 > L_0,
\end{alignat}
\end{subequations}
each of which may be split further into the subcases \eqref{D:98} and \eqref{D:99} (recall that \eqref{D:97} has been completely dealt with). In each case it is of course understood that the region of integration in \eqref{C:58} is restricted accordingly.

In view of \eqref{G:6}, the cases \eqref{D:98}, \eqref{D:99} simplify to
\begin{subequations}\label{G:8}
\begin{alignat}{2}
  \label{G:8a}
  &\theta_{12} \ll \phi \lesssim \theta_{34},&
  \qquad &\bigl( \implies \abs{q_{1234}} \lesssim \phi\theta_{34} \bigr)
  \\
  \label{G:8b}
  &\theta_{12} \le \theta_{34} \ll \phi,&
  \qquad &\bigl( \implies \abs{q_{1234}} \lesssim \phi^2 \bigr),
\end{alignat}
\end{subequations}
where the symbol estimates on the right hold by Lemma \ref{D:Lemma2}. By Lemma \ref{D:Lemma1},
\begin{equation}\label{D:134:10}
  \theta_{12} \lesssim \gamma \equiv \min \left( \gamma^*, \biggl(\frac{N_0\Lmax^{0'12}}{N_1N_2}\biggr)^{1/2} \right),
  \qquad \text{for some $0 < \gamma^* \ll 1$}.
\end{equation}
In fact, here we can choose any $0 < \gamma^* \ll 1$ that we want, by adjusting the implicit constant in \eqref{D:97}. By Lemma \ref{D:Lemma1} we also have
\begin{equation}\label{E:4}
  \theta_{34} \lesssim \gamma' \equiv \left( \frac{\Lmax^{034}}{\Nmin^{34}} \right)^{1/2}.
\end{equation}

Observe that, with notation as in \eqref{C:40} and \eqref{D:70},
\begin{equation}\label{E:2}
  \phi \le \min(\theta_{01},\theta_{02}) + \min(\theta_{03},\theta_{04}),
\end{equation}
since $\theta_{jk} \le \theta_{0j} + \theta_{0k}$. By Lemma \ref{F:Lemma4},
\begin{equation}\label{E:2:2}
  \min(\theta_{01},\theta_{02})
  \lesssim \biggl(\frac{\Lmax^{0'12}}{N_0} \biggr)^{1/2},
  \qquad
  \min(\theta_{03},\theta_{04})
  \lesssim \biggl(\frac{\Lmax^{034}}{N_0} \biggr)^{1/2}.
\end{equation}

To simplify, we introduce the shorthand
\begin{equation}
  \label{E:3:1}
  u_{0'12}
  = \Proj_{K^{\pm_0}_{N_0,L_0'}}\left( u_1 \overline{u_2} \right),
  \qquad
  u_{043}
  = \Proj_{K^{\pm_0}_{N_0,L_0}}\left( u_4 \overline{u_3} \right),
\end{equation}
with notation as in \eqref{B:200}. We define $\pm_{12}$ and $\pm_{43}$ as in Definition \ref{F:Def}, recalling that $\xi_0 = \xi_1-\xi_2 = \xi_4-\xi_3$ by \eqref{C:35}. Note the following important relations between the angles $\theta_{12}, \theta_{01}, \theta_{02}$ in the low and high output interactions (recall \eqref{C:41}):
\begin{alignat}{2}
  \label{E:3:2}
  &\pm_0=\pm_{12}, \; \theta_{12} \ll 1, \; N_0 \ll N_1 \sim N_2&
  \quad \implies \quad
  &\theta_{01} \sim \theta_{02} \sim \frac{N_1}{N_0} \theta_{12},
  \\
  \label{E:3:3}
  &\pm_0=\pm_{12}, \; \theta_{12} \ll 1, \; N_1 \lesssim N_0 \sim N_2&
  \quad \implies \quad
  &\theta_{12} \sim \theta_{01} \sim \frac{N_0}{N_1} \theta_{02}.
\end{alignat}
This follows by Lemma \ref{F:Lemma1}, \eqref{F:10} in the proof of that lemma, and Lemma \ref{F:Lemma3}. Note also that \eqref{E:3:2} can only happen if $\pm_1=\pm_2$. Of course, \eqref{E:3:3} applies symmetrically when $N_2 \lesssim N_0 \sim N_1$. Analogous estimates apply to the index 043.

\subsection{Case \eqref{G:4a}}

Then we can treat the cases \eqref{G:8a} and \eqref{G:8b} simultaneously by using Lemma \ref{D:Lemma6} and pairing up $u_1$ with $u_3$, and $u_2$ with $u_4$. After an angular decomposition we then apply the null form estimate with tube restriction.

Discard the characteristic functions of $K^{\pm_0}_{N_0,L_0}, K^{\pm_0}_{N_0,L_0'}$ in \eqref{C:58} (this is reasonable in view of \eqref{G:4a} and \eqref{G:2}), and estimate $\abs{q_{1234}}$ by Lemma \ref{D:Lemma6}. To achieve the desired pairing of the $u$'s, we change the variables $(\tau_0',\tau_0,\xi_0)$ in \eqref{C:58} to
$$
  \tilde\tau_0' = \tau_1+\tau_3,
  \qquad
  \tilde\tau_0 = \tau_2+\tau_4,
  \qquad
  \tilde\xi_0 = \xi_1+\xi_3 = \xi_2+\xi_4.
$$
By \eqref{C:35} and \eqref{C:49}, $\tilde\tau_0' - \tilde\tau_0 = \tau_0' - \tau_0$, so the symbol \eqref{C:90:2} is invariant under the change of variables:
\begin{equation}\label{G:5}
  a_{L_0,L_0'}^{\pm_0}(\tau_0,\tau_0',\xi_0) =
  a_{L_0,L_0'}^{\pm_0}(\tilde\tau_0,\tilde\tau_0',\tilde\xi_0).
\end{equation}
Note that this relies on the assumption $L_0 \sim L_0'$. We conclude that
\begin{equation}\label{D:134}
\begin{aligned}
  J_{\boldN,\boldL}^{\mathbf\Sigma}
  &\lesssim
  \int
  T_{L_0,L_0'} \mathcal F\mathfrak B'_{\theta_{13}}(u_1,u_3)(\tilde X_0)
  \cdot \mathcal F \mathfrak B'_{\theta_{24}}(u_2,u_4)(\tilde X_0)
  \, d\tilde X_0
  \\
  &\lesssim
  \norm{\mathfrak B'_{\theta_{13}}(u_1,u_3)}
  \norm{\mathfrak B'_{\theta_{24}}(u_2,u_4)},
\end{aligned}
\end{equation}
where we used Lemma \ref{D:Lemma3}.

The question is now how to estimate the last two factors in \eqref{D:134}. Observe that there is no point in applying Lemma \ref{D:Lemma1} with respect to the angles $\theta_{13}, \theta_{24}$, since we know nothing about the associated $\hypwt_0$-weights. In other words, Theorem \ref{D:Thm} is not useful here. Instead, we use the upper bound \eqref{D:134:10} for $\theta_{12}$. Although this angle does not explicitly appear in \eqref{D:134}, it can still make its influence felt if we use Lemma \ref{B:Lemma4} to angularly decompose $(\pm_1\xi_1,\pm_2\xi_2)$ before we use \eqref{D:134}. The angular decomposition restricts the spatial Fourier supports to tubes, so we can apply Theorem \ref{G:Thm}. Let us turn to the details.

Making use of the restriction \eqref{D:134:10}, we apply Lemma \ref{B:Lemma5} to the pair $(\pm_1\xi_1,\pm_2\xi_2)$, and then we use \eqref{D:134}, thus obtaining
\begin{equation}\label{D:134:2}
  J_{\boldN,\boldL}^{\mathbf\Sigma}
  \lesssim
  \sum_{\omega_1,\omega_2}
  \norm{\mathfrak B'_{\theta_{13}}(u_1^{\gamma,\omega_1},u_3)}
  \norm{ \mathfrak B'_{\theta_{24}}(u_2^{\gamma,\omega_2},u_4)},
\end{equation}
where the sum is over $\omega_1,\omega_2 \in \Omega(\gamma)$ satisfying $\theta(\omega_1,\omega_2) \lesssim \gamma$, and we use the notation \eqref{B:202} (indexed by $1,2$, except for $\gamma$). Since the spatial frequency $\xi_j$ of $u_j^{\gamma,\omega_j}$ is restricted to a tube of radius comparable to $N_j\gamma$ about $\R\omega_j$, we can apply Theorem \ref{G:Thm} to \eqref{D:134:2}. The result is
\begin{equation}\label{D:134:3}
\begin{aligned}
  J_{\boldN,\boldL}^{\mathbf\Sigma}
  &\lesssim
  N_1 N_2 \gamma^2
  \left( L_1 L_2 L_3 L_4 \right)^{1/2}
  \sum_{\omega_1,\omega_2}\norm{u_1^{\gamma,\omega_1}}\norm{u_2^{\gamma,\omega_2}}\norm{u_3}\norm{u_4}
  \\
  &\lesssim
  N_0 L_0'
  \left( L_1 L_2 L_3 L_4 \right)^{1/2}
  \norm{u_1}\norm{u_2}\norm{u_3}\norm{u_4},
\end{aligned}
\end{equation}
where we summed $\omega_1,\omega_2$ as in \eqref{B:208}, and used the definition\eqref{D:134:10} of $\gamma$, taking into account the assumption \eqref{G:4a}. In view of \eqref{G:2}, this proves \eqref{C:80}.

This concludes case \eqref{G:4a}.

\subsection{Case \eqref{G:4b}}

If $\theta_{34} \ll 1$, then we have the analogue of \eqref{D:134:10}, so by symmetry the argument used for case \eqref{G:4a} above applies, with the roles of the indices 12 and 34 reversed.

It then remains to consider $\theta_{34} \sim 1$. Then $\Nmin^{34} \lesssim L_0$, by Lemma \ref{D:Lemma1}. Moreover, we may assume $L_2 > L_0'$, since the case $L_2 \le L_0'$ was completely dealt with above. Now trivially estimate $\abs{q_{1234}} \lesssim 1$. Then with notation as in \eqref{E:3:1},
\begin{align*}
  J_{\boldN,\boldL}^{\mathbf\Sigma}
  &\lesssim
  \norm{u_{0'12}}
  \norm{u_{043}}
  \\
  &\lesssim
   \left(N_0^2L_0'L_1\right)^{1/2}\left((\Nmin^{34})^2L_3L_4\right)^{1/2}
  \norm{u_1} \norm{u_2} \norm{u_3} \norm{u_4}
  \\
  &\lesssim
   \left(N_0^2L_1L_2\right)^{1/2}\left(L_0^2L_3L_4\right)^{1/2}
  \norm{u_1} \norm{u_2} \norm{u_3} \norm{u_4},
\end{align*}
where the first inequality holds by the Cauchy-Schwarz inequality and Lemma \ref{D:Lemma3}, the second by Theorem \ref{M:Thm}, and the third by the assumption $L_2 > L_0'$ and the fact that $\Nmin^{34} \lesssim L_0$. In view of \eqref{G:2}, this implies \eqref{C:80}.

This concludes the case \eqref{G:4b}.

\subsection{Case \eqref{G:4c}}

So far we were able to treat \eqref{G:8a} and \eqref{G:8b} simultaneously, but from now on we need to separate the two, and we further divide into subcases depending on which term dominates in the right hand side of \eqref{E:2}:
\begin{subequations}\label{G:10}
\begin{alignat}{2}
  \label{G:10a}
  &\theta_{12} \ll \phi \lesssim \theta_{34},&
  \qquad
  &\min(\theta_{01},\theta_{02}) \ge \min(\theta_{03},\theta_{04}),
  \\
  \label{G:10b}
  &\theta_{12} \ll \phi \lesssim \theta_{34},&
  \qquad
  &\min(\theta_{01},\theta_{02}) < \min(\theta_{03},\theta_{04}),
  \\
  \label{G:10c}
  &\theta_{12} \le \theta_{34} \ll \phi,&
  \qquad
  &\min(\theta_{01},\theta_{02}) < \min(\theta_{03},\theta_{04}),
  \\
  \label{G:10d}
  &\theta_{12} \le \theta_{34} \ll \phi,&
  \qquad
  &\min(\theta_{01},\theta_{02}) \ge \min(\theta_{03},\theta_{04}).
\end{alignat}
\end{subequations}
Subcase \eqref{G:10b} is by far the most difficult, and will be split further into  subcases.

\subsection{Case \eqref{G:4c}, subcase \eqref{G:10a}}\label{G:12:0}

Then by \eqref{G:8a} and \eqref{E:4}--\eqref{E:2:2},
\begin{equation}\label{G:12}
  \abs{q_{1234}}
  \lesssim \phi\theta_{34}
  \lesssim
  \biggl( \frac{L_2}{N_0} \biggr)^{1/2}
  \left( \frac{L_4}{\Nmin^{34}} \right)^{1/2},
\end{equation}
hence
\begin{equation}\label{G:12:2}
\begin{aligned}
  J_{\boldN,\boldL}^{\mathbf\Sigma}
  &\lesssim
  \biggl( \frac{L_2}{N_0} \biggr)^{1/2}
  \left( \frac{L_4}{\Nmin^{34}} \right)^{1/2} 
  \norm{u_{0'12}} \bignorm{u_{043}}
  \\
  &\lesssim
  \biggl( \frac{L_2}{N_0}
  \cdot \frac{L_4}{\Nmin^{34}} 
  \cdot N_0^2 L_0' L_1
  \cdot N_0 \Nmin^{034} L_0 L_3 \biggr)^{1/2}
  \prod_{j=1}^4 \norm{u_j},
\end{aligned}
\end{equation}
where we used the Cauchy-Schwarz inequality, Lemma \ref{D:Lemma3} and Theorem \ref{M:Thm}.

\subsection{Case \eqref{G:4c}, subcase \eqref{G:10b}}\label{G:14:0}

Then by \eqref{G:8a} and \eqref{E:4},
\begin{equation}\label{G:14:2}
  \abs{q_{1234}} \lesssim \phi\theta_{34}
  \lesssim \min(\theta_{03},\theta_{04}) \left( \frac{L_4}{\Nmin^{34}} \right)^{1/2},
\end{equation}
so by \eqref{D:134:10} and Lemma \ref{B:Lemma5}, applied to the pair $(\pm_1\xi_1,\pm_2\xi_2)$,
\begin{equation}\label{G:14:4}
  J_{\boldN,\boldL}^{\mathbf\Sigma}
  \lesssim
  \sum_{\omega_1,\omega_2} \left( \frac{L_4}{\Nmin^{34}} \right)^{1/2}
  \norm{\mathfrak B_{\theta_{03}}' \left( \Proj_{K^{\pm_0}_{N_0,L_0}} \mathcal F^{-1} T_{L_0,L_0'}^{\pm_0} \mathcal F u_{0'12}^{\gamma,\omega_1,\omega_2}, u_3 \right)} \norm{u_4},
\end{equation}
where
\begin{equation}\label{G:14:6}
  u^{\gamma,\omega_1,\omega_2}_{0'12}
  =
  \Proj_{K^{\pm_0}_{N_0,L_0'}} 
  \left( u_1^{\gamma,\omega_1}
  \overline{ u_2^{\gamma,\omega_2} } \right)
\end{equation}
and the sum is over $\omega_1,\omega_2 \in \Omega(\gamma)$ with $\theta(\omega_1,\omega_2) \lesssim \gamma$. Since the spatial Fourier support of $u_j^{\gamma,\omega_j}$ is contained in a tube of radius comparable to $N_j\gamma$ about $\R\omega_j$, it follows that the spatial Fourier support of $u^{\gamma,\omega_1,\omega_2}_{0'12}$ is contained in a tube of radius comparable to $\Nmax^{12}\gamma$ around $\R\omega_1$. Therefore, by Theorem \ref{G:Thm}, Lemma \ref{D:Lemma3} and Theorem \ref{M:Thm},
\begin{equation}\label{G:14:8}
\begin{aligned}
  J_{\boldN,\boldL}^{\mathbf\Sigma}
  &\lesssim
  \sum_{\omega_1,\omega_2} \left( \frac{L_4}{\Nmin^{34}} \right)^{1/2}
  \Nmax^{12}\gamma \left(L_0L_3\right)^{1/2}
  \norm{u_{0'12}^{\gamma,\omega_1,\omega_2}}
  \norm{u_3} \norm{u_4}
  \\
  &\lesssim
  \left( \frac{L_4}{\Nmin^{34}} \right)^{1/2}
  \Nmax^{12} \left( \frac{N_0L_2}{N_1N_2} \right)^{1/2} \left(L_0L_3
  \cdot N_0 \Nmin^{012} L_0'L_1 \right)^{1/2} 
  \prod_{j=1}^4 \norm{u_j}
  \\
  &\lesssim
  \left( \frac{N_0}{\Nmin^{34}} N_0^2 L_0L_0'L_1L_2L_3L_4  \right)^{1/2} 
  \prod_{j=1}^4 \norm{u_j}.
\end{aligned}
\end{equation}
Here we summed $\omega_1,\omega_2$ as in \eqref{B:208}, in the second step we used the definition of $\gamma$ from \eqref{D:134:10}, recalling the assumption \eqref{G:4c}, and in the last step we used \eqref{C:43}.

From \eqref{G:14:8} we get the desired estimate if $N_0 \lesssim \Nmin^{34}$, but also whenever we are able to gain an extra factor $(\Nmin^{34}/N_0)^{1/2}$. In particular, this happens if $\pm_0 \neq \pm_{43}$, since then $N_0 \lesssim \Lmax^{034}$ by Lemma \ref{F:Lemma2}, so instead of \eqref{E:4} we can use the estimate $\theta_{34} \lesssim 1 \lesssim (\Lmax^{034}/N_0)^{1/2}$, thereby gaining  the desired factor.

Thus, we may assume $\pm_0 = \pm_{43}$. Moreover, we can assume $\pm_0=\pm_{12}$, since if this fails to hold, then Lemma \ref{F:Lemma2} implies $N_0 \lesssim L_2$, hence the argument in section \ref{G:12:0} applies.

It remains to consider $N_3 \ll N_0 \sim N_4$ and $N_4 \ll N_0 \sim N_3$, but the case $N_3 \ll N_0 \sim N_4$ is easy: By \eqref{E:3:3} and \eqref{G:14:2},
\begin{equation}\label{G:14:10}
  N_3 \ll N_0 \sim N_4 \implies
  \theta_{04} \lesssim \frac{N_3}{N_0} \theta_{34}, \quad \theta_{03} \sim \theta_{34},
  \quad
  \abs{q_{1234}} \lesssim \frac{N_3}{N_0} \theta_{03} \theta_{34},
\end{equation}
hence we gain a factor $N_3/N_0$ in \eqref{G:14:8}, which is more than enough.

That leaves the interaction $N_4 \ll N_0 \sim N_3$, which is much harder; we split it further into $N_0 \lesssim N_2$ and $N_2 \ll N_0$, treated in the next two sections. Recall also that $\pm_0=\pm_{12}=\pm_{34}$ by the above reductions, and that we are assuming \eqref{G:6:2}, by symmetry.

\subsection{Case \eqref{G:4c}, subcase \eqref{G:10b}, $\pm_0=\pm_{12}=\pm_{43}$, $N_4 \ll N_0 \sim N_3$, $N_0 \lesssim N_2$}\label{G:20:0}

We can insert $\Proj_{\abs{\xi_4} \lesssim N_4}$ in front of $\mathfrak B_{\theta_{03}}'$ in \eqref{G:14:8}, and instead of Theorem \ref{G:Thm} we then apply Theorem \ref{J:Thm1}. Let us check that the hypotheses of the latter are satisfied. We have $N_4 \ll N_0 \sim N_3$, hence $N_0,N_3 \gg 1$ and $\theta_{03} \ll 1$ (by the analogue of \eqref{G:14:10}), so $\mathfrak B_{\theta_{03}}$ can be replaced by $\mathfrak B_{\theta_{03} \ll 1}$ in \eqref{G:14:4}. The hypothesis \eqref{J:0:2} translates to, in the present situation,
\begin{equation}\label{G:20:2}
  \Nmax^{12} \gamma \ll N_0,
\end{equation}
where $\gamma$ is given by \eqref{D:134:10}.

But if \eqref{G:20:2} fails, then $N_0 \ll N_1 \sim N_2$, so by \eqref{D:134:10} we see that the failure of \eqref{G:20:2} is equivalent to $N_0 \lesssim L_2$, hence the argument in section \ref{G:12:0} applies.

Thus, we can assume \eqref{G:20:2}, hence Theorem \ref{J:Thm1} applies, so in \eqref{G:14:8} we can replace $\norm{u_{0'12}^{\gamma,\omega_1,\omega_2}}$ by
$$
  \sup_{I} \norm{\Proj_{\xi_0 \cdot \omega_1 \in I} u_{0'12}^{\gamma,\omega_1,\omega_2}},
$$
where the supremum is over all intervals $I \subset \R$ with length $\abs{I} = N_4$. But since $\gamma \ll 1$, Theorem \ref{J:Thm3} implies, via duality,
\begin{equation}\label{G:20:6}
  \sup_I \norm{\Proj_{\xi_0 \cdot \omega_1 \in I}
  u_{0'12}^{\gamma,\omega_1,\omega_2}}
  \lesssim
  \bigl(N_4 \Nmin^{01} L_0'L_1 \bigr)^{1/2}
  \norm{u_1^{\gamma,\omega_1}} \norm{u_2^{\gamma,\omega_2}}.
\end{equation}
Thus, in the second line of \eqref{G:14:8}, the combination $N_0\Nmin^{012}L_0'L_1$ can be replaced by $N_4\Nmin^{01}L_0'L_1$, and since we are assuming $N_0 \lesssim N_2$, this means that we gain the desired factor $(N_4/N_0)^{1/2}$.

\subsection{Case \eqref{G:4c}, subcase \eqref{G:10b}, $\pm_0=\pm_{12}=\pm_{43}$, $N_4 \ll N_0 \sim N_3$, $N_2 \ll N_0$}\label{G:20:7}

To make the argument from the previous section work, we need to somehow gain a factor $(N_2/N_0)^{1/2}$. We shall apply an argument which essentially is the same as the one introduced in \cite{Selberg:2008a} to prove the result stated here as Theorem \ref{J:Thm2}.

Let us first dispose of the easy case $N_2 \sim 1$. Then we simply estimate
\begin{equation}\label{G:20:8:2}
  J_{\boldN,\boldL}^{\mathbf\Sigma}
  \lesssim
  \norm{u_{0'12}} \bignorm{u_{043}}
  \lesssim
  \left( N_2^2 L_1 L_2
  \cdot N_0^2 L_0 L_3 \right)^{1/2}
  \prod_{j=1}^4 \norm{u_j},
\end{equation}
where we used Lemma \ref{D:Lemma3} and Theorem \ref{M:Thm}.

For the rest of this section we therefore assume $1 \ll N_2 \ll N_0 \sim N_1$. This ensures that the region described by $\angles{\xi_2} \sim N_2$ does not degenerate to a ball.

By \eqref{E:3:3} and \eqref{D:134:10},
\begin{equation}\label{G:20:8}
  \theta_{12} \sim \theta_{02} \sim \frac{N_0}{N_2} \theta_{01},
  \qquad \text{hence}
  \qquad
  \theta_{01} \lesssim \alpha \equiv \frac{N_2}{N_0} \gamma.
\end{equation}
Now modify \eqref{G:14:4} by applying Lemma \ref{B:Lemma5} again, this time to $(\pm_0\xi_0,\pm_1\xi_1)$:
\begin{multline}\label{G:20:10}
  J_{\boldN,\boldL}^{\mathbf\Sigma}
  \lesssim
  \sum_{\omega_1,\omega_2} \sum_{\omega_0',\omega_1'} \left( \frac{L_4}{N_4} \right)^{1/2}
  \\
  \times
  \norm{\Proj_{\abs{\xi_4} \lesssim N_4}
  \mathfrak B_{\theta_{03} \ll 1}' \left( \Proj_{K^{\pm_0}_{N_0,L_0}} \mathcal F^{-1} T_{L_0,L_0'}^{\pm_0} \mathcal F u_{0'12}^{\gamma,\omega_1,\omega_2;\alpha,\omega_0',\omega_1'}, u_3 \right)} \norm{u_4},
\end{multline}
where the second sum is over $\omega_0',\omega_1' \in \Omega(\alpha)$ satisfying $\theta(\omega_0',\omega_1') \lesssim \alpha$, and
\begin{align}
  \label{G:20:12}
  u_{0'12}^{\gamma,\omega_1,\omega_2;\alpha,\omega_0',\omega_1'}
  &= \Proj_{\pm_0\xi_0 \in \Gamma_{\alpha}(\omega_0')} 
  \Proj_{K^{\pm_0}_{N_0,L_0'}} \left( u_1^{\gamma,\omega_1;\alpha,\omega_1'}
  u_2^{\gamma,\omega_2} \right)
  \\
  \label{G:20:14}
  u_1^{\gamma,\omega_1;\alpha,\omega_1'}
  &=
  \Proj_{\pm_1\xi_1 \in \Gamma_{\alpha}(\omega_1')} u_1^{\gamma,\omega_1}.
\end{align}
Observe that the spatial Fourier support of \eqref{G:20:12} is contained in a tube of radius comparable to $N_0\alpha \sim N_2\gamma$ around $\R\omega_0'$; this tube is much thinner than the one for \eqref{G:14:6}, which is of radius comparable to $N_1\gamma$ around $\R\omega_1$, hence we gain a factor $N_2/N_0$ when we apply Theorem \ref{J:Thm1}, compared to our estimates in section \ref{G:20:0}. On the other hand, we now have the additional sum over $\omega_0',\omega_1'$. To come out on top, we have to make sure that this sum does not cost us more than a factor $(N_0/N_2)^{1/2}$; this requires some orthogonality, which is supplied by the following crucial information, to be fed into Lemma \ref{J:Lemma}.

As in \eqref{C:35:1}, let $(X_0',X_1,X_2)$ denote the bilinear interaction for \eqref{G:20:12}, so that $X_0'=X_1-X_2$. By \eqref{B:112},
$$
  X_0' \in H_{\max(L_0',N_0\alpha^2)}(\omega_1'),
  \qquad
  X_1 \in H_{\max(L_1,N_0\alpha^2)}(\omega_1'),
$$
where the latter relies on the assumption $\theta(\omega_0',\omega_1') \lesssim \alpha$. Therefore,
\begin{equation}\label{G:20:16}
  X_2 = X_1-X_0' \in H_d(\omega_1'),
  \qquad \text{where}
  \qquad
  d = \max\bigl(\Lmax^{0'1},N_0\alpha^2\bigr),
\end{equation}
so we can insert $\Proj_{H_d(\omega_1')}$ in front of $u_2^{\gamma,\omega_2}$ in \eqref{G:20:12}.

With this information in hand, we estimate \eqref{G:20:10}. Apply Theorem \ref{J:Thm1}, recalling the crucial fact that the tube radius is now $N_2\gamma$. Repeating also the argument from section \ref{G:20:0}, but now with $u_2^{\gamma,\omega_2}$ replaced by $\Proj_{H_d(\omega_1')}u_2^{\gamma,\omega_2}$, we get
\begin{equation}\label{G:20:18}
\begin{aligned}
  J_{\boldN,\boldL}^{\mathbf\Sigma}
  &\lesssim
  \left( \frac{L_4}{N_4} \right)^{1/2}
  N_2\gamma \left(L_0L_3\right)^{1/2}
  \bigl(N_4 N_0 L_0'L_1 \bigr)^{1/2}
  \\
  &\qquad
  \times
  \sum_{\omega_1,\omega_2} \sum_{\omega_0',\omega_1'}
  \bignorm{u_1^{\gamma,\omega_1;\alpha,\omega_1'}}
  \norm{\Proj_{H_d(\omega_1')} u_2^{\gamma,\omega_2}}
  \bignorm{u_3} \bignorm{u_4}
  \\
  &\lesssim
  \left( \frac{L_4}{N_4} \right)^{1/2}
  N_2 \left( \frac{L_2}{N_2} \right)^{1/2} \left(L_0L_3\right)^{1/2}
  \bigl(N_4 N_0 L_0'L_1 \bigr)^{1/2}
  \\
  &\qquad
  \times
  \sum_{\omega_1,\omega_2}
  \sqrt{B(\omega_1)}
  \bignorm{u_1^{\gamma,\omega_1}}
  \norm{u_2^{\gamma,\omega_2}}
  \bignorm{u_3} \bignorm{u_4},
\end{aligned}
\end{equation}
where
\begin{equation}\label{G:20:20}
  B(\omega_1)
  =
  \sup_{(\tau,\xi), \; \abs{\xi} \sim N_2}
  \sum_{\genfrac{}{}{0pt}{1}{\omega_1' \in \Omega(\alpha)}{\theta(\omega_1',\omega_1) \lesssim \gamma}} \chi_{H_d(\omega_1')}(\tau,\xi)
\end{equation}
The second inequality in \eqref{G:20:18} was obtained by the Cauchy-Schwarz inequality, and we used also Lemma \eqref{B:Lemma3} and \eqref{B:204}.

If we can prove that
\begin{equation}\label{G:20:22}
  \sup_{\omega \in \mathbb S^2} B(\omega)
  \lesssim
  \frac{N_0}{N_2},
\end{equation}
then summing $\omega_1,\omega_2$ as in \eqref{B:208} we get the desired estimate \eqref{C:80} from \eqref{G:20:18}.

By Lemma \ref{J:Lemma},
$$
  \sup_{\omega \in \mathbb S^2} B(\omega)
  \lesssim
  \frac{\gamma}{\alpha} + \frac{d}{N_2\alpha^2}.
$$
The first term on the right hand side is comparable to $N_0/N_2$, by \eqref{G:20:8}. As for the second term, this is also comparable to $N_0/N_2$ if $d = N_0\alpha^2$. In view of the definition of $d$, in \eqref{G:20:16}, it then remains to consider $d=\Lmax^{0'1}$, which happens when $N_0\alpha^2 \le \Lmax^{0'1}$. Then instead of \eqref{G:20:22} we only get
\begin{equation}\label{G:20:24}
  \sup_{\omega \in \mathbb S^2} B(\omega)
  \lesssim
  \frac{\Lmax^{0'1}}{N_2\alpha^2},
\end{equation}
but to compensate we can use the following replacement for \eqref{G:20:6}:
\begin{equation}\label{G:20:26}
  \norm{\Proj_{\xi_0 \cdot \omega_1 \in I}
  u_{0'12}^{\gamma,\omega_1,\omega_2}}
  \lesssim
  \bigl(N_4 (N_2\gamma)^2 \Lmin^{0'1} \bigr)^{1/2}
  \norm{u_1^{\gamma,\omega_1}} \norm{u_2^{\gamma,\omega_2}},
\end{equation}
which in fact reduces to a more or less trivial volume estimate; a proof is given in \cite[Sect.\ 8]{Selberg:2008a}. If we now apply \eqref{G:20:26} instead of \eqref{G:20:6}, but of course with $u_2^{\gamma,\omega_2}$ replaced by $\Proj_{H_d(\omega_1')}u_2^{\gamma,\omega_2}$, then we must multiply \eqref{G:20:18} by the square root of
$$
  \frac{(N_2 \gamma)^2}{N_0\Lmax^{0'1}}
$$
But multiplying this by the right hand side of \eqref{G:20:24} gives the factor in the right hand side of \eqref{G:20:22}, so the net effect is the same.

The proof for case \eqref{G:4c}, subcase \eqref{G:10b} is now complete. Note that we did not use the assumption $L_0 \sim L_0'$ for this.

\subsection{Case \eqref{G:4c}, subcase \eqref{G:10c}}

Then by \eqref{G:8b}, \eqref{E:2} and \eqref{E:2:2},
\begin{equation}\label{G:22:2}
  \abs{q_{1234}} \lesssim \phi^2
  \lesssim \min(\theta_{03},\theta_{04}) \left(\frac{L_4}{N_0} \right)^{1/2}.
\end{equation}
Comparing with \eqref{G:14:2}, we then we get \eqref{G:14:8} with an extra factor $(\Nmin^{34}/N_0)^{1/2}$, implying the desired estimate.

\subsection{Case \eqref{G:4c}, subcase \eqref{G:10d}} This follows by the argument from the previous section, by symmetry (reverse the roles of 12 and 34). This works because we know that $\theta_{12},\theta_{34} \ll 1$, hence \eqref{D:134:10} holds, as does its analogue for $\theta_{34}$.

This concludes the proof of Theorem \ref{C:Thm2} for $L_0 \sim L_0'$.

\section{Proof of the dyadic quadrilinear estimate, Part II}\label{E}

It remains to prove Theorem \ref{C:Thm2} when
\begin{equation}\label{G:24}
  L_0 \ll L_0' \qquad \text{or} \qquad L_0 \gg L_0'.
\end{equation}
By symmetry we may assume
\begin{equation}\label{G:24:2}
  L_1 \le L_2,
  \qquad
  L_3 \le L_4.
\end{equation}
Unlike in section \ref{G}, we do not use \eqref{G:4a}--\eqref{G:4c} as our main cases, the reason being that the argument used in cases \eqref{G:4a} and \eqref{G:4b} only works when $L_0 \sim L_0'$, since it relies on \eqref{G:5}, and moreover $L_0,L_0'$ do not appear symmetrically in \eqref{D:134:3}. Instead we use \eqref{D:98} and \eqref{D:99} as our main cases. In the case \eqref{D:98}, we assume without loss of generality that $\theta_{12} \le \theta_{34}$. This we could also assume in case \eqref{D:99}, of course, but it serves no purpose. Instead, in case \eqref{D:99} we use the symmetry to assume without loss of generality that the first term in \eqref{E:2} dominates. We also use the latter to split \eqref{D:98} into two subcases. Thus, it suffices to consider the following cases:
\begin{subequations}\label{E:1}
\begin{alignat}{2}
  \label{E:1a}
  &\theta_{12} \ll \phi \lesssim \theta_{34},&
  \qquad
  &\min(\theta_{01},\theta_{02}) \ge \min(\theta_{03},\theta_{04}),
  \\
  \label{E:1c}
  &\theta_{12} \ll \phi \lesssim \theta_{34},&
  \qquad
  &\min(\theta_{01},\theta_{02}) < \min(\theta_{03},\theta_{04}),
  \\
  \label{E:1e}
  &\theta_{12}, \theta_{34} \ll \phi,&
  \qquad
  &\min(\theta_{01},\theta_{02}) \le \min(\theta_{03},\theta_{04}).
\end{alignat}
\end{subequations}
The latter two we will further split into
\begin{subequations}\label{E:0}
\begin{alignat}{2}  
  \label{E:0a}
  &L_2 > L_0',&
  \qquad 
  &L_4 > L_0,
  \\
  \label{E:0b}
  &L_2 \le L_0',&
  \qquad 
  &L_4 > L_0,
  \\
  \label{E:0c}
  &L_2 \le L_0',&
  \qquad
  &L_4 \le L_0,
  \\
  \label{E:0d}
  &L_2 > L_0',&
  \qquad
  &L_4 \le L_0.
\end{alignat}
\end{subequations}

We may assume that
\begin{equation}\label{G:24:4:2}
  \Nmin^{12}, \Nmin^{34} \gg 1,
\end{equation}
since otherwise trivial estimates analogous to \eqref{G:20:8:2} apply.

\subsection{Case \eqref{E:1a}}

Then by \eqref{E:4}--\eqref{E:2:2},
$$
  \abs{q_{1234}}
  \lesssim \phi\theta_{34}
  \lesssim
  \biggl( \frac{\Lmax^{0'12}}{N_0} \biggr)^{1/2}
  \left( \frac{\Lmax^{034}}{\Nmin^{34}} \right)^{1/2},
$$
so with notation as in \eqref{E:3:1},
\begin{equation}\label{E:6}
  J_{\boldN,\boldL}^{\mathbf\Sigma}
  \lesssim
  \biggl( \frac{\Lmax^{0'12}}{N_0} \biggr)^{1/2}
  \left( \frac{\Lmax^{034}}{\Nmin^{34}} \right)^{1/2}
  \bignorm{T_{L_0,L_0'}^{\pm_0}\mathcal F u_{0'12}} \bignorm{u_{043}}.
\end{equation}
If we apply Lemma \ref{D:Lemma3} followed by Theorem \ref{M:Thm}, we get the desired estimate except in the case $N_0 \ll N_1 \sim N_2$. But we can do better by applying the following:

\begin{lemma}\label{E:Lemma1} Assuming \eqref{G:24} holds, we have
\begin{equation}\label{E:12}
  \bignorm{T_{L_0,L_0'}^{\pm_0}\mathcal F u_{0'12}}
  \lesssim
  \left(N_0 \Nmin^{012} \Lmin^{0'12} \Lmed^{0'12} \right)^{1/2} 
  \norm{u_1} \norm{u_2}.
\end{equation}
\end{lemma}

The main point, comparing with \eqref{M:17} in Theorem \ref{M:Thm}, is that $\Nmin^{12}$ there is replaced by $\Nmin^{012}$ in \eqref{E:12}. Plugging the latter into \eqref{E:6}, and estimating $\norm{u_{043}}$ by the analogue of \eqref{M:17}, we get \eqref{C:80}.

\begin{proof}[Proof of Lemma \ref{E:Lemma1}]
If $\Lmax^{0'12} = \Lmax^{12}$, then \eqref{E:12} holds by Lemma \ref{D:Lemma3} and \eqref{M:14}, so we assume $\Lmax^{0'12} = L_0'$ for the rest of the proof.

Since $\theta_{12} \ll 1$, we have $\theta_{12} \lesssim \gamma$ with $\gamma$ as in \eqref{D:134:10}. Then by Lemma \ref{B:Lemma5}, the left hand side of \eqref{E:12} is dominated by the sum, with notation as in \eqref{G:14:6},
$$
  S = \sum_{\omega_1,\omega_2}
  \norm{T_{L_0,L_0'}^{\pm_0} \mathcal F u^{\gamma,\omega_1,\omega_2}_{0'12}},
$$
where the sum is over $\omega_1,\omega_2 \in \Omega(\gamma)$ with $\theta(\omega_1,\omega_2) \lesssim \gamma$. By \eqref{B:112},
\begin{equation}\label{E:12:2}
  \mathcal F u_1^{\gamma,\omega_1} \subset H_{\max(L_1,N_1\gamma^2)}(\omega_1),
  \qquad
  \mathcal F u_2^{\gamma,\omega_2} \subset H_{\max(L_2,N_2\gamma^2)}(\omega_1),
\end{equation}
where the latter relies on the assumption $\theta(\omega_1,\omega_2) \lesssim \gamma$. We conclude that
\begin{equation}\label{E:12:4}
  \supp \mathcal F u^{\gamma,\omega_1,\omega_2}_{0'12}
  \subset H_{d'}(\omega_1),
  \qquad \text{where}
  \qquad
  d' = \max\left(\Lmax^{12},\Nmax^{12}\gamma^2\right).
\end{equation}
Then by Lemma \ref{D:Lemma5},
\begin{equation}\label{E:12:5}
  S
  \lesssim
  \sum_{\omega_1,\omega_2}
  \left(\frac{d'}{L_0'}\right)^{1/2}
  \norm{u^{\gamma,\omega_1,\omega_2}_{0'12}}.
\end{equation}
If $d' = \Lmax^{12}$, use \eqref{M:14} and sum $\omega_1,\omega_2$ as in \eqref{B:208} to obtain
\begin{equation}\label{E:12:6}
  S
  \lesssim
  \left(\frac{\Lmax^{12}}{L_0'}\right)^{1/2}
  \left( N_0 \Nmin^{012} L_0' \Lmin^{12} \right)^{1/2}
  \norm{u_1} \norm{u_2},
\end{equation}
proving \eqref{E:12}. The other possibility is $d' = \Nmax^{12}\gamma^2 \sim  N_0L_0'/\Nmin^{12}$, where the estimate holds by \eqref{D:134:10}. Then by \eqref{E:12:5} and \eqref{M:10},
\begin{equation}\label{E:12:8}
  S
  \lesssim
  \left(\frac{N_0L_0'}{\Nmin^{12}L_0'}\right)^{1/2}
  \left( \Nmin^{012} \Nmin^{12} L_1 L_2 \right)^{1/2}
  \norm{u_1} \norm{u_2},
\end{equation}
completing the proof of Lemma \ref{E:Lemma1}.
\end{proof}

This concludes case \eqref{E:1a}. 

\subsection{Case \eqref{E:1c}, subcase \eqref{E:0a}}

This is covered by the proof for $L_0 \sim L_0'$, where it corresponds to case \eqref{G:4c}, subcase \eqref{G:10b}; see sections \ref{G:14:0}--\ref{G:20:7}.

\subsection{Case \eqref{E:1c}, subcase \eqref{E:0b}}\label{E:0b:0}

Again we use the argument in sections \ref{G:14:0}--\ref{G:20:7}, but with certain modifications. Observe that \eqref{G:14:2} holds. Now we repeat the argument leading to \eqref{G:14:8}, but instead of using Lemma \ref{D:Lemma3} and Theorem \ref{M:Thm} to estimate $T_{L_0,L_0'}^{\pm_0} \mathcal F u_{0'12}^{\gamma,\omega_1,\omega_2}$, we use Lemma \ref{E:Lemma1}. Then we get
\begin{equation}\label{G:14:8:2}
\begin{aligned}
  J_{\boldN,\boldL}^{\mathbf\Sigma}
  &\lesssim
  \sum_{\omega_1,\omega_2} \left( \frac{L_4}{\Nmin^{34}} \right)^{1/2}
  \Nmax^{12}\gamma \left(L_0L_3\right)^{1/2}
  \bignorm{T_{L_0,L_0'}^{\pm_0} \mathcal F u_{0'12}^{\gamma,\omega_1,\omega_2}}
  \norm{u_3} \norm{u_4}
  \\
  &\lesssim
  \left( \frac{L_4}{\Nmin^{34}} \right)^{1/2}
  \Nmax^{12} \left( \frac{N_0L_0'}{N_1N_2} \right)^{1/2} \left(L_0L_3
  \cdot N_0 \Nmin^{012} L_1L_2 \right)^{1/2} 
  \prod_{j=1}^4 \norm{u_j}
  \\
  &\lesssim
  \left( \frac{N_0}{\Nmin^{34}} N_0^2 L_0L_0'L_1L_2L_3L_4  \right)^{1/2}
  \prod_{j=1}^4 \norm{u_j},
\end{aligned}
\end{equation}
and the last line is exactly as in \eqref{G:14:8}.

Now we continue as in section \ref{G:14:0}. We are done if $N_0 \lesssim \Nmin^{34}$, or whenever we are able to gain an extra factor $(\Nmin^{34}/N_0)^{1/2}$ in \eqref{G:14:8:2}; the latter happens when $N_3 \ll N_0 \sim N_4$, in view of \eqref{G:14:10}.

Thus, we are left with $N_4 \ll N_0 \sim N_3$. Then we proceed as in section \ref{G:20:0}. We may assume \eqref{G:20:2} (otherwise $N_0 \lesssim L_0'$, and then \eqref{E:6} holds), hence Theorem \ref{J:Thm1} applies, so in \eqref{G:14:8:2} we can replace $\bignorm{T_{L_0,L_0'}^{\pm_0} \mathcal F u_{0'12}^{\gamma,\omega_1,\omega_2}}$ by
\begin{equation}\label{G:14:8:4}
  \sup_{I} \norm{T_{L_0,L_0'}^{\pm_0} \mathcal F
  \Proj_{\xi_0 \cdot \omega_1 \in I}  
  u_{0'12}^{\gamma,\omega_1,\omega_2}},
\end{equation}
where the supremum is over $I \subset \R$ with $\abs{I} = N_4$. By Theorem \ref{J:Thm3},
\begin{equation}\label{E:42}
  \sup_{I} \norm{\Proj_{\xi_0 \cdot \omega_1 \in I} u_{0'12}^{\gamma,\omega_1,\omega_2}}
  \lesssim
  \left(N_4 \Nmin^{12} L_1L_2 \right)^{1/2}
  \norm{u_1^{\gamma,\omega_1}} \norm{u_2^{\gamma,\omega_2}}.
\end{equation}
If we combine this with Lemma \ref{D:Lemma3}, we get
\begin{equation}\label{E:42:0}
  \text{l.h.s.\eqref{G:14:8:4}}
  \lesssim
  \left(N_4 \Nmin^{12} L_1L_2 \right)^{1/2}
  \norm{u_1^{\gamma,\omega_1}} \norm{u_2^{\gamma,\omega_2}},
\end{equation}
but this is not enough: It allows us to replace the combination $N_0\Nmin^{012}L_1L_2$ in the second line of \eqref{G:14:8:2} by $N_4\Nmin^{12}L_1L_2$, but what we need is $N_4\Nmin^{012}L_1L_2$. That is, we need
\begin{equation}\label{E:42:1}
  \text{l.h.s.\eqref{G:14:8:4}}
  \lesssim
  \left(N_4 \Nmin^{012} L_1L_2 \right)^{1/2}
  \norm{u_1^{\gamma,\omega_1}} \norm{u_2^{\gamma,\omega_2}}.
\end{equation}
If this holds, then we gain the necessary factor $(N_4/N_0)^{1/2}$ in \eqref{G:14:8:2}.

So let us prove \eqref{E:42:1}. We assume $N_0 \ll N_1 \sim N_2$, since otherwise \eqref{E:42:1} reduces to \eqref{E:42:0}. Recalling \eqref{E:12:4} from the proof of Lemma \ref{E:Lemma1}, we use Lemma \ref{D:Lemma5} to estimate the norm inside the supremum in \eqref{G:14:8:4}, and then we apply either \eqref{E:42} or the variation
\begin{equation}\label{E:42:2}
  \sup_I \norm{\Proj_{\xi_0 \cdot \omega_1 \in I}
  u_{0'12}^{\gamma,\omega_1,\omega_2}}
  \lesssim
  \bigl(N_4 \Nmin^{01} L_0'L_1 \bigr)^{1/2}
  \norm{u_1^{\gamma,\omega_1}} \norm{u_2^{\gamma,\omega_2}},
\end{equation}
which follows from Theorem \ref{J:Thm3} via duality, again using the fact that $\gamma \ll 1$. Specifically, if $d'=\Lmax^{12}$, we use \eqref{E:42:2}, whereas \eqref{E:42} is used if $d'=\Nmax^{12}\gamma^2$; cf.\ \eqref{E:12:6} and \eqref{E:12:8} in the proof of Lemma \ref{E:Lemma1}. Then \eqref{E:42:1} follows.

This concludes case \eqref{E:1c}, subcase \eqref{E:0b}.

\subsection{Case \eqref{E:1c}, subcase \eqref{E:0c}}\label{G:40:0:0}

The argument from section \ref{E:0b:0} does not work, since the $L_4$ in \eqref{G:14:8:2} becomes $L_0$, so there are two $L_0$'s; one of them is due to \eqref{E:4}, and is unavoidable. The other, unwanted factor comes from estimating the null form in \eqref{G:20:10} by Theorem \ref{G:Thm} or Theorem \ref{J:Thm1}, so we have to avoid using those theorems. Instead we shall use an extra decomposition of the spatial frequencies into cubes to gain better control. Then we use Theorems \ref{J:Thm3} and \ref{G:Thm2}, and also Lemma \ref{D:Lemma5}.

For the remainder of section \ref{E}, we change the notation from \eqref{D:134:10}, writing now
\begin{equation}\label{D:134:10:2}
  \theta_{12} \lesssim \gamma \equiv \biggl(\frac{N_0\Lmax^{0'12}}{N_1N_2}\biggr)^{1/2}.
\end{equation}

By \eqref{E:1c}, \eqref{E:2} and \eqref{E:2:2},
\begin{equation}\label{G:40:0:2}
  \phi
  \lesssim \min(\theta_{03},\theta_{04})
  \lesssim
  \left( \frac{L_0}{N_0} \right)^{1/2},
\end{equation}
hence
\begin{equation}\label{G:40:0}
  \abs{q_{1234}} \lesssim \phi\theta_{34}
  \lesssim
  \left( \frac{L_0}{N_0} \right)^{1/2}
  \theta_{34}.
\end{equation}
Therefore, applying Lemma \ref{B:Lemma5} to the pair $(\pm_1\xi_1,\pm_2\xi_2)$ and Lemma \ref{B:Lemma4} to the pair $(\pm_3\xi_3,\pm_4\xi_4)$, and recalling \eqref{D:134:10} and \eqref{E:4},
\begin{multline}\label{G:40:2}
  J_{\boldN,\boldL}^{\mathbf\Sigma}
  \lesssim
  \sum_{\omega_1,\omega_2}
  \sum_{0 < \gamma_{34} \lesssim \gamma'}
  \sum_{\omega_3,\omega_4}
  \left( \frac{L_0}{N_0} \right)^{1/2}
  \gamma_{34}
  \\
  \times
  \int T_{L_0,L_0'}^{\pm_0} \mathcal F u_{0'12}^{\gamma,\omega_1,\omega_2}(X_0)
  \cdot
  \mathcal F u_{043}^{\gamma_{34},\omega_4,\omega_3}(X_0) \, dX_0,
\end{multline}
where $\gamma'$ is defined as in \eqref{E:4}, $u_{0'12}^{\gamma,\omega_1,\omega_2}$ is defined as in \eqref{G:14:6}, and similarly
\begin{equation}\label{G:14:6:2}
  u^{\gamma_{34},\omega_4,\omega_3}_{043}
  =
  \Proj_{K^{\pm_0}_{N_0,L_0}} 
  \left( u_4^{\gamma_{34},\omega_4}
  \overline{ u_3^{\gamma_{34},\omega_3} } \right)
\end{equation}
The sums in \eqref{G:40:2} are over $\omega_1,\omega_2 \in \Omega(\gamma)$ with $\theta(\omega_1,\omega_2) \lesssim \gamma$, over dyadic $\gamma_{34}$, and over $\omega_3,\omega_4 \in \Omega(\gamma_{34})$ satisfying
\begin{equation}\label{G:14:6:4}
  3\gamma_{34} \le \theta(\omega_3,\omega_4) \le 12\gamma_{34}.
\end{equation}
Due to this separation, $\theta_{34} \sim \gamma_{34}$ in the bilinear interaction of \eqref{G:14:6:2}.

Recall that the spatial Fourier support of $u^{\gamma,\omega_1,\omega_2}_{0'12}$ is contained in a tube of radius
$$
  r \sim \Nmax^{12} \gamma
$$
around $\R\omega_1$, where $\gamma$ is given by \eqref{D:134:10:2}. Cover $\R$ by almost disjoint intervals $I$ of length $r$, and write
$$
  u_{0'12}^{\gamma,\omega_1,\omega_2}
  =
  \sum_I
  \Proj_{\xi_0 \cdot \omega_1 \in I} u_{0'12}^{\gamma,\omega_1,\omega_2},
$$
where the sum has cardinality $O(N_0/r)$. The spatial frequency $\xi_0$ of the summand is restricted to a cube
$$
  Q_0=Q_0(I)
$$
of side-length comparable to $r$. Let $\mathcal Q_0$ denote the cover of $\R^3$ by almost disjoint translates of $Q_0$, and restrict the spatial frequencies $\xi_1,\xi_2,\xi_3,\xi_4$ to cubes $Q_1,Q_2,Q_3,Q_4 \in \mathcal Q_0$, respectively. Since $\xi_0=\xi_1-\xi_2$, then once $Q_1$ has been chosen, the choice of $Q_2$ is restricted to the set
\begin{equation}\label{G:40:3}
  \left\{ Q_2 \in \mathcal Q_0 \colon Q_2 \cap (Q_1-Q_0) \neq \emptyset \right\},
\end{equation}
which has cardinality $O(1)$. This implies that, by the Cauchy-Schwarz inequality,
\begin{equation}\label{G:40:4}
  \sum_{Q_1,Q_2}
  \bignorm{\Proj_{Q_1}u_1^{\gamma,\omega_1}}
  \bignorm{\Proj_{Q_2}u_2^{\gamma,\omega_2}}
  \lesssim
  \bignorm{u_1^{\gamma,\omega_1}}
  \bignorm{u_2^{\gamma,\omega_2}},
\end{equation}
where the sum over $Q_2$ of course is restricted to the set \eqref{G:40:3}. The pair $(Q_3,Q_4)$ is similarly restricted, since $\xi_0=\xi_4-\xi_3$.

After this extra decomposition, \eqref{G:40:2} is replaced by
\begin{multline}\label{G:40:12}
  J_{\boldN,\boldL}^{\mathbf\Sigma}
  \lesssim
  \sum_{\omega_1,\omega_2}
  \sum_{0 < \gamma_{34} \lesssim \gamma'}
  \sum_{\omega_3,\omega_4}
  \sum_I
  \sum_{Q_1,Q_2}
  \sum_{Q_3,Q_4}
  \left( \frac{L_0}{N_0} \right)^{1/2} \gamma_{34}
  \\
  \times
  \int T_{L_0,L_0'}^{\pm_0} \mathcal F u_{0'12}^{\gamma,\omega_1,\omega_2;Q_1,Q_2}(X_0)
  \cdot
  \mathcal F u_{043}^{\gamma_{34},\omega_4,\omega_3;Q_4,Q_3}(X_0) \, dX_0,
\end{multline}
where
\begin{equation}\label{G:40:13}
\begin{aligned}
  u_{0'12}^{\gamma,\omega_1,\omega_2;Q_1,Q_2}
  &= \Proj_{K^{\pm_0}_{N_0,L_0'}}
  \left( \Proj_{Q_1} u_1^{\gamma,\omega_1}
  \overline{ \Proj_{Q_2} u_2^{\gamma,\omega_2} } \right),
  \\
  u_{043}^{\gamma_{34},\omega_4,\omega_3;Q_4,Q_3}
  &= \Proj_{K^{\pm_0}_{N_0,L_0}}
  \left( \Proj_{Q_4} u_4^{\gamma_{34},\omega_4}
  \overline{ \Proj_{Q_3} u_3^{\gamma_{34},\omega_3} } \right).
\end{aligned}
\end{equation}
The sum over $Q_2$ in \eqref{G:40:12} is restricted to the set \eqref{G:40:3} determined by $Q_1$ and $I$, and similarly for the pair $(Q_3,Q_4)$.

Consider the integral in \eqref{G:40:12} for a fixed choice of $\gamma_{34}$, the $\omega$'s, $I$ and the $Q$'s. We use the notation \eqref{C:35:1}--\eqref{C:35} for the bilinear interactions in this integral, so in particular $\xi_j \in Q_j$ and $\abs{\xi_j} \sim N_j$ for $j=1,2,3,4$, recalling \eqref{G:24:4:2}. Our plan is now to apply Lemma \ref{D:Lemma5}. By \eqref{E:12:4} we have
\begin{equation}\label{G:40:14}
  \tau_0' + \xi_0 \cdot \omega_1 = O(d'),
  \qquad
  \text{where}
  \qquad
  d'=\max\left(L_2,\Nmax^{12}\gamma^2\right).
\end{equation}
We claim that also
\begin{equation}\label{G:40:16}
  \tau_0 + \xi_0 \cdot \omega_3 = c + O(d),
  \qquad
  \text{where}
  \qquad
  d=\max\left(L_4,\frac{r^2}{\Nmin^{34}},r\gamma_{34}\right),
\end{equation}
and $c \in \R$ is a constant depending on $(Q_3,Q_4)$ and $(\omega_3,\omega_4)$.

Let us prove \eqref{G:40:16}. Denote by $\xi_j^*$ the center of the cube $Q_j$, so that
\begin{equation}\label{G:40:17:0}
\abs{\xi_j-\xi_j^*} \lesssim r.
\end{equation}
Let $\omega_j^* = \pm_j\xi_j^*/\abs{\xi_j^*}$. Replacing the side-length $r$ of the cubes by $2r$ if necessary, we may assume that $\omega_j^* \in \Gamma_{\gamma_{34}}(\omega_j)$ for $j=3,4$. Since $\theta(\omega_3,\omega_4) \lesssim \gamma_{34}$, we then have
\begin{equation}\label{G:40:17}
  \theta(\omega_j^*,\omega_3) \lesssim \gamma_{34} \qquad \text{for $j=3,4$}.
\end{equation}
Now write, for $j=3,4$,
\begin{equation}\label{G:40:18}
  \tau_j + \xi_j \cdot \omega_3 =
  \tau_j + \xi_j \cdot \omega_j^*
  + (\xi_j-\xi_j^*) \cdot (\omega_3 - \omega_j^*)
  + c_j,
\end{equation}
where
$$
  c_j = \xi_j^* \cdot (\omega_3 - \omega_j^*).
$$
Since $\xi_j \in Q_j$ and $\abs{\xi_j} \sim N_j$, we conclude from \eqref{B:112} that
\begin{equation}\label{G:40:20}
  \tau_j + \xi_j \cdot \omega_j^* = O\left(\max\left(L_j,\frac{r^2}{N_j}\right)\right)
  \qquad (j=3,4).
\end{equation}
By \eqref{G:40:17:0} and \eqref{G:40:17},
\begin{equation}\label{G:40:22}
  (\xi_j-\xi_j^*) \cdot (\omega_3 - \omega_j^*)
  = O(r\gamma_{34}).
\end{equation}
Now plug the estimates \eqref{G:40:20} and \eqref{G:40:22} into \eqref{G:40:18} for $j=3,4$, and subtract. Since $\tau_0=\tau_4-\tau_3$ and $\xi_0=\xi_4-\xi_3$, this proves \eqref{G:40:16}, with $c=c_4-c_3$.

In view of \eqref{G:40:14}, \eqref{G:40:16} and Lemma \ref{D:Lemma5}, we can dominate the integral in \eqref{G:40:12} by the product of
\begin{equation}\label{G:40:30:2}
  \left(\min\left(1,\frac{d'}{L_0'}\right) \right)^{1/2}
  \norm{u_{0'12}^{\gamma,\omega_1,\omega_2;Q_1,Q_2}}
\end{equation}
and
\begin{equation}\label{G:40:30:4}
  \left(\frac{d}{L_0} \right)^{1/2}
  \norm{u_{043}^{\gamma_{34},\omega_4,\omega_3;Q_4,Q_3}}.
\end{equation}

By Theorem \ref{J:Thm3},
\begin{equation}\label{G:40:32}
  \norm{u_{0'12}^{\gamma,\omega_1,\omega_2;Q_1,Q_2}}
  \lesssim C\,
  \bignorm{\Proj_{Q_1}u_1^{\gamma,\omega_1}}
  \bignorm{\Proj_{Q_2}u_2^{\gamma,\omega_2}}
\end{equation}
holds with
\begin{align}
  \label{G:40:32:4}
  C^2 &\sim r\Nmin^{01}L_0'L_1,
  \\
  \label{G:40:32:2}
  C^2 &\sim r\Nmin^{12}L_1L_2.
\end{align}
From the definitions of $d'$ and $\gamma$ (see \eqref{G:40:14} and \eqref{D:134:10}) and by the assumption \eqref{E:0c},
\begin{equation}\label{G:40:32:6}
  \frac{d'}{L_0'}
  \sim \max\left(\frac{L_2}{L_0'},\frac{N_0}{\Nmin^{12}}\right).
\end{equation}
If $N_0 \sim \Nmax^{12}$, we estimate \eqref{G:40:30:2} by l.h.s.\eqref{G:40:32} and use \eqref{G:40:32:2}. If, on the other hand, $N_0 \ll N_1 \sim N_2$, then we combine \eqref{G:40:32:6} with \eqref{G:40:32}, observing that the product of \eqref{G:40:32:6} with the minimum of \eqref{G:40:32:4} and \eqref{G:40:32:2} is dominated by $rN_0L_1L_2$. We conclude:
\begin{equation}\label{G:40:32:8}
  \eqref{G:40:30:2} \lesssim \left(r\Nmin^{012}L_1L_2\right)^{1/2}
  \bignorm{\Proj_{Q_1}u_1^{\gamma,\omega_1}}
  \bignorm{\Proj_{Q_2}u_2^{\gamma,\omega_2}}.
\end{equation}

We further claim that
\begin{equation}\label{G:40:34}
  \norm{u_{043}^{\gamma_{34},\omega_4,\omega_3;Q_4,Q_3}}
  \lesssim C\,
  \bignorm{\Proj_{Q_3}u_3^{\gamma_{34},\omega_3}}
  \bignorm{\Proj_{Q_4}u_4^{\gamma_{34},\omega_4}}
\end{equation}
holds with
\begin{align}
  \label{G:40:36}
  C^2
  &\sim
  r^3 L_3,
  \\
  \label{G:40:38}
  C^2
  &\sim
  r \Nmin^{34} L_3 L_4,
  \\
  \label{G:40:40}
  C^2
  &\sim
  \frac{r^2 L_3 L_4}{\gamma_{34}}.
\end{align}
In fact, \eqref{G:40:40} holds by Theorem \ref{G:Thm2}, in view of the separation assumption \eqref{G:14:6:4}; \eqref{G:40:38} holds by Theorem \ref{J:Thm3}, and \eqref{G:40:36} reduces to a trivial volume estimate (see the proof of Theorem \ref{G:Thm2} in \cite{Selberg:2008a}).

Now observe that $d$, defined by \eqref{G:40:16}, times the minimum of \eqref{G:40:36}--\eqref{G:40:40}, is comparable to $r^3 L_3 L_4$. Therefore,
\begin{equation}\label{G:40:42}
  \eqref{G:40:30:4}
  \lesssim
  \left(\frac{r^3L_3L_4}{L_0}\right)^{1/2}
  \bignorm{\Proj_{Q_3}u_3^{\gamma_{34},\omega_3}}
  \bignorm{\Proj_{Q_4}u_4^{\gamma_{34},\omega_4}}.
\end{equation}

Now estimate the integral in \eqref{G:40:12} by the product of \eqref{G:40:30:2} and \eqref{G:40:30:4}, and use \eqref{G:40:32:8} and \eqref{G:40:42}. The result is
\begin{equation}\label{G:40:42:2}
\begin{aligned}
  J_{\boldN,\boldL}^{\mathbf\Sigma}
  &\lesssim
  \sum_{\omega_1,\omega_2}
  \sum_{0 < \gamma_{34} \lesssim \gamma'}
  \sum_{\omega_3,\omega_4}
  \sum_I
  \left( \frac{L_0}{N_0} \right)^{1/2}
  \gamma_{34}
  \left(
  \frac{r^3L_3L_4}{L_0}
  \cdot
  r\Nmin^{012}L_1L_2 \right)^{1/2}
  \\
  &\quad
  \times
  \left(\sum_{Q_1,Q_2}
  \bignorm{\Proj_{Q_1}u_1^{\gamma,\omega_1}}
  \bignorm{\Proj_{Q_2}u_2^{\gamma,\omega_2}}\right)
  \left(\sum_{Q_3,Q_4}
  \bignorm{\Proj_{Q_3}u_3^{\gamma_{34},\omega_3}}
  \bignorm{\Proj_{Q_4}u_4^{\gamma_{34},\omega_4}}\right)
  \\
  &\lesssim \left( \sum_I \frac{r}{N_0} \right)
  \sum_{0 < \gamma_{34} \lesssim \gamma'} r\gamma_{34}
  \left( \frac{\Nmin^{012}}{N_0} \right)^{1/2}
  \left( N_0^2 L_1L_2 L_3L_4 \right)^{1/2}
  \\
  &\quad
  \times
  \left(\sum_{\omega_1,\omega_2}
  \bignorm{u_1^{\gamma,\omega_1}}
  \bignorm{u_2^{\gamma,\omega_2}} \right)
  \left( \sum_{\omega_3,\omega_4}
  \bignorm{u_3^{\gamma_{34},\omega_3}}
  \bignorm{u_4^{\gamma_{34},\omega_4}} \right)
  \\
  &\lesssim
  \frac{\Nmax^{12}\gamma\gamma'}{(L_0L_0')^{1/2}}
  \left( \frac{\Nmin^{012}}{N_0} \right)^{1/2}
  \left( N_0^2 L_0L_0'L_1L_2 L_3L_4 \right)^{1/2}
  \norm{u_1}\norm{u_2}\norm{u_3}\norm{u_4},
\end{aligned}
\end{equation}
where to get the second inequality we summed the $Q$'s using \eqref{G:40:4} and its analogue for $(Q_3,Q_4)$. In the final step we used the definition $r \sim \Nmax^{12}\gamma$, we summed $I$ using the fact that the index set has cardinality $O(N_0/r)$, we summed the $\omega$'s as in \eqref{B:208}, and finally we used the fact that
\begin{equation}\label{G:40:50}
  \sum_{0 < \gamma_{34} \lesssim \gamma'} \gamma_{34}
  \sim \gamma',
\end{equation}
where the sum is over dyadic $\gamma_{34}$, of course.

Note that the above implies \eqref{C:80} if the expression
\begin{equation}\label{G:40:52}
  A = \frac{(\Nmax^{12}\gamma\gamma')^2\Nmin^{012}}{N_0L_0L_0'}
\end{equation}
is $O(1$). In view of \eqref{D:134:10:2}, \eqref{E:4} and \eqref{E:0c},
\begin{equation}\label{G:40:54}
  A
  \lesssim
  \frac{(\Nmax^{12})^2\Nmin^{012}}{N_0L_0L_0'}
  \cdot \frac{N_0L_0'}{N_1N_2}
  \min\left(1,\frac{L_0}{\Nmin^{34}}\right).
\end{equation}
If we use the second factor in the last minimum, we get
\begin{equation}\label{G:40:56}
  A 
  \lesssim
  \frac{\Nmax^{12}\Nmin^{012}}{\Nmin^{12}\Nmin^{34}}
  \lesssim \frac{N_0}{\Nmin^{34}},
\end{equation}
where we used \eqref{C:43}. This proves \eqref{G:40:52} except when
$$
  \Nmin^{34} \ll N_0,
$$
which we now assume. If $\pm_0\neq\pm_{43}$, then $N_0 \lesssim L_0$ by Lemma \ref{F:Lemma2}, so we can estimate the first factor in the minimum in \eqref{G:40:54} by $1 \lesssim L_0/N_0$, thereby gaining a factor $\Nmin^{34}/N_0$ compared to \eqref{G:40:56}. If, on the other hand, $\pm_0=\pm_{43}$, then by Lemma \ref{F:Lemma1} and \eqref{E:4},
$$
  \min(\theta_{03},\theta_{04}) \lesssim \frac{\Nmin^{34}}{N_0} \theta_{34}
  \lesssim
  \frac{\Nmin^{34}}{N_0} \left( \frac{L_0}{\Nmin^{34}} \right)^{1/2}
  = \left(\frac{\Nmin^{34}}{N_0} \right) \left( \frac{L_0}{N_0} \right)^{1/2},
$$
which means that compared to \eqref{G:40:0:2} we gain a factor $(\Nmin^{34}/N_0)^{1/2}$, which comes up squared in \eqref{G:40:52}.

This completes case \eqref{E:1c}, subcase \eqref{E:0c}.

\subsection{Case \eqref{E:1c}, subcase \eqref{E:0d}}\label{G:40:51}

This adds another layer of difficulty compared to the previous section, for a certain asymmetric interaction. As far as possible, however, we repeat the preceding argument.

The only difference from the previous case is that now $L_2 > L_0'$, instead of $L_2 \le L_0'$. This difference only shows up in the expression \eqref{D:134:10} for $\gamma$, however, and this expression is not used explicitly in the previous section until the estimate \eqref{G:40:32:6}. But in the present case, $d'/L_0' > 1$, hence \eqref{G:40:30:2} is just equal to the left hand side of \eqref{G:40:32}, so instead of \eqref{G:40:32:8} we use \eqref{G:40:32} with constant $C$ as in \eqref{G:40:32:4} (using \eqref{G:40:32:2} will not work now). Comparing \eqref{G:40:32:4} with \eqref{G:40:32:8}, we see that there will only be a problem if $N_2 \ll N_0 \sim N_1$.

To be precise, instead of \eqref{G:40:52} we will now have
\begin{equation}\label{G:40:52:2}
  A = \frac{(\Nmax^{12}\gamma\gamma')^2\Nmin^{12}}{N_0L_0L_2},
\end{equation}
leading to
\begin{equation}\label{G:40:54:2}
  A
  \lesssim
  \frac{(\Nmax^{12})^2\Nmin^{12}}{N_0L_0L_2}
  \cdot \frac{N_0L_2}{N_1N_2}
  \cdot \min\left(1,\frac{L_0}{\Nmin^{34}}\right)
  = \frac{\Nmin^{12}}{\Nmin^{012}} \times \text{r.h.s.}\eqref{G:40:54},
\end{equation}
so we are done except in the case
$$
  N_2 \ll N_0 \sim N_1,
$$
which we now assume. Then we must somehow gain a factor $N_2/N_0$ in \eqref{G:40:52:2}. We use the same idea as in section \ref{G:20:7}. We may assume
$$
  N_2 \gg 1,
$$
since otherwise \eqref{G:20:8:2} applies. We may further assume
$$
  \pm_0=\pm_{012},
$$
since otherwise \eqref{E:6} applies.

Then \eqref{G:20:8} holds, and we use this to make an extra angular decomposition in the analysis of the previous section, for the pair $(\pm_0\xi_0,\pm_1\xi_1)$. In view of \eqref{G:20:16}, the effect of this extra decomposition is that we can replace $\Proj_{Q_1} u_1^{\gamma,\omega_1}$ and $\Proj_{Q_2} u_2^{\gamma,\omega_2}$ in \eqref{G:40:13} by, respectively,
\begin{equation}\label{G:40:13:2}
  \Proj_{Q_1} u_1^{\gamma,\omega_1;\alpha,\omega_1'},
  \qquad
  \Proj_{H_d(\omega_1')} \Proj_{Q_2} u_2^{\gamma,\omega_2},
\end{equation}
where $u_1^{\gamma,\omega_1;\alpha,\omega_1'}$ is defined by \eqref{G:40:14} and $d$ is given by \eqref{G:20:16}. Here $\omega_1' \in \Omega(\alpha)$. There is also a vector $\omega_0' \in \Omega(\alpha)$, but since $\theta(\omega_0',\omega_1') \lesssim \alpha$ we know that only $O(1)$ $\omega_0'$'s can interact with a given $\omega_1'$, hence summing $\omega_0'$ is not a problem.

A key observation is that the spatial output $\xi_0$ is now restricted to a tube of radius
$$
  r' \sim N_0\alpha \sim N_2\gamma
$$
around $\R\omega_0'$, and the relation between this and the radius $r \sim N_0 \gamma$ used in the previous section is
$$
  \frac{r'}{r} \sim \frac{N_2}{N_0}.
$$

If we now repeat the decomposition into cubes as in the previous section, but now with $r$ replaced by $r'$, then apply \eqref{G:40:32} with $C$ as in \eqref{G:40:32:4} but with $r$ replaced by $r'$ and with the substitutions \eqref{G:40:13:2}, and we apply also \eqref{G:40:34} with $r$ replaced by $r'$, then we get
\begin{equation}\label{G:40:42:4}
\begin{aligned}
  J_{\boldN,\boldL}^{\mathbf\Sigma}
  &\lesssim
  \sum_{\omega_1,\omega_2}
  \sum_{0 < \gamma_{34} \lesssim \gamma'}
  \sum_{\omega_3,\omega_4}
  \sum_{\omega_1'}
  \sum_I
  \left( \frac{L_0}{N_0} \right)^{1/2}
  \gamma_{34}
  \left( \frac{(r')^3L_3L_4}{L_0}
  \cdot
  r'N_0L_0'L_1 \right)^{1/2}
  \\
  &\quad
  \times
  \bignorm{u_1^{\gamma,\omega_1;\alpha,\omega_1'}}
  \bignorm{\Proj_{H_d(\omega_1')}u_2^{\gamma,\omega_2}}
  \bignorm{u_3^{\gamma_{34},\omega_3}}
  \bignorm{u_4^{\gamma_{34},\omega_4}}
  \\
  &\lesssim \left( \sum_I \frac{r'}{N_0} \right)
  \sum_{0 < \gamma_{34} \lesssim \gamma'}
  \frac{r'\gamma_{34}}{(L_0L_2)^{1/2}}
  \left( N_0^2 L_0 L_0' L_1 L_2 L_3 L_4 \right)^{1/2}
  \\
  &\quad
  \times
  \left(\sum_{\omega_1,\omega_2}
  \sum_{\omega_1'}
  \bignorm{u_1^{\gamma,\omega_1;\alpha,\omega_1'}}
  \bignorm{\Proj_{H_d(\omega_1')}u_2^{\gamma,\omega_2}} \right)
  \left( \sum_{\omega_3,\omega_4}
  \bignorm{u_3^{\gamma_{34},\omega_3}}
  \bignorm{u_4^{\gamma_{34},\omega_4}} \right)
  \\
  &\lesssim
  \frac{r'\gamma'}{(L_0L_2)^{1/2}}
  \left( N_0^2 L_0L_0'L_1L_2 L_3L_4 \right)^{1/2}
  \sqrt{\sup_{\omega_1 \in \mathbb S^2} B(\omega_1)} 
  \norm{u_1} \norm{u_2} \norm{u_3} \norm{u_4},
\end{aligned}
\end{equation}
where $B(\omega_1)$ is defined by \eqref{G:20:20}. So now instead of \eqref{G:40:52} we have
$$
  A = \frac{(r'\gamma')^2}{L_0L_2} \sup_{\omega_1 \in \mathbb S^2} B(\omega_1),
$$
and \eqref{G:40:54} is replaced by
\begin{align*}
  A
  &\lesssim
  \frac{N_2^2}{L_0L_2}
  \cdot \frac{N_0L_2}{N_1N_2}
  \min\left(1,\frac{L_0}{\Nmin^{34}}\right) \sup_{\omega_1 \in \mathbb S^2} B(\omega_1)
  \\
  &\lesssim
  \frac{N_2}{L_0}
  \min\left(1,\frac{L_0}{\Nmin^{34}}\right) \sup_{\omega_1 \in \mathbb S^2} B(\omega_1).
\end{align*}

When \eqref{G:20:22} holds we are done, since then we get
\begin{equation}\label{G:40:56:2}
  A \lesssim \frac{N_0}{L_0} \min\left(1,\frac{L_0}{\Nmin^{34}}\right).
\end{equation}
and by the same argument as at the end of the previous section we also know how to deal with the case $\Nmin^{34} \ll N_0$.

If, on the other hand, \eqref{G:20:22} does not hold, then as shown in section \ref{G:20:7} we have instead \eqref{G:20:24}. But to compensate we can use the fact that \eqref{G:40:32} holds with
$$
  C^2 \sim r' (N_2\gamma)^2 \Lmin^{0'1},
$$
as follows from \eqref{G:20:26}. Then as observed in section \ref{G:20:7}, the net effect is the same, hence \eqref{G:40:56:2} holds.

This completes case \eqref{E:1c}, subcase \eqref{E:0d}

\subsection{Case \eqref{E:1e}, subcases \eqref{E:0a} and \eqref{E:0b}}

Then by \eqref{G:8b}, \eqref{E:2} and \eqref{E:2:2},
$$
  \abs{q_{1234}} \lesssim \phi^2
  \lesssim \min(\theta_{03},\theta_{04})^2
  \lesssim \min(\theta_{03},\theta_{04}) \left( \frac{L_4}{N_0} \right)^{1/2},
$$
hence we can proceed as in section \ref{G:14:0}, but recalling also that we have Lemma \ref{E:Lemma1} at our disposal. The result is that we can dominate $J_{\boldN,\boldL}^{\mathbf\Sigma}$ by the last line of \eqref{G:14:8}, but without the factor $N_0/\Nmin^{34}$. Thus, \eqref{C:80} holds.

\subsection{Case \eqref{E:1e}, subcase \eqref{E:0c}}

Here we would like to follow as closely as possible the argument for case \eqref{E:1c}, subcase \eqref{E:0c}, from section \ref{G:40:0:0}.

Since $\theta_{12},\theta_{34} \ll 1$, we have \eqref{D:134:10:2} and similarly
\begin{equation}\label{D:134:10:4}
  \theta_{34} \lesssim \gamma' \equiv \left(\frac{N_0L_0}{N_3N_4}\right)^{1/2},
\end{equation}
which replaces \eqref{E:4}.

We still have \eqref{G:40:0:2}, but  \eqref{G:40:0} is replaced by
\begin{equation}\label{G:50:4}
  \abs{q_{1234}} \lesssim \phi^2,
\end{equation}
hence the factor $\gamma_{34}$ in \eqref{G:40:12} is replaced by the upper bound for $\phi$ in \eqref{G:40:0:2}.

Since there is no $\gamma_{34}$, it may seem that we have a problem with the estimate \eqref{G:40:34} with $C$ as in \eqref{G:40:40}, since this is a null form estimate which requires that we have at least a square root of $\gamma_{34}$ (the dyadic size of $\theta_{34}$).

But the combination \eqref{G:40:34}, \eqref{G:40:40} is only used when we pick up the third factor $r\gamma_{34}$ in the maximum defining $d$ in \eqref{G:40:16}, so we are still able to use \eqref{G:40:40}.

Proceeding as in section \ref{G:40:0:0}, we then get \eqref{G:40:42:2} with the following  modifications: The factors $\gamma_{34}$ in the first and third lines are replaced by the upper bound in \eqref{G:40:0:2}, and instead of the factor $\gamma'$ in the last line, which comes from the sum \eqref{G:40:50}, we now have
\begin{equation}\label{G:50:6}
  \left(\frac{L_0}{N_0}\right)^{1/2} \sum_{0 < \gamma_{34} \lesssim \gamma'} 1.
\end{equation}
Of course the sum diverges, unless we can further restrict the range of the dyadic number $\gamma_{34}$.

The separation assumption \eqref{G:14:6:4} is only needed when we apply the null form estimate \eqref{G:40:40}, i.e., when the factor $r\gamma_{34}$ dominates in the definition of $d$ in \eqref{G:40:16}; then in particular,
\begin{equation}\label{D:134:10:6}
  \gamma_{34} \gtrsim \frac{r}{\Nmin^{34}} \sim \frac{\Nmax^{12}}{\Nmin^{34}}\gamma
  = \frac{\Nmax^{12}}{\Nmin^{34}} \left( \frac{N_0L_0'}{N_1N_2} \right)^{1/2}.
\end{equation}
On the other hand, we also have the upper bound \eqref{D:134:10:4} for $\gamma_{34}$. The cardinality of the set of dyadic numbers $\gamma_{34}$ satisfying both \eqref{D:134:10:4} and \eqref{D:134:10:6} is comparable to
\begin{equation}\label{G:50:10}
  \log
  \left( \frac{\gamma'}{r/\Nmin^{34}} \right)
  \sim
  \log \left( \frac{(N_1N_2)^{1/2}}{\Nmax^{12}}
  \cdot \frac{\Nmin^{34}}{(N_3N_4)^{1/2}} \left(\frac{L_0}{L_0'}\right)^{1/2}
  \right)
  \lesssim \log L_0,
\end{equation}
so the corresponding part of the sum in \eqref{G:50:6} is $O(\log L_0)$.

It then remains to consider
$$
  \theta_{34} \ll \frac{r}{\Nmin^{34}},
$$
but then we do not need the separation assumption \eqref{G:14:6:4}, so here we can avoid a summation over $\gamma_{34}$ altogether by using Lemma \ref{B:Lemma5} instead of Lemma \ref{B:Lemma4}.

\subsection{Case \eqref{E:1e}, subcase \eqref{E:0d}}

This follows by the argument from section \eqref{G:40:51} with the same modifications as in the previous section. Now $L_0'$ in \eqref{G:50:10} is replaced by $L_2$, but this does not change the final estimate in \eqref{G:50:10} (recall that all the $L$'s are greater than or equal to one).

This completes the proof of Theorem \ref{C:Thm2}.

\section{Summation of the dyadic pieces}\label{Y}

By summing the dyadic estimates from Theorem \ref{C:Thm2}, we prove that \eqref{C:12} holds for any $s > 0$ and all sufficiently small $\varepsilon > 0$ (depending on $s$). Split the integral \eqref{C:26} into two parts:
$$
  J^{\boldsymbol\Sigma}
  =
  J^{\boldsymbol\Sigma}_{\abs{\xi_0} \ge 1}
  +
  J^{\boldsymbol\Sigma}_{\abs{\xi_0} < 1},
$$
by restricting to the regions where $\abs{\xi_0} \ge 1$ and $\abs{\xi_0} < 1$, respectively.

\subsection{The high frequency part}

Recall that \eqref{C:56} holds for $J^{\boldsymbol\Sigma}_{\abs{\xi_0} \ge 1}$. Now we combine the estimate from Theorem \ref{C:Thm2} with the trivial estimate
\begin{equation}\label{Y:2}
  J_{\boldN,\boldL}^{\mathbf\Sigma}
  \lesssim \left( \left(\Nmin^{012} \Nmin^{034}\right)^3
  \Lmin^{0'12} \Lmin^{034} \right)^{1/2}
  \prod_{j=1}^4\norm{u_j},
\end{equation}
which is immediate from \eqref{M:18} and Lemma \ref{D:Lemma3}. Taking \eqref{Y:2} to the power $8\varepsilon$ and \eqref{C:80} to the power $1-8\varepsilon$, we get
$$
  J_{\boldN,\boldL}^{\mathbf\Sigma}
  \lesssim
  \left( \left(\Nmin^{012} \Nmin^{034}\right)^3
  \Lmin^{0'12} \Lmin^{034} \right)^{4\varepsilon}
  \left( N_0^2 L_0L_0'L_1L_2L_3L_4 \right)^{1/2-4\varepsilon}
  \log\angles{L_0}
  \prod_{j=1}^4\norm{u_j}.
$$
Estimating $\log\angles{L_0} \lesssim L_0^{\varepsilon}$ and $(\Lmin^{0'12} \Lmin^{034})^{4\varepsilon} \lesssim (L_1L_3)^{4\varepsilon}$, inserting the above into \eqref{C:56}, and recalling the notation \eqref{B:200}, we see that is enough to prove
\begin{equation}\label{Y:10}
  S \lesssim \norm{F_1}\norm{F_2}\norm{F_3}\norm{F_4},
\end{equation}
where
$$
  S = \sum_{\boldN,\boldL} \frac{N_4^s  \bigl( \Nmin^{012} \Nmin^{034} \bigr)^{12\varepsilon}}{N_0^{8\varepsilon}(N_1N_2N_3)^s (L_0L_0'L_1 L_2 L_3L_4)^{\varepsilon}}
  \prod_{j=1}^4 \bignorm{\chi_{K^{\pm_j}_{N_j,L_j}} F_j}.
$$
The sum over $\boldN$ is restricted by the condition \eqref{C:41} and its counterpart for the inded $034$. Recall that all the $N$'s and $L$'s are greater than or equal to one.

Summing $\boldL$ is trivial:
$$
  \sum_{\boldL} \frac{1}{(L_0L_0'L_1 L_2 L_3L_4)^{\varepsilon}}
  \prod_{j=1}^4 \Bignorm{\chi_{K^{\pm_j}_{N_j,L_j}} F_j}
  \le C
  \prod_{j=1}^4 \norm{\chi_{\angles{\xi_j} \sim N_j} F_j},
$$
where $C = \sum_{\boldL} (L_0L_0'L_1 L_2 L_3L_4)^{-\varepsilon} < \infty$, so it only remains to prove \eqref{Y:10} for the reduced sum
\begin{equation}\label{Y:20}
  S' = \sum_{\boldN} \frac{N_4^s  \bigl( \Nmin^{012} \Nmin^{034} \bigr)^{12\varepsilon}}{N_0^{8\varepsilon}(N_1N_2N_3)^s}
  \prod_{j=1}^4 \norm{\chi_{\angles{\xi_j} \sim N_j} F_j}.
\end{equation}
Since  $(\Nmin^{012} \Nmin^{034})^{12\varepsilon} \le N_0^{24\varepsilon} \lesssim (\Nmax^{12})^{24\varepsilon}$,
\begin{equation}\label{Y:22}
  S' \lesssim \sum_{\boldN} \frac{N_4^s}{N_0^{8\varepsilon}(N_1N_2)^{s-24\varepsilon} N_3^s}
  \prod_{j=1}^4 \norm{\chi_{\angles{\xi_j} \sim N_j} F_j}.
\end{equation}
To ensure that $s-25\varepsilon > 0$, we choose $\varepsilon > 0$ so small that $25\varepsilon \le s$.

We now split $S' = S_1 + S_2 + S_3$, corresponding to the cases $N_4 \lesssim N_0 \sim N_3$,  $N_0 \ll N_3 \sim N_4$ and $N_3 \ll N_0 \sim N_4$, respectively.

\subsubsection{The case $N_4 \lesssim N_0 \sim N_3$} Since $\sum_{N_4 \lesssim N_0} N_4^s \sim N_0^s \sim N_3^s$,
$$
  S_1 \lesssim \sum_{N_0,N_1,N_2,N_3} \frac{1}{N_0^{7\varepsilon}(N_1N_2)^{s-24\varepsilon} N_3^\varepsilon}
  \prod_{j=1}^4 \norm{\chi_{\angles{\xi_j} \sim N_j} F_j},
$$
and this is trivially bounded by right hand side of \eqref{Y:10}.

\subsubsection{The case $N_0 \ll N_3 \sim N_4$} Then
$$
  S_2 \lesssim \sum_{\boldN} \frac{\chi_{N_3 \sim N_4}}{N_0^{8\varepsilon}(N_1N_2)^{s-24\varepsilon}}
  \prod_{j=1}^4 \norm{\chi_{\angles{\xi_j} \sim N_j} F_j}.
$$
Here $N_0,N_1,N_2$ sum outright, whereas $N_3 \sim N_4$ can be summed using the Cauchy-Schwarz inequality:
$$
  \sum_{N_3 \sim N_4 \ge 1} \norm{\chi_{\angles{\xi_3} \sim N_3} F_3} \norm{\chi_{\angles{\xi_4} \sim N_4} F_4}
  \lesssim \norm{F_3}\norm{F_4}.
$$

\subsubsection{The case $N_3 \ll N_0 \sim N_4$} Then $N_4 \lesssim \Nmax^{12}$. Now \eqref{Y:22} is too crude, but from \eqref{Y:20} we see that we can reduce to the sum
\begin{align*}
  S_3 &= \sum_{\boldN} \frac{N_4^s \chi_{N_0 \sim N_4 \lesssim \Nmax^{12}}}{N_0^{8\varepsilon}(\Nmin^{12})^{s-12\varepsilon}(\Nmax^{12})^s N_3^{s-12\varepsilon}}
  \prod_{j=1}^4 \norm{\chi_{\angles{\xi_j} \sim N_j} F_j}
  \\
  &\lesssim \sum_{\boldN} \frac{\chi_{N_0 \sim N_4}}{N_0^{7\varepsilon}(\Nmin^{12})^{s-12\varepsilon}(\Nmax^{12})^{\varepsilon} N_3^{s-12\varepsilon}}
  \prod_{j=1}^4 \norm{\chi_{\angles{\xi_j} \sim N_j} F_j},
\end{align*}
and this is trivial to sum.

\subsection{The low frequency part}\label{R}

Here we prove \eqref{C:12} for $J^{\boldsymbol\Sigma}_{\abs{\xi_0} < 1}$, without any dyadic decomposition. For this, we need the estimate, for $f \in \mathcal S(\R^3)$,
\begin{equation}\label{R:2}
  \norm{\Proj_{\abs{\xi} < 1} f}_{L^\infty} \le \abs{B(0,1)}^{1/2} \norm{f}_{L^2},
\end{equation}
or rather its dual,
\begin{equation}\label{R:4}
  \norm{\Proj_{\abs{\xi} < 1} f}_{L^2} \le \abs{B(0,1)}^{1/2} \norm{f}_{L^1}.
\end{equation}
Here $B(0,1)$ denotes the unit ball $\{ \xi \in \R^3 \colon \abs{\xi} < 1 \}$. Note that \eqref{R:2} follows from the Riemann-Lebesgue lemma and the Cauchy-Schwarz inequality.

We also need (this follows from the triangle inequality in Fourier space)
\begin{equation}\label{R:5}
  \norm{\Proj_{\abs{\xi} < 1} (fg)} \le C_s \norm{P_{\abs{\xi} < 1}(\angles{D}^s \fourierabs{f} \cdot \angles{D}^{-s} \fourierabs{g})},
\end{equation}
where we use the notation $\fourierabs{f} = \mathcal F_x^{-1}\abs{\widehat f\,}$. Furthermore, we need the crude estimate
\begin{equation}\label{R:6}
  \bignorm{\rho{\square}^{-1}F} \lesssim \norm{F}, 
\end{equation}
which follows from a cut-off estimate proved in \cite{Klainerman:1995b}, and we need
\begin{equation}\label{R:8}
  \norm{F}_{L_t^pL_x^2} \le C_{p,b} \norm{F}_{X_\pm^{0,(1-2/p)b}} \qquad (2 \le p \le \infty, \; b > 1/2).
\end{equation} 
The latter is trivial for $p=2$, so by interpolation it suffices to prove it for $p=\infty$, but then by Minkowski's integral inequality, the Riemann-Lebesgue lemma and Plancherel's theorem, the left hand side is bounded by $\fixednorm{\widetilde F(\tau,\xi)}_{L_\xi^2L_\tau^1}$. Inserting $\angles{\tau\pm\abs{\xi}}^b\angles{\tau\pm\abs{\xi}}^{-b}$ and applying the Cauchy-Schwarz inequality in $\tau$, one easily obtains the desired estimate.

Now we estimate, for any $b > 1/2$,
\begin{align*}
  J^{\boldsymbol\Sigma}_{\abs{\xi_0} < 1} &\le \bignorm{\rho\square^{-1} \Proj_{\abs{\xi} < 1}\innerprod{\boldsymbol\alpha^\mu\mathbf\Pi_{\pm_1}\psi_1}{\mathbf\Pi_{\pm_2}\psi_2}} \norm{\Proj_{\abs{\xi} < 1} \innerprod{\boldsymbol\alpha_\mu\mathbf\Pi_{\pm_3}\psi_3}{\mathbf\Pi_{\pm_4}\psi_4}}
  \\
  &\lesssim \norm{\Proj_{\abs{\xi} < 1}\innerprod{\boldsymbol\alpha^\mu\mathbf\Pi_{\pm_1}\psi_1}{\mathbf\Pi_{\pm_2}\psi_2}} \norm{\Proj_{\abs{\xi} < 1} \innerprod{\boldsymbol\alpha_\mu\mathbf\Pi_{\pm_3}\psi_3}{\mathbf\Pi_{\pm_4}\psi_4}}
  \\
  &\le C_s \norm{\Proj_{\abs{\xi} < 1}(\fourierabs{\psi_1}\fourierabs{\psi_2})} \norm{\Proj_{\abs{\xi} < 1} (\angles{D}^s\fourierabs{\psi_3}\cdot\angles{D}^{-s}\fourierabs{\psi_4})}
  \\
  &\le C_s \norm{\fourierabs{\psi_1}\fourierabs{\psi_2}}_{L_t^2L_x^1} \norm{\angles{D}^s\fourierabs{\psi_3}\cdot\angles{D}^{-s}\fourierabs{\psi_4}}_{L_t^2L_x^1}
  \\
  &\le C_s \norm{\psi_1}_{L_t^4L_x^2} \norm{\psi_2}_{L_t^4L_x^2} \norm{\angles{D}^s\psi_3}_{L_t^4L_x^2} \norm{\angles{D}^{-s}\psi_4}_{L_t^4L_x^2}
  \\
  &\le C_{s,b} \norm{\psi_1}_{X_{\pm_1}^{0,b/2}} \norm{\psi_2}_{X_{\pm_2}^{0,b/2}} \norm{\psi_3}_{X_{\pm_3}^{s,b/2}} \norm{\psi_4}_{X_{\pm_4}^{-s,b/2}},
\end{align*}
where to get the second inequality we used \eqref{R:6}, and then we used \eqref{R:5}, \eqref{R:4}, H\"older's inequality and \eqref{R:8}. Finally, if we write $b=1/2+\varepsilon$, where $\varepsilon > 0$, then we see that $b/2 \le 1/2-2\varepsilon$ for all $\varepsilon \le 1/10$, hence we have proved \eqref{C:12} for the low frequency part. Notice that we did not need the null structure.

This concludes the proof of \eqref{C:12}.

\section{Proof of the trilinear estimate}\label{S}

Here we prove \eqref{C:10}. Recall that $A_\mu^{\mathrm{hom.}}$ is the solution of \eqref{A:62}, the data being determined by \eqref{A:42}, \eqref{A:48}, with regularity as in \eqref{A:102}. Thus,
$$
  A_0^{\mathrm{hom.}} = 0,
$$
whereas $A_j^{\mathrm{hom.}}$ for $j=1,2,3$ splits in the usual way:
$$
  A_j^{\mathrm{hom.}} = A_{j,+}^{\mathrm{hom.}} + A_{j,-}^{\mathrm{hom.}},
$$
where
$$
  \widetilde{A_{j,\pm_0}^{\mathrm{hom.}}}(X_0)
  =
  \delta(\tau_0\pm_0\abs{\xi_0}) 
  \frac{g_j^{\pm_0}(\xi_0)}{\abs{\xi_0}\angles{\xi_0}^{s-1/2}}
  \qquad (X_0 = (\tau_0,\xi_0)),
$$
and $g_j^+,g_j^- \in L^2(\R^3)$ are defined by
\begin{equation}\label{S:4}
  g_j^{\pm_0}(\xi_0)
  =
  \abs{\xi_0}\angles{\xi_0}^{s-1/2}
  \left( \frac{\widehat{a_j}(\xi_0)}{2} \pm_0 \frac{\widehat{\dot a_j}(\xi_0)}{2i\abs{\xi_0}} \right),
\end{equation}
hence
\begin{equation}\label{S:6}
  \norm{g^\pm} \lesssim \mathcal I_0,
\end{equation}
where $\mathcal I_0$ is as in \eqref{C:3}. By \eqref{A:32}--\eqref{A:48}, $\nabla \cdot \mathbf a = 0$ and $\nabla \cdot \dot{\mathbf a} = - \abs{\psi_0}^2$, hence
\begin{equation}\label{S:10}
  \xi_0^j g_j^{\pm_0}(\xi_0) \simeq \angles{\xi_0}^{s-1/2} \widehat{\abs{\psi_0}^2}(\xi_0),
\end{equation}
where we implicitly sum over $j=1,2,3$ on the left hand side.

Now write $\boldsymbol\Sigma = (\pm_0,\pm_1,\pm_2)$, and let $I^{\boldsymbol\Sigma}$ be defined like $I^{\pm_1,\pm_2}$ in \eqref{C:14}, except that $A_j^{\mathrm{hom.}}$ is replaced by $A_{j,\pm_0}^{\mathrm{hom.}}$. By Plancherel's formula,
$$
  I^{\boldsymbol\Sigma}
  \simeq
  \iint
  \widehat\rho(\tau_0\pm_0\abs{\xi_0}) 
  \frac{\sigma^j(X_1,X_2) g_j^{\pm_0}(\xi_0)F_1(X_1)
  F_2(X_2)}
  {\abs{\xi_0}\angles{\xi_0}^{s-1/2}\prod_{k=1}^2\angles{\xi_k}^{s_k} \angles{\tau_k\pm_k\abs{\xi_k}}^{b_k}}
  \, d\mu^{21}_{X_0}
  \, dX_0,
$$
where
\begin{gather*}
  s_1=-s_2=s,
  \qquad
  b_1 = 1/2+\varepsilon,
  \qquad
  b_2 = 1/2-2\varepsilon,
  \\
  \widetilde{\psi_k}
  = z_k \bigabs{\widetilde{\psi_k}},
  \qquad
  \bigabs{\widetilde{\psi_k}(X_k)}
  = \frac{F_k(X_k)}{\angles{\xi_k}^{s_k} \angles{\tau_k\pm_k\abs{\xi_k}}^{b_k}},
  \qquad X_k = (\tau_k,\xi_k), 
  \\
  \sigma^j(X_1,X_2) = \innerprod{\boldsymbol\alpha^j\mathbf\Pi(e_1)z_1(X_1)}{\mathbf\Pi(e_2)z_2(X_2)},
  \qquad
  e_k = \pm_k \frac{\xi_k}{\abs{\xi_k}} \in \mathbb S^2.
\end{gather*}
Here $z_k : \R^{1+3} \to \C^4$ is measurable, $\abs{z_k}=1$, $F_k \in L^2(\R^{1+3)}$ and $F_k \ge 0$, for $k=1,2$. The convolution measure $ d\mu^{21}_{X_0}$ is given by the rule in \eqref{B:2}, hence
$$
  X_0 = X_2 - X_1 \qquad \left( \iff \tau_0=\tau_2-\tau_1, \quad \xi_0=\xi_2-\xi_1 \right)
$$
in the above integral. We also define the angles $\theta_{01}$, $\theta_{02}$ and $\theta_{12}$ by \eqref{C:40}.

We want to prove the estimate
\begin{equation}\label{S:28}
  \Abs{I^{\boldsymbol\Sigma}_{\abs{\xi_0} \ge 1}} \lesssim C(\mathcal I_0) \norm{F_1}\norm{F_2}.
\end{equation}
Split
$$
  I^{\boldsymbol\Sigma} =  I^{\boldsymbol\Sigma}_{\abs{\xi_0} \ge 1} +  I^{\boldsymbol\Sigma}_{\abs{\xi_0} < 1}
$$
corresponding to the regions $\abs{\xi_0} \ge 1$  and $\abs{\xi_0} < 1$.

\subsection{Estimate for $I^{\pm_1,\pm_2}_{\abs{\xi_0} \ge 1}$} Let $N_0,N_1,N_2,L_0,L_1,L_2 \ge 1$ be dyadic numbers representing the sizes of the weights, as in section \ref{C}. Taking the absolute value and using the fact that $\widehat\rho$ is a Schwartz function (hence we can get as many powers as we like of $L_0$ in the denominator), we get
\begin{equation}\label{S:20}
  \Abs{I^{\boldsymbol\Sigma}_{\abs{\xi_0} \ge 1}}
  \lesssim
  \sum_{\boldN,\boldL} \frac{N_2^s I^{\boldsymbol\Sigma}_{\boldN,\boldL}}{N_0^{s+1/2} N_1^s L_0 L_1^{1/2+\varepsilon} L_2^{1/2-2\varepsilon}},
\end{equation}
where $\boldN = (N_0,N_1,N_2)$, $\boldL = (L_0,L_1,L_2)$ and, with notation as in \eqref{B:200},
\begin{gather*}
  I^{\boldsymbol\Sigma}_{\boldN,\boldL}
  =
  \iint
  \abs{\sigma_{12}^j(X_1,X_2)}
  \widetilde{u_{0,j}}(X_0)
  \widetilde{u_1}(X_1)
  \widetilde{u_2}(X_2)
  \, d\mu^{21}_{X_0}
  \, dX_0,
  \\
  \widetilde{u_{0,j}} = \chi_{K^{\pm_0}_{N_0,L_0}} F_{0,j},
  \qquad
  F_{0,j}(X_0)
  =
  \frac{\abs{g_j^{\pm_0}(\xi_0)}}{\angles{\tau_0\pm_0\abs{\xi_0}}}
  \qquad (j=1,2,3).
\end{gather*}
The sum over $\boldN$ is restricted by \eqref{C:41}.

By the same type of summation argument that was used in section \ref{Y} (we omit the details), \eqref{S:28} is easily deduced from \eqref{S:20} if we can prove the following:
\begin{align}
  \label{S:30}
  I^{\boldsymbol\Sigma}_{\boldN,\boldL}
  &\lesssim \left(N_0L_0^2L_1L_2\right)^{1/2}
  C(\mathcal I_0) \norm{u_1}\norm{u_2},
  \\
  \label{S:32}
  I^{\boldsymbol\Sigma}_{\boldN,\boldL}
  &\lesssim \left(
  \left(\Nmin^{012}\right)^3 \Lmin^{012}\right)^{1/2}
  C(\mathcal I_0) \norm{u_1}\norm{u_2}.
\end{align}

First, \eqref{S:32} follows from \eqref{M:18}, if we estimate $\abs{\sigma^j(X_1,X_2)} \lesssim 1$ and use the fact that $\norm{u_0} \lesssim \mathcal I_0$, by \eqref{S:6}.

To prove \eqref{S:30}, on the other hand, we need to use the structure of the symbol $\sigma_{12}^j$, encoded in the identity \eqref{D:54}. We claim that
\begin{multline}\label{S:40}
  \Abs{\sigma^j(X_1,X_2) g_j^{\pm_0}(\xi_0)}
  \lesssim \theta_{12} \abs{g^{\pm_0}(\xi_0)}
  + \min(\theta_{01},\theta_{02}) \abs{g^{\pm_0}(\xi_0)}
  \\
  +
  \angles{\xi_0}^{s-3/2} \Abs{\widehat{\abs{\psi_0}^2}(\xi_0)},
\end{multline}
where again we sum over $j=1,2,3$, and we assume $\abs{\xi_0} \ge 1$.

To prove \eqref{S:40}, we use the identity \eqref{D:54} to write
\begin{align*}
  \sigma^j(X_1,X_2) g_j^{\pm_0}(\xi_0)
  &=
  \innerprod{\boldsymbol\alpha^j\mathbf\Pi(e_1)z_1}{\mathbf\Pi(e_2)z_2}
  g_j^{\pm_0}(\xi_0)
  \\
  &=
  \innerprod{\mathbf\Pi(e_2)\mathbf\Pi(-e_1)\boldsymbol\alpha^j z_1}{z_2}
  g_j^{\pm_0}(\xi_0)
  + \innerprod{z_1}{\mathbf\Pi(e_2)z_2} e_1^j g_j^{\pm_0}(\xi_0)
  \\
  &=
  \innerprod{\mathbf\Pi(e_2)\mathbf\Pi(-e_1)\boldsymbol\alpha^j z_1}{z_2}
  g_j^{\pm_0}(\xi_0)
  \\
  & \;\; + \innerprod{z_1}{\mathbf\Pi(e_2)z_2} (e_1^j-e_0^j) g_j^{\pm_0}(\xi_0)
  +  \innerprod{z_1}{\mathbf\Pi(e_2)z_2} e_0^j g_j^{\pm_0}(\xi_0),
\end{align*}
and by \eqref{B:17} and \eqref{S:10} this implies \eqref{S:40} with $\min(\theta_{01},\theta_{02})$ replaced by $\theta_{01}$. But since $\boldsymbol\alpha^j$ is self-adjoint, we can also move it onto the second factor in the inner product defining $\sigma^j$, and then we get instead the angle $\theta_{02}$. This proves \eqref{S:40}.

Corresponding to the first and second terms in the right side of \eqref{S:40}, we need to prove \eqref{S:30} for the integrals
\begin{align*}
  I^{\boldsymbol\Sigma,1}_{\boldN,\boldL}
  &=
  \iint
  \theta_{12}
  \widetilde{u_0}(X_0)
  \widetilde{u_1}(X_1)
  \widetilde{u_2}(X_2)
  \, d\mu^{21}_{X_0}
  \, dX_0,
  \\
  I^{\boldsymbol\Sigma,2}_{\boldN,\boldL}
  &=
  \iint
  \min(\theta_{01},\theta_{02})
  \widetilde{u_0}(X_0)
  \widetilde{u_1}(X_1)
  \widetilde{u_2}(X_2)
  \, d\mu^{21}_{X_0}
  \, dX_0,
\end{align*}
where now
$$
  F_0(X_0) = \frac{\abs{g^{\pm_0}(\xi_0)}}
  {\angles{\tau_0\pm_0\abs{\xi_0}}},
$$
hence $\norm{u_0} \simeq \norm{F_0} \lesssim \mathcal I_0$, by \eqref{S:6}.

For $I^{\boldsymbol\Sigma,1}_{\boldN,\boldL}$, we get \eqref{S:30} (with only one power of $L_0$ inside the parentheses) from the null form estimate in Theorem \ref{D:Thm}.

Now consider $I^{\boldsymbol\Sigma,2}_{\boldN,\boldL}$. By Lemma \ref{F:Lemma4},
$$
  \min(\theta_{01},\theta_{02})
  \lesssim \left( \frac{\Lmax^{012}}{N_0} \right)^{1/2},
$$
so by the Cauchy-Schwarz inequality,
\begin{align*}
  I^{\boldsymbol\Sigma,2}_{\boldN,\boldL}
  &\lesssim
  \left( \frac{\Lmax^{012}}{N_0} \right)^{1/2}
  \norm{u_0} \bignorm{ \Proj_{K^{\pm_0}_{N_0,L_0}} \left( u_1 \overline{u_2} \right) }
  \\
  &\lesssim
  \left( \frac{\Lmax^{012}}{N_0} N_0^2 L_0 \Lmin^{12} \right)^{1/2}
  \mathcal I_0 \norm{u_1}\norm{u_2},
\end{align*}
where we used \eqref{M:14} and $\norm{u_0} \lesssim \mathcal I_0$. This proves \eqref{S:30} for $I^{\boldsymbol\Sigma,2}_{\boldN,\boldL}$. Note that here we may actually pick up two powers of $L_0$ inside the parentheses (recall that this is allowed because $\widehat \rho$ is rapidly decreasing).

Now consider the case where the third term in the right side of \eqref{S:40} dominates. We may assume $\theta_{12} \ll 1$, since otherwise we can reduce to $I^{\boldsymbol\Sigma,1}_{\boldN,\boldL}$ by estimating $\abs{\sigma^j(X_1,X_2)} \lesssim 1$. So by Lemma \ref{D:Lemma1},
$$
  \theta_{12} \lesssim \gamma \equiv \left( \frac{N_0\Lmax^{012}}{N_1N_2} \right)^{1/2},
$$
hence we need to prove \eqref{S:30} for
$$
  I^{\boldsymbol\Sigma,3}_{\boldN,\boldL}
  = N_0^{s-3/2}
  \iint
  \chi_{\theta_{12} \lesssim \gamma}
  \widetilde{u_0}(X_0)
  \widetilde{u_1}(X_1)
  \widetilde{u_2}(X_2)
  \, d\mu^{21}_{X_0}
  \, dX_0,
$$
where now
$$
  F_0(\tau_0,\xi_0)
  =
  \frac{\Abs{\widehat{\abs{\psi_0}^2}(\xi_0)}}{\angles{\tau_0\pm_0\abs{\xi_0}}}.
$$
By Lemma \ref{B:Lemma5} applied to the pair $(\pm_1\xi_1,\pm_2\xi_2)$,
\begin{equation}\label{S:60}
  I^{\boldsymbol\Sigma,3}_{\boldN,\boldL}
  \lesssim N_0^{s-3/2}
  \sum_{\omega_1,\omega_2}
  \iint
  \widetilde{u_0}(X_0)
  \widetilde{u_1^{\gamma,\omega_1}}(X_1)
  \widetilde{u_2^{\gamma,\omega_2}}(X_2)
  \, d\mu^{21}_{X_0}
  \, dX_0,
\end{equation}
where the sum is over $\omega_1,\omega_2 \in \Omega(\gamma)$ with $\theta(\omega_1,\omega_2) \lesssim \gamma$. Thus, $\xi_1,\xi_2$ are both restricted to a tube of radius
$$
  r \sim \Nmax^{12}\gamma \sim \left( \frac{\Nmax^{12}N_0\Lmax^{012}}{\Nmin^{12}} \right)^{1/2}
$$
around $\R\omega_1$, hence the same is true of $\xi_0=\xi_2-\xi_1$, so we get
\begin{align*}
  I^{\boldsymbol\Sigma,3}_{\boldN,\boldL}
  &\lesssim N_0^{s-3/2}
  \sum_{\omega_1,\omega_2}
  \norm{\Proj_{\R \times T_r(\omega_1)} u_0}
  \norm{ \Proj_{K^{\pm_0}_{N_0,L_0}} \left( u_1^{\gamma,\omega_1}
  \overline{u_2^{\gamma,\omega_2}} \right) }
  \\
  &\lesssim
  N_0^{s-3/2}
  \sum_{\omega_1,\omega_2}
  \norm{\Proj_{T_r(\omega_1)}
  \Proj_{\angles{\xi_0} \sim N_0} \abs{\psi_0}^2}
  \left( N_0^2 L_0 \Lmin^{12} \right)^{1/2} \norm{u_1^{\gamma,\omega_1}} 
  \norm{u_2^{\gamma,\omega_2}}.
\end{align*}
where we used \eqref{M:14}. Applying the estimate (proved below)
\begin{equation}\label{S:56}
  \sup_{\omega \in \mathbb S^2} \norm{\Proj_{T_r(\omega)}
  \Proj_{\angles{\xi_0} \sim N_0} \abs{\psi_0}^2}
  \lesssim
  \left( r^2 N_0 \right)^{1/2} N_0^{-s} \norm{\psi_0}_{H^s}^2,
\end{equation}
and summing $\omega_1,\omega_2$ as in \eqref{B:208}, we then obtain
\begin{align*}
  I^{\boldsymbol\Sigma,3}_{\boldN,\boldL}
  &\lesssim N_0^{s-3/2}
  \left( r^2 N_0 \right)^{1/2} N_0^{-s} \norm{\psi_0}_{H^s}^2
  \left( N_0 \Nmin^{012} L_0 \Lmin^{12} \right)^{1/2}
  \norm{u_1}\norm{u_2}
  \\
  &\sim
  N_0^{-3/2}
  \left( \frac{\Nmax^{12}N_0\Lmax^{012}}{\Nmin^{12}} N_0 \right)^{1/2} \norm{\psi_0}_{H^s}^2
  \left( N_0 \Nmin^{012} L_0 \Lmin^{12} \right)^{1/2}
  \norm{u_1}\norm{u_2}
  \\
  &\sim
  \left( \frac{\Nmax^{12}\Nmin^{012}}{\Nmin^{12}} L_0 \Lmin^{12} \Lmax^{012}  \right)^{1/2} \norm{\psi_0}_{H^s}^2
  \norm{u_1}\norm{u_2}
  \\
  &\lesssim
  N_0^{-3/2}
  \left( N_0 L_0 \Lmin^{12} \Lmax^{012} \right)^{1/2} \norm{\psi_0}_{H^s}^2
  \norm{u_1}\norm{u_2},
\end{align*}
where we used \eqref{C:43} in the last step. This proves \eqref{S:30} for $I^{\boldsymbol\Sigma,3}_{\boldN,\boldL}$, under the assumption that \eqref{S:56} holds.

In fact, \eqref{S:56} is an easy consequence of the estimate
\begin{equation}\label{S:66}
  \sup_{\omega \in \mathbb S^2} \norm{\Proj_{T_r(\omega)}
  \Proj_{\angles{\xi_0} \sim N_0}(fg)}
  \lesssim
  \left( r^2 N_0 \right)^{1/2} \norm{f} \norm{g}
  \qquad \left(\forall f,g \in \mathcal S(\R^3)\right)
\end{equation}
which reduces, by an argument based on the Cauchy-Schwarz inequality (see \cite{Tao:2001} or \cite{Selberg:2008a}), to the fact that the volume of the $\xi_0$-support is $O(r^2N_0)$.

This completes the proof of \eqref{S:28}.

\subsection{Estimate for $I^{\pm_1,\pm_2}_{\abs{\xi_0} < 1}$} Since $\angles{\xi_2} \lesssim \angles{\xi_0} + \angles{\xi_1} \lesssim \angles{\xi_1}$, and since $\widehat \rho$ is rapidly decreasing,
\begin{align*}
  \Abs{I^{\boldsymbol\Sigma}_{\abs{\xi_0} < 1}}
  &\lesssim
  \iint
  \frac{\chi_{\abs{\xi_0} < 1} \abs{g^{\pm_0}(\xi_0)}}{\abs{\xi_0}\angles{\tau_0\pm_0\abs{\xi_0}}^2}
  F_1(X_1) F_2(X_2)
  \, d\mu^{21}_{X_0}
  \, dX_0
  \\
  &=
  \iint
  \frac{\chi_{\abs{\xi_0} < 1} \abs{g^{\pm_0}(\xi_0)}}{\abs{\xi_0}\angles{\tau_0\pm_0\abs{\xi_0}}^2}
  F_1(X_1) F_2(X_0+X_1)
  \, dX_1
  \, dX_0
  \\
  &\le
  \int
  \frac{\chi_{\abs{\xi_0} < 1} \abs{g^{\pm_0}(\xi_0)}}{\abs{\xi_0}\angles{\tau_0\pm_0\abs{\xi_0}}^2}
  \, dX_0
  \norm{F_1}\norm{F_2}
  \\
  &\lesssim
  \int
  \frac{\chi_{\abs{\xi_0} < 1} \abs{g^{\pm_0}(\xi_0)}}{\abs{\xi_0}}
  \, d\xi_0
  \norm{F_1}\norm{F_2}
  \\
  &\le
  \left(\int
  \frac{\chi_{\abs{\xi_0} < 1}}{\abs{\xi_0}^2}
  \, d\xi_0 \right)^{1/2}
  \norm{g^{\pm_0}}
  \norm{F_1}\norm{F_2},
\end{align*}
proving \eqref{S:28}.

This completes the proof of \eqref{C:10}.

\section{Estimates for the electromagnetic field}\label{N}

Here we prove Theorem \ref{A:Thm2}. Since Maxwell's equations are linear, uniqueness is trivial, so we only need to construct the solution. Let us define $(\mathbf E,\mathbf B)$ by \eqref{A:6}. Since we know that $\square A_\mu = - J_\mu$ and that $A_\mu$ satisfies the Lorenz gauge condition, a direct calculation shows that \eqref{A:2} is satisfied, so it only remains to prove
$$
  \norm{\mathbf E(t)}_{H^{s-1/2}}
  +
  \norm{\mathbf B(t)}_{H^{s-1/2}}
  \le C
  \qquad (\forall t \in [-T,T]),
$$
where $C$ depends on the data norm $\mathcal I_0$ defined by \eqref{C:3}. But the components of $\mathbf E, \mathbf B$ are just the nonzero components of the electromagnetic tensor
$$
  F_{\kappa\lambda} = \partial_\kappa A_\lambda - \partial_\lambda A_\kappa,
$$
so we need to prove
\begin{equation}\label{N:10}
  \norm{F_{\kappa\lambda}(t)}_{H^{s-1/2}}
  \le C \qquad (\forall t \in [-T,T]).
\end{equation}
Of course, it suffices to consider indices $(\kappa,\lambda) = (k,l), (k,0)$, where $k,l=1,2,3$.

Since $\square A_\kappa = - J_\kappa$,
$$
  \square F_{\kappa\lambda} = -\partial_\kappa J_\lambda + \partial_\lambda J_\kappa.
$$
Split
$$
  F_{\kappa\lambda} = F_{\kappa\lambda}^{\text{hom.}} + F_{\kappa\lambda}^{\text{inh.}},
$$
where $\square F_{\kappa\lambda}^{\text{hom.}} = 0$ with the initial data determined by $(\mathbf E_0,\mathbf B_0)$, and
\begin{equation}\label{N:10:2}
  F_{\kappa\lambda}^{\text{inh.}}
  =
  \square^{-1}
  \left( -\partial_\kappa J_\lambda + \partial_\lambda J_\kappa \right).
\end{equation}

For the homogeneous part $F_{\kappa\lambda}^{\text{hom.}}$, \eqref{N:10} holds by the energy inequality for the wave equation and the assumption $\mathbf E_0,\mathbf B_0 \in H^{s-1/2}$.

It remains to prove \eqref{N:10} for $F_{\kappa\lambda}^{\text{inh.}}$. Splitting $\psi = \psi_+ + \psi_-$ we see from \eqref{A:10} and \eqref{B:2} that
$$
  \Abs{\widetilde{J_\kappa}(X_0)}
  \le
  \sum_{\pm_1,\pm_2}
  \int
  \frac{\Innerprod{\boldsymbol\alpha_\kappa\boldsymbol\Pi(e_1)z_1}{\boldsymbol\Pi(e_2)z_2} G_1(X_1) G_2(X_2)}
  {\angles{\xi_1}^s \angles{\xi_2}^s \angles{\tau_1\pm_1\abs{\xi_1}}^{1/2+\varepsilon} \angles{\tau_2\pm_2\abs{\xi_2}}^{1/2+\varepsilon}}
  \, d\mu^{12}_{X_0},
$$
where $e_j = \pm_j \xi_j/\abs{\xi_j}$, $z_j : \R^{1+3} \to \C^4$ is measurable, $\abs{z_j}=1$, $G_j \in L^2(\R^{1+3)}$ and $G_j \ge 0$, for $j=1,2$. Now observe that the symbol of $(1/i)\partial_\kappa$ is
$$
  X_0^\kappa = \begin{cases} \tau_0 \quad &\text{for $\kappa=0$},
  \\
  \xi_0^\kappa \quad &\text{for $\kappa=1,2,3$},
  \end{cases}
$$
recalling that $\xi_0=(\xi_0^1,\xi_0^2,\xi_0^3)$. Thus, applying Lemma \ref{C:Lemma1} to \eqref{N:10:2}, and writing $\boldsymbol\Sigma = (\pm_0,\pm_1,\pm_2)$, we have
\begin{equation}\label{N:10:4}
  \norm{F_{\kappa\lambda}^{\text{inh.}}(t)}_{H^{s-1/2}}
  \lesssim 
  \sum_{\boldsymbol\Sigma} I^{\boldsymbol\Sigma}
  \qquad (\forall t \in [-T,T]),
\end{equation}
where
\begin{gather*}
  I^{\boldsymbol\Sigma}
  =
  \norm{ \iint \frac{\sigma_{\kappa\lambda}(X_1,X_2) G_1(X_1) G_2(X_2)}
  {\abs{\xi_0}\angles{\xi_0}^{1/2-s} \angles{\tau_0\pm_0\abs{\xi_0}}
  \prod_{j=1}^2\angles{\xi_j}^s\angles{\tau_j\pm_j\abs{\xi_j}}^{1/2+\varepsilon}}
  \, d\mu^{12}_{X_0} \, d\tau_0 }_{L^2_{\xi_0}},
  \\
  \sigma_{\kappa\lambda}(X_1,X_2)
  = X_0^\kappa \Innerprod{\boldsymbol\alpha_\lambda\boldsymbol\Pi(e_1)z_1}{\boldsymbol\Pi(e_2)z_2}
  -
  X_0^\lambda \Innerprod{\boldsymbol\alpha_\kappa\boldsymbol\Pi(e_1)z_1}{\boldsymbol\Pi(e_2)z_2}.
\end{gather*}

Define $\theta_{12}, \theta_{01}, \theta_{02}$ as in \eqref{C:40}. We have the following null structure:

\begin{lemma}\label{N:Lemma} With notation as above,
\begin{align}
  \label{N:12}
  \frac{\Abs{\sigma_{kl}(X_1,X_2)}}{\abs{\xi_0}}
  &\lesssim
  \theta_{12} + \min(\theta_{01},\theta_{02}),
  \\
  \label{N:14}
  \frac{\Abs{\sigma_{k0}(X_1,X_2)}}{\abs{\xi_0}}
  &\lesssim
  \theta_{12} + \min(\theta_{01},\theta_{02})
  + \frac{\bigabs{\tau_0\pm_0\abs{\xi_0}}}{\abs{\xi_0}},
\end{align}
for $k,l=1,2,3$.
\end{lemma} 

\begin{proof} By the rule for raising or lowering indices, $\boldsymbol\alpha_k=\boldsymbol\alpha^k$ for $k=1,2,3$, whereas $\boldsymbol\alpha_0=-\boldsymbol\alpha^0 = \mathbf I_{4 \times 4}$. Thus,
\begin{align*}
  \frac{\sigma_{kl}(X_1,X_2)}{\abs{\xi_0}} &= \frac{\xi_0^k}{\abs{\xi_0}} \Innerprod{\boldsymbol\alpha^l\boldsymbol\Pi(e_1)z_1}{\boldsymbol\Pi(e_2)z_2}
  -
  \frac{\xi_0^l}{\abs{\xi_0}} \Innerprod{\boldsymbol\alpha^k\boldsymbol\Pi(e_1)z_1}{\boldsymbol\Pi(e_2)z_2},
  \\
  - \frac{\sigma_{k0}(X_1,X_2)}{\abs{\xi_0}} &= \frac{\xi_0^k}{\abs{\xi_0}} \Innerprod{\boldsymbol\Pi(e_1)z_1}{\boldsymbol\Pi(e_2)z_2}
  +
  \frac{\tau_0}{\abs{\xi_0}} \Innerprod{\boldsymbol\alpha^k\boldsymbol\Pi(e_1)z_1}{\boldsymbol\Pi(e_2)z_2}.
\end{align*}

Then by the commutation identity \eqref{D:54} we see that
$$
  \frac{\Abs{\sigma_{kl}(X_1,X_2)}}{\abs{\xi_0}}
  \lesssim \theta_{12}
  + \abs{e_0^k e_1^l - e_0^l e_1^k}
  \lesssim \theta_{12} + \theta_{01},
$$
but since the $\boldsymbol\alpha^k$ are self-adjoint, we get also the same estimate with $\theta_{02}$ instead of $\theta_{01}$, proving \eqref{N:12}.

Similarly we find that
$$
  \frac{\Abs{\sigma_{k0}(X_1,X_2)}}{\abs{\xi_0}}
  \lesssim \theta_{12}
  + \abs{e_0^k - e_1^k}
  + \frac{\bigabs{\tau_0\pm_0\abs{\xi_0}}}{\abs{\xi_0}}
  \lesssim \theta_{12} + \theta_{01} + \frac{\bigabs{\tau_0\pm_0\abs{\xi_0}}}{\abs{\xi_0}},
$$
but again, by the self-adjointness of the $\boldsymbol\alpha^k$, we can also get $\theta_{02}$ instead of $\theta_{01}$, proving \eqref{N:14}.
\end{proof}

In view of \eqref{N:10:4} and Lemma \ref{N:Lemma}, we reduce \eqref{N:10} to proving (dropping the superscript $\boldsymbol\Sigma$ for simplicity)
\begin{equation}\label{N:20}
  I_1, I_2, I_3 \lesssim \norm{F_1}\norm{F_2},
\end{equation}
where
\begin{align*}
  I_1 &=
  \norm{ \iint \frac{\theta_{12} G_1(X_1) G_2(X_2)}{\angles{\xi_0}^{1/2-s} \angles{\tau_0\pm_0\abs{\xi_0}}
  \prod_{j=1}^2\angles{\xi_j}^s\angles{\tau_j\pm_j\abs{\xi_j}}^{1/2+\varepsilon}}
  \, d\mu^{12}_{X_0} \, d\tau_0 }_{L^2_{\xi_0}},
  \\
  I_2 &=
  \norm{ \iint \frac{\min(\theta_{01},\theta_{02}) G_1(X_1) G_2(X_2)}{\angles{\xi_0}^{1/2-s} \angles{\tau_0\pm_0\abs{\xi_0}} 
  \prod_{j=1}^2\angles{\xi_j}^s\angles{\tau_j\pm_j\abs{\xi_j}}^{1/2+\varepsilon}}
  \, d\mu^{12}_{X_0} \, d\tau_0 }_{L^2_{\xi_0}},
  \\
  I_3 &=
  \norm{ \iint \frac{G_1(X_1) G_2(X_2)}{\angles{\xi_0}^{1/2-s}\abs{\xi_0} 
  \prod_{j=1}^2\angles{\xi_j}^s\angles{\tau_j\pm_j\abs{\xi_j}}^{1/2+\varepsilon}}
  \, d\mu^{12}_{X_0} \, d\tau_0 }_{L^2_{\xi_0}}.
\end{align*}

\subsection{Estimate for $I_1$} By dyadic decomposition as in section \ref{C},
\begin{equation}\label{N:30}
  I_1 \lesssim
  \sum_{\boldN,\boldL}
  \frac{J_{\boldN,\boldL}}{N_0^{1/2-s}(N_1N_2)^sL_0(L_1L_2)^{1/2+\varepsilon}},
\end{equation}
where $\boldN = (N_0,N_1,N_2)$, $\boldL = (L_0,L_1,L_2)$ and
\begin{equation}\label{N:31}
  J_{\boldN,\boldL}
  =
  \norm{ \iint \chi_{K^{\pm_0}_{N_0,L_0}}\!(X_0)
  \, \theta_{12} \,
  \widetilde{u_1}(X_1) \widetilde{u_2}(X_2)
  \, d\mu^{12}_{X_0} \, d\tau_0 }_{L^2_{\xi_0}},
\end{equation}
where we use the notation \eqref{B:200}, but with $G$'s instead of $F$'s.

Applying the Cauchy-Schwarz inequality with respect to $\tau_0$, followed by either Theorem \ref{D:Thm} or the trivial estimate \eqref{M:18} (which holds without the angle),
\begin{align}
  \label{N:32}
  J_{\boldN,\boldL}
  &\lesssim L_0^{1/2} \left(N_0 L_0 L_1 L_2\right)^{1/2}
  \norm{u_1}\norm{u_2},
  \\
  \label{N:34}
  J_{\boldN,\boldL}
  &\lesssim L_0^{1/2} \left(\left(\Nmin^{012}\right)^3 \Lmin^{012}\right)^{1/2}
  \norm{u_1}\norm{u_2},
\end{align}
Take the former to the power $1-2\varepsilon$ and the latter to the power $2\varepsilon$. Then plugging the interpolated estimate into \eqref{N:30} and summing by the same type of argument that was used in section \ref{Y}, we get \eqref{N:20} for $I_1$ (for $\varepsilon > 0$ sufficiently small depending on $s > 0$).

\subsection{Estimate for $I_2$} Again we dominate by a sum like \eqref{N:30}, but now $\theta_{12}$ in \eqref{N:31} is replaced by $\min(\theta_{01},\theta_{02})$. We need to prove \eqref{N:32}. By Lemma \ref{F:Lemma4},
\begin{equation}\label{N:40}
  \min(\theta_{01},\theta_{02})
  \lesssim \left( \frac{\Lmax^{012}}{N_0} \right)^{1/2}.
\end{equation}
By comparison, in the estimate for $I_1$ we used $\theta_{12} \lesssim (\Lmax^{012}/\Nmin^{12})^{1/2}$ to get \eqref{N:32} (implicitly, since we used Theorem \ref{D:Thm}). Thus, the analysis for $I_1$ applies also here if $N_0 \sim \Nmin^{12}$, so we may assume
\begin{equation}\label{N:41}
  N_0 \ll N_1 \sim N_2.
\end{equation}
We distinguish the cases $L_0 \lesssim \Lmax^{12}$ and $L_1,L_2 \ll L_0$.

\subsubsection{The case $L_0 \lesssim \Lmax^{12}$} We apply Cauchy-Schwarz as we did for $I_1$ (hence we pick up a factor $L_0^{1/2}$), and use \eqref{N:40} and the bilinear estimate \eqref{M:14}, obtaining
\begin{align*}
  J_{\boldN,\boldL}
  &\lesssim L_0^{1/2} \left( \frac{\Lmax^{12}}{N_0} \right)^{1/2}
  \left(N_0^2 L_0 \Lmin^{12}\right)^{1/2}
  \norm{u_1}\norm{u_2},
\end{align*}
proving \eqref{N:32}.

\subsubsection{The case $L_1,L_2 \ll L_0$} Now we get into trouble if we try to follow the same approach as above, since we would need Theorem \ref{M:Thm} to hold with $C^2 \sim N_0^2 L_1 L_2$, but this is not true in general. The problem is that we pick up too many powers of $L_0$. But instead of using $\tau_0\pm_0\abs{\xi_0}=O(L_0)$ when we apply the Cauchy-Schwarz inequality with respect to $\tau_0$, we can find another restriction on $\tau_0$ by decomposing into angular sectors based on the maximal size of $\theta_{12}$, as we now show.

We may assume $\theta_{12} \ll 1$, since otherwise we reduce to $I_1$. Therefore
\begin{equation}\label{N:42}
  \theta_{12} \lesssim \gamma \equiv \biggl(\frac{N_0L_0}{N_1N_2}\biggr)^{1/2},
\end{equation}
and now we apply Lemma \ref{D:Lemma5} and \eqref{N:40}, thus dominating $J_{\boldN,\boldL}$ by
\begin{equation}\label{N:43}
  \sum_{\omega_1,\omega_2}
  \left( \frac{L_0}{N_0} \right)^{1/2}
  \norm{ \iint \chi_{K^{\pm_0}_{N_0,L_0}}\!(X_0)
  \,
  \widetilde{u_1^{\gamma,\omega_1}}(X_1)
  \widetilde{u_2^{\gamma,\omega_2}}(X_2)
  \, d\mu^{12}_{X_0} \, d\tau_0 }_{L^2_{\xi_0}},
\end{equation}
where the sum is over $\omega_1,\omega_2 \in \Omega(\gamma)$ with $\theta(\omega_1,\omega_2) \lesssim \gamma$. Then by \eqref{E:12:4},
\begin{equation}\label{N:44}
  \tau_0 + \xi_0\cdot\omega_1 = O(d'),
\end{equation}
where
$$
  d' = \max\left( \Lmax^{12}, \Nmax^{12} \gamma^2 \right)
  \sim \max\left( \Lmax^{12}, \frac{N_0L_0}{N_1} \right),
$$
recalling \eqref{N:41}. So now if we apply the Cauchy-Schwarz inequality with respect to $\tau_0$ in \eqref{N:43} using \eqref{N:44}, we get, defining $u_{012}^{\gamma,\omega_1,\omega_2}$ as in \eqref{G:14:6},
\begin{align*}
  \eqref{N:43}
  &\lesssim
  \sum_{\omega_1,\omega_2}
  \left(d'\right)^{1/2}
  \left( \frac{L_0}{N_0} \right)^{1/2}
  \norm{u_{012}^{\gamma,\omega_1,\omega_2}}
  \\
  &\lesssim
  \left(\max\left( \Lmax^{12}, \frac{N_0L_0}{N_1} \right)\right)^{1/2}
  \left( \frac{L_0}{N_0} \right)^{1/2}
  \\
  &\quad
  \times\min\left( N_0^2L_0\Lmin^{12}, N_0N_1L_1L_2 \right)^{1/2}
  \sum_{\omega_1,\omega_2}
  \norm{u_1^{\gamma,\omega_1}} \norm{u_2^{\gamma,\omega_2}}
  \\
  &\lesssim
  \left(N_0L_0^2L_1L_2\right)^{1/2}
  \norm{u_1}\norm{u_2},
\end{align*}
where we used Theorem \ref{M:Thm}, and we summed $\omega_1,\omega_2$ as in \eqref{B:208}.

This concludes the proof of \eqref{N:32} for $I_2$.

\subsection{Estimate for $I_3$}

Here the $\tau$-integrations decouple, so $I_3$ can be written
$$
  I_3 =
  \norm{ \iint \frac{g_1(\xi_1) g_2(\xi_2)}{\angles{\xi_0}^{1/2-s}\abs{\xi_0} 
  \angles{\xi_1}^s\angles{\xi_2}^s}
   \, \delta(\xi_0-\xi_1+\xi_2) \,d\xi_1 \, d\xi_2 }_{L^2_{\xi_0}},
$$
where
$$
  g_j(\xi_j) = \int \frac{G(X_j)}{\angles{\tau_j\pm_j\abs{\xi_j}}^{1/2+\varepsilon}} \, d\tau_j \qquad (j=1,2),
$$
hence
$$
  \norm{g_j} \le C_\varepsilon \norm{G_j}.
$$
Thus, it suffices to prove
$$
  I_3 \lesssim \norm{g_1}\norm{g_2},
$$
but this follows from the Sobolev product estimate (in physical space)
\begin{equation}\label{N:50}
  \norm{\abs{D}^{-1}(f_1f_2)}_{H^{s-1/2}} \lesssim \norm{f_1} \norm{f_2}_{H^s}
  \qquad (\forall f_1,f_2 \in \mathcal S(\R^3)),
\end{equation}
which we claim holds for any $s > 0$.

In fact, by Sobolev embedding and H\"older's inequality,
$$
  \norm{\abs{D}^{-1}(f_1f_2)}_{H^{s-1/2}} \lesssim
  \norm{f_1f_2}_{L^p}
  \lesssim \norm{f_1} \norm{f_2}_{L^{\frac{3}{3/2-s}}}
  \lesssim \norm{f_1} \norm{f_2}_{H^s},
$$
where
$$
  \frac{1}{p} - \frac12 = \frac{3/2-s}{3},
$$
and this proves \eqref{N:50}.

This concludes the proof of the estimate \eqref{N:20}, hence \eqref{N:10}.

\subsection{Proof of \eqref{A:110}} By the energy inequality for the wave equation, it is enough to prove
$$
  \int_{-T}^T \norm{\square A_\mu^{\text{inh.}}(t)}_{H^{s-3/2}} \, dt
  < \infty,
$$
which in view of \eqref{A:64} reduces to
$$
  \int_{-T}^T \norm{\abs{\psi(t)}^2}_{H^{s-3/2}} \, dt
  < \infty.
$$
But the latter follows immediately from \eqref{N:50} (which holds with $\abs{D}^{-1}$ replaced by $\angles{D}^{-1}$), since $\psi \in C([-T,T];H^s)$.

This concludes the proof of Theorem \ref{A:Thm2}.

\bibliographystyle{amsplain} 
\bibliography{mybibliography}

\end{document}